\documentclass[11pt]{article}
\usepackage{thmtools}
\usepackage{thm-restate}
\usepackage{bm}
\usepackage{mathrsfs}           %
\usepackage{tikz}
\usetikzlibrary{positioning,chains,fit,shapes,calc}
\usepackage{algorithm}
\usepackage{algorithmic}
\usepackage[linesnumbered,ruled,vlined,algo2e]{algorithm2e}

\usepackage{subcaption}
\usepackage{enumerate}

\usepackage{amsmath,amssymb, amsthm}    %
\usepackage{verbatim}           %
\usepackage{xspace}             %
\usepackage{graphicx,float}     %
\usepackage{ifthen,calc}        %
\usepackage{textcomp}           %
\usepackage{fancybox}           %
\usepackage{hhline}             %
\usepackage{float}              %

\usepackage[vflt]{floatflt}     %
\usepackage[small,compact]{titlesec}
\usepackage{setspace}

\usepackage{enumitem}
\setenumerate[1]{label={(\arabic*)}}

\usepackage{color}
\definecolor{ForestGreen}{rgb}{0.1333,0.5451,0.1333}
\usepackage{thumbpdf}
\usepackage[letterpaper,
colorlinks,linkcolor=ForestGreen,citecolor=ForestGreen,
backref,
bookmarks,bookmarksopen,bookmarksnumbered]
{hyperref}
\usepackage{cleveref}

\usepackage[margin = 1in]{geometry}

\newcommand{\showccc}[0]{0}
\newcommand{\ccc}[2][nothing]{%
	\ifthenelse{\showccc=0}{}{
		\ensuremath{^{\Lsh\Rsh}}\marginpar{\raggedright\tiny\textsf{%
				\ifthenelse{\equal{#1}{nothing}}{}{\textbf{#1}\\}#2}}}}
\newtheorem{theorem}{Theorem}

\newtheorem{proposition}{Proposition}
\newtheorem{corollary}{Corollary}
\newtheorem{definition}{Definition}

\newtheorem{lemma}{Lemma}
\newtheorem{fact}{Fact}

\newtheorem{assumption}{Assumption}

\usepackage{float}
\floatstyle{plain}\newfloat{myfig}{t}{figs}[section]
\floatname{myfig}{\textsc{Figure}}
\floatstyle{plain}\newfloat{myalg}{H}{algs}[section]
\floatname{myalg}{}
\newcommand{\defeq}{:=}
\newcommand{\norm}[1]{\left\lVert#1\right\rVert}
\newcommand{\inprod}[2]{\left\langle#1, #2\right\rangle}
\newcommand{\eps}{\epsilon}
\newcommand{\lam}{\lambda}
\newcommand{\argmax}{\textup{argmax}}
\newcommand{\argmin}{\textup{argmin}} 
\newcommand{\R}{\mathbb{R}}

\newcommand{\N}{\mathbb{N}}

\newcommand{\half}{\frac{1}{2}}

\newcommand{\1}{\mathbf{1}}
\newcommand{\E}{\mathbb{E}}

\newcommand{\xset}{\mathcal{X}}
\newcommand{\yset}{\mathcal{Y}}
\newcommand{\zset}{\mathcal{Z}}
\newcommand{\ma}{\mathbf{A}}

\newcommand{\id}{\mathbf{I}}
\newcommand{\tO}{\widetilde{O}}

\newcommand{\Par}[1]{\left(#1\right)}
\newcommand{\Brack}[1]{\left[#1\right]}
\newcommand{\Brace}[1]{\left\{#1\right\}}
\newcommand{\Abs}[1]{\left|#1\right|}
\newcommand{\Prox}{\textup{Prox}}
\newcommand{\normx}[1]{\norm{#1}_{\xset}}
\newcommand{\normy}[1]{\norm{#1}_{\yset}}

\newcommand{\normxd}[1]{\norm{#1}_{\xset, *}}
\newcommand{\normyd}[1]{\norm{#1}_{\yset, *}}
\newcommand{\omx}{\omega\x}
\newcommand{\omy}{\omega\y}
\newcommand{\Lam}{\Lambda}
\newcommand{\x}{^\mathsf{x}}
\newcommand{\y}{^\mathsf{y}}
\newcommand{\xx}{^\mathsf{xx}}
\newcommand{\yy}{^\mathsf{yy}}
\newcommand{\xy}{^\mathsf{xy}}

\newcommand{\gop}{\Phi}
\newcommand{\gopt}{\Phi}
\newcommand{\goptilde}{\widetilde{\Phi}}
\newcommand{\gcross}{\gop^{\textup{bilin}}}
\newcommand{\gcrosst}{\gopt^{\textup{bilin}}}
\newcommand{\gsept}{\gopt^{\textup{sep}}}
\newcommand{\gsep}{\gop^{\textup{sep}}}

\newcommand{\bz}{\bar{z}}
\newcommand{\bx}{\bar{x}}

\newcommand{\bw}{\bar{w}}

\newcommand*\circled[1]{\tikz[baseline=(char.base)]{
            \node[shape=circle,draw,inner sep=2pt] (char) {#1};}}
\newcommand{\Fmm}{F_{\textup{mm}}}
\newcommand{\Ffs}{F_{\textup{fs}}}
\newcommand{\Fmmpd}{F_{\textup{mm-pd}}}
\newcommand{\Ffspd}{F_{\textup{fs-pd}}}
\newcommand{\Ffsreg}{F_{\textup{fs-reg}}}

\newcommand{\yj}{^\mathsf{y_j}}

\newcommand{\Freg}{F_{\textup{mm-reg}}}
\newcommand{\Fregpd}{F_{\textup{mm-pd}}}
\newcommand{\Fmmfs}{F_{\textup{mmfs}}}
\newcommand{\Fmmfsreg}{F_{\textup{mmfs-reg}}}
\newcommand{\Fmmfspd}{F_{\textup{mmfs-pd}}}
\newcommand{\ssin}{\sum_{i \in [n]}}
\newcommand{\nsin}{\frac 1 n \sum_{i \in [n]}}
\newcommand{\bin}[1]{\Brace{#1}_{i \in [n]}}

\newcommand{\gtot}{g_{\textup{tot}}}
\newcommand{\Ltot}{\Lam^{\textup{tot}}}

\newcommand{\xsup}{^\mathsf{x}}

\newcommand{\xpssup}[1]{^{\mathsf{f}_\mathsf{#1}}}
\newcommand{\xdssup}[1]{^{\mathsf{f}^*_\mathsf{#1}}}
\newcommand{\ypssup}[1]{^{\mathsf{g}_\mathsf{#1}}}
\newcommand{\ydssup}[1]{^{\mathsf{g}^*_\mathsf{#1}}}
\newcommand{\xpsup}{^{\mathsf{f}}}
\newcommand{\xpisup}{^{\mathsf{f}_\mathsf{i}}}
\newcommand{\xdsup}{^{\mathsf{f}^*}}
\newcommand{\xdisup}{^{\mathsf{f}^*_\mathsf{i}}}
\newcommand{\xdjsup}{^{\mathsf{f}^*_\mathsf{j}}}
\newcommand{\xdi}{\xdisup}
\newcommand{\xpi}{\xpisup}
\newcommand{\xdj}{\xdjsup}
\newcommand{\ysup}{^\mathsf{y}}
\newcommand{\ypsup}{^{\mathsf{g}}}
\newcommand{\ypisup}{^{\mathsf{g}_\mathsf{i}}}
\newcommand{\ydsup}{^{\mathsf{g}^*}}
\newcommand{\ydisup}{^{\mathsf{g}^*_\mathsf{i}}}
\newcommand{\ydksup}{^{\mathsf{g}^*_\mathsf{k}}}
\newcommand{\ydi}{\ydisup}
\newcommand{\ypi}{\ypisup}
\newcommand{\ydk}{\ydksup}
\newcommand{\varfull}[1]{#1\xsup,#1\ysup,\left\{#1\xdisup\right\}_{i\in[n]},\left\{#1\ydisup\right\}_{i\in[n]}}
\newcommand{\varfulls}[1]{#1\xsup,#1\ysup,\{#1\xdisup\}_{i\in[n]},\{#1\ydisup\}_{i\in[n]}}

\newcommand{\goptot}{\gop^{\textup{mmfs-pd}}}

\newcommand{\xp}{\xpsup}
\newcommand{\xd}{\xdsup}
\newcommand{\yp}{\ypsup}
\newcommand{\yd}{\ydsup}

\newcommand{\xdual}{x^*}
\newcommand{\ydual}{y^*}

\newcommand{\xfun}{f}
\newcommand{\yfun}{g}
\newcommand{\xyfun}{h}
\newcommand{\zhat}{\hat{z}}

\newcommand{\dual}{^{\mathsf{f}^*\mathsf{g}^*}}
\newcommand{\mdpt}{_{\mathsf{aux}}}
\newcommand{\Hno}{^{fg}}
\newcommand{\Eps}{\epsilon}

\newcommand{\gap}{\textup{Gap}}

\newcommand{\mm}{\textsc{Minimax}}
\newcommand{\fs}{\textsc{Finitesum}}

\newcommand{\codeStyle}[1]{{\bfseries #1} }
\newcommand{\codeInput}{\codeStyle{Input:}}	
	
\newcommand{\codeReturn}{\codeStyle{Return:}}	
\newcommand{\codeParameter}{\codeStyle{Parameter(s):}}

\SetCommentSty{mycommfont}
\SetKwInput{KwInput}{Input} 
\SetKwInput{KwParameter}{Parameters}                %
\SetKwInput{KwOutput}{Output}              %
\SetKwInput{KwReturn}{Return}              %
\SetKwComment{Comment}{$\triangleright$\ }{}
\SetCommentSty{mycommfont}

  \usepackage{nth}
  \usepackage{intcalc}

\usepackage{url}

\begin{document}

	\begin{titlepage}
		\def\thepage{}
		\thispagestyle{empty}
		
		\title{Sharper Rates for Separable Minimax and Finite Sum Optimization \\via Primal-Dual Extragradient Methods} 
		
		\date{}
		\author{
			Yujia Jin\thanks{Stanford University, {\tt yujiajin@stanford.edu}}
			\and
			Aaron Sidford\thanks{Stanford University, {\tt sidford@stanford.edu}}
			\and
			Kevin Tian\thanks{Stanford University, {\tt kjtian@stanford.edu}}
		}
		
		\maketitle

\abstract{
	
We design accelerated algorithms with improved rates for several fundamental classes of optimization problems. Our algorithms all build upon  techniques related to the analysis of primal-dual extragradient methods via relative Lipschitzness proposed recently by \cite{CohenST21}.
\begin{enumerate}
	\item \textbf{Separable minimax optimization.} We study separable minimax optimization problems $\min_x \max_y f(x) - g(y) + h(x, y)$, where $f$ and $g$ have smoothness and strong convexity parameters $(L\x, \mu\x)$, $(L\y, \mu\y)$, and $h$ is convex-concave with a $(\Lam\xx, \Lam\xy, \Lam\yy)$-blockwise operator norm bounded Hessian. We provide an algorithm with gradient query complexity
	\[\tO\Par{\sqrt{\frac{L\x}{\mu\x}} + \sqrt{\frac{L\y}{\mu\y}} + \frac{\Lam\xx}{\mu\x} + \frac{\Lam\xy}{\sqrt{\mu\x\mu\y}} + \frac{\Lam\yy}{\mu\y}}.\]
	Notably, for convex-concave minimax problems with bilinear coupling (e.g.\ quadratics), where $\Lam\xx = \Lam\yy = 0$, our rate matches a lower bound of \cite{ZhangHZ19}.
	\item \textbf{Finite sum optimization.} We study finite sum optimization problems $\min_x \nsin f_i(x)$, where each $f_i$ is $L_i$-smooth and the overall problem is $\mu$-strongly convex. We provide an algorithm with gradient query complexity
	\[\tO\Par{n + \ssin  \sqrt{\frac{L_i}{n\mu}} }\,. \]%
	Notably, when the smoothness bounds $\bin{L_i}$ are non-uniform, our rate improves upon accelerated SVRG \cite{LinMH15, FrostigGKS15} and Katyusha \cite{Allen-Zhu17} by up to a $\sqrt{n}$ factor.
	\item \textbf{Minimax finite sums.} We generalize our algorithms for minimax and finite sum optimization to solve a natural family of minimax finite sum optimization problems at an accelerated rate, encapsulating both above results up to a logarithmic factor.
\end{enumerate}
}
 		
	\end{titlepage}

	\pagenumbering{gobble}
	\setcounter{tocdepth}{2}
	{
		\hypersetup{linkcolor=black}
		\tableofcontents
	}
	\newpage
	\pagenumbering{arabic}

\section{Introduction}
\label{sec:intro}

We study several fundamental families of optimization problems, namely (separable) minimax optimization, finite sum optimization, and minimax finite sum optimization (which generalizes both). These families have received widespread recent attention from the optimization community due to their prevalence in modeling tasks arising in modern data science. For example, minimax optimization has been used in both convex-concave settings and beyond to model robustness to (possibly adversarial) noise in many training tasks \cite{MadryMSTV18, RahimianM19, GoodfellowPMXWO20}. Moreover, finite sum optimization serves as a fundamental subroutine in many of the empirical risk minimization tasks of machine learning today \cite{BottouCN18}. Nonetheless, and perhaps surprisingly, there remain gaps in our understanding of the optimal rates for these problems. Toward closing these gaps, we provide new accelerated algorithms improving upon the state-of-the-art for each family of problems. 

Our results build upon recent advances in using primal-dual extragradient methods to recover accelerated rates for smooth, convex optimization in \cite{CohenST21}, which considered the problem\footnote{Throughout, $\xset, \yset$ are unconstrained, Euclidean spaces and $\norm{\cdot}$ denotes the Euclidean norm (see \Cref{sec:prelims}).}
\begin{equation}\label{eq:regintro}
	\min_{x \in \xset} f(x) + \frac \mu 2 \norm{x}^2~\text{for}~L\text{-smooth and convex}~f,
\end{equation}
and its equivalent primal-dual formulation as an appropriate ``Fenchel game''
\begin{equation}\label{eq:regintropd}
\min_{x \in \xset} \max_{\xdual \in \xset^*} \frac \mu 2 \norm{x}^2 + \inprod{\xdual}{x} - f^*(\xdual),~\text{where}~f^*\text{ is the convex conjugate of }f\,.
\end{equation}
In particular, \cite{CohenST21} showed that applying extragradient methods \cite{Nemirovski04, Nesterov07} and analyzing them through a condition the paper refers to as \emph{relative Lipschitzness} recovers an accelerated gradient query complexity for computing \eqref{eq:regintro}, which was known to be optimal \cite{Nesterov03}.

Both the Fenchel game \cite{AbernethyLLW18, WangA18} and the relative Lipschitzness property (independently proposed in \cite{StonyakinaTGDADPAP20}) have a longer history, discussed in Section~\ref{ssec:prev}. This work is particularly motivated by their synthesis in \cite{CohenST21}, which used these tools to provide a general recipe for designing accelerated methods. This recipe consists of the following ingredients.
\begin{enumerate}
\item Choose a primal-dual formulation of an optimization problem and a regularizer, $r$.
\item Bound iteration costs, i.e.\ the cost of implementing mirror steps with respect to $r$.
\item Bound the relative Lipschitzness of the gradient operator of the problem with respect to $r$.
\end{enumerate}
In \cite{CohenST21}, this recipe was applied with \eqref{eq:regintropd} as the primal-dual formulation and $r(x, x^*) \defeq \frac \mu 2 \norm{x}^2 + f^*(x^*)$. Further, it was shown that each iteration could be implemented (implicitly) with $O(1)$ gradient queries and that the gradient operator $\gop$ of the objective~\eqref{eq:regintropd} is $O(\sqrt{L/\mu})$-relatively Lipschitz with respect to $r$. Combining these ingredients gave the accelerated rate for \eqref{eq:regintropd}; we note that additional tools were further developed in \cite{CohenST21} for other settings including accelerated coordinate-smooth optimization (see \Cref{ssec:fsintro}). %

In this paper, we broaden the primal-dual extragradient approach of \cite{CohenST21} and add new recipes to the optimization cookbook. As a result, we obtain methods with improved rates for minimax optimization, finite sum optimization, and minimax finite sum optimization. We follow a similar recipe as \cite{CohenST21} but change the ingredients with different primal-dual formulations, regularizers, extragradient methods, and analyses. In the following Sections~\ref{ssec:mmintro},~\ref{ssec:fsintro}, and~\ref{ssec:mmfsintro}, we discuss each problem family, our results and approach, and situate them in the relevant literature. 

\subsection{Minimax optimization}\label{ssec:mmintro}

In Section~\ref{sec:minimax}, we study separable convex-concave minimax optimization problems of the form\footnote{Our results in Section~\ref{sec:minimax} apply generally to non-twice differentiable, gradient Lipschitz $h$, but we use these assumptions for simplicity in the introduction. All norms are Euclidean (see Section~\ref{sec:prelims} for relevant definitions).}
\begin{gather}
\min_{x \in \xset} \max_{y \in \yset} \Fmm(x, y) \defeq f(x) + h(x, y) - g(y),\label{eq:smmintro}
\end{gather}
where $f$ is $L\x$-smooth and $\mu\x$-strongly convex, $g$ is $L\y$-smooth and $\mu\y$-strongly convex, and $h$ is convex-concave and twice-differentiable with $\norm{\nabla^2_{xx} h} \le \Lam\xx$, $\norm{\nabla^2_{xy} h} \le \Lam\xy$, and $\norm{\nabla^2_{yy} h} \le \Lam\yy$. Our goal is to compute a pair of points $(x, y)$ with bounded duality gap with respect to $\Fmm$: $\gap_{\Fmm}(x, y) \le \eps$ (see Section~\ref{sec:prelims} for definitions).

The problem family \eqref{eq:smmintro} contains as a special case the following family of convex-concave minimax optimization problems with bilinear coupling (with $\Lam\xx = \Lam\yy = 0$ and $\Lam\xy = \norm{A}$):
\begin{equation}\label{eq:bmmintro}
\min_{x \in \xset} \max_{y \in \yset} f(x) + \Par{y^\top \ma x - \inprod{b}{y} + \inprod{c}{x}} - g(y).
\end{equation}
Problem \eqref{eq:bmmintro} has been widely studied in the optimization literature, dating at least to the classic work of \cite{ChambolleP11}, which used \eqref{eq:bmmintro} to relax convex optimization with affine constraints related to imaging inverse problems. Problem \eqref{eq:bmmintro} also encapsulates convex-concave quadratics and has been used to model problems in reinforcement learning \cite{DuCLXZ17} and decentralized optimization \cite{KovalevSR20}.

\paragraph{Our results.} We give the following result on solving \eqref{eq:smmintro}.

\begin{theorem}[informal, cf.\ Theorem~\ref{thm:main}]\label{thm:mmintro}
There is an algorithm that, given $(x_0, y_0) \in \xset \times \yset$ satisfying $\gap_{\Fmm}(x_0, y_0) \le \Eps_0$, returns $(x, y)$ with $\gap_{\Fmm}(x, y) \le \eps$ using $T$ gradient evaluations to $f$, $h$, and $g$ for
\[T = O\Par{\kappa_{\textup{mm}} \log \Par{\frac{\kappa_{\textup{mm}}\Eps_0}{\eps}}},\text{ for } \kappa_{\textup{mm}} \defeq \sqrt{\frac{L\x}{\mu\x}} + \sqrt{\frac{L\y}{\mu\y}} + \frac{\Lam\xx}{\mu\x} + \frac{\Lam\xy}{\sqrt{\mu\x\mu\y}} + \frac{\Lam\yy}{\mu\y}.\]
\end{theorem}
In the special case of \eqref{eq:bmmintro}, Theorem~\ref{thm:mmintro} matches a lower bound of \cite{ZhangHZ19}, which applies to the family of quadratic minimax problems obeying our smoothness and strong convexity bounds. More generally, Theorem~\ref{thm:mmintro} matches the lower bound whenever $\Lam\xx$ and $\Lam\yy$ are sufficiently small compared to the remaining parameters, improving prior state-of-the-art rates \cite{WangL20} in this regime. 

By applying reductions based on explicit regularization used in \cite{LinJJ20}, Theorem~\ref{thm:mmintro} also yields analogous accelerated rates depending polynomially on the desired accuracy when we either $f$, $g$, or both are not strongly convex. For conciseness, in this paper we focus on the strongly convex-concave regime discussed previously in this section.

\paragraph{Our approach.} Our algorithm for solving~\eqref{eq:smmintro} is based on the simple observation that minimax problems with the separable structure can be effectively ``decoupled'' by using convex conjugation on the components $f$ and $g$. In particular, following a similar recipe as the one in \cite{CohenST21} for smooth convex optimization, we rewrite (an appropriate regularized formulation of) the problem \eqref{eq:smmintro} using convex conjugates as follows:
\[\min_{x \in \xset, \ydual \in \yset^*} \max_{y \in \yset, \xdual \in \xset^*} \frac{\mu\x}{2} \norm{x}^2 - \frac{\mu\y}{2} \norm{y}^2 + \inprod{\xdual}{x} - \inprod{\ydual}{y} + h(x, y) - f^*(\xdual) + g^*(\ydual).\]
Further, we define the regularizer $r(x,y,\xdual,\ydual) \defeq \frac{\mu\x}{2} \norm{x}^2 + \frac{\mu\y}{2} \norm{y}^2 + f^*(\xdual)+ g^*(\ydual)$. Finally, we apply an extragradient method for strongly monotone operators to our problem, using this regularizer. As in \cite{CohenST21} we demonstrate efficient implementability, and analyze the relative Lipschitzness of the problem's gradient operator with respect to $r$, yielding Theorem~\ref{thm:mmintro}. In the final gradient oracle complexity, our method obtains the accelerated trade-off between primal and dual blocks for $\frac{\mu\x}{2} \norm{x}^2 + \inprod{\xdual}{x} - f^*(\xdual)$ and its $\yset$ analog, for the separable parts $f$ and $g$ respectively. It also obtains an unaccelerated rate for the $h$ component, by bounding the relative Lipschitzness corresponding to $h$ via our assumptions.

\paragraph{Prior work.} Many recent works obtaining improved rates for minimax optimization under smoothness and strong convexity restrictions concentrate on a more general family of problems of the form:
\begin{equation}\label{eq:mmintro}
\min_{x \in \xset} \max_{y \in \yset} F(x, y).
\end{equation}
Typically, these works assume (for simplicity, assuming $F$ is twice-differentiable), $\nabla^2_{xx} F$ is bounded between $\mu\x\id$ and $\Lam\xx\id$ everywhere, $\nabla^2_{yy} F$ is bounded between $\mu\y\id$ and $\Lam\yy\id$ everywhere, and $\nabla^2_{xy} F$ is operator norm bounded by $\Lam\xy$. It is straightforward to see that \eqref{eq:mmintro} contains \eqref{eq:smmintro} as a special case, by setting $f \gets \frac {\mu\x} 2 \norm{\cdot}^2$, $g \gets \frac {\mu\y} 2 \norm{\cdot}^2$, and $h \gets F - f + g$. 

For \eqref{eq:mmintro}, under gradient access to $F$, the works \cite{LinJJ20, WangL20, CohenST21} presented different approaches yielding a variety of query complexities. Letting $\Lam^{\max} \defeq \max\Par{\Lam\xx, \Lam\xy, \Lam\yy}$, these complexities scaled respectively as\footnote{$\tO$ hides logarithmic factors throughout, see Section~\ref{sec:prelims}.}
\[\tO\Par{\sqrt{\frac{\max\Par{\Lam\xx, \Lam\xy, \Lam\yy}^2}{\mu\x\mu\y}}},\; \tO\Par{\sqrt{\frac{\Lam\xx}{\mu\x}} + \sqrt{\frac{\Lam\yy}{\mu\y}} + \sqrt{\frac{\Lam\xy\Lam^{\max}}{\mu\x\mu\y}} },\; \tO\Par{\frac{\Lam\xx}{\mu\x} + \frac{\Lam\yy}{\mu\y} + \frac{\Lam\xy}{\sqrt{\mu\x\mu\y}}}.\]
The state-of-the-art rate (ignoring logarithmic factors) is due to \cite{WangL20}, which obtained the middle gradient query complexity above.
Theorem~\ref{thm:mmintro} matches the rate obtained by \cite{CohenST21} and improves upon \cite{LinJJ20} in some regimes.
Notably, Theorem~\ref{thm:mmintro} never improves upon \cite{WangL20} in the general regime, up to logarithmic factors. On the other hand, the method in Theorem~\ref{thm:mmintro} uses only a single loop, as opposed to the multi-loop methods in \cite{LinJJ20, WangL20} which lose logarithmic factors. 

Up to logarithmic factors, there is a gap between \cite{WangL20} and the lower bound of \cite{ZhangHZ19} only when $\Lam\xy \ll \Lam^{\max}$. We close this gap for minimax problems admitting the separable structure \eqref{eq:bmmintro}. In the special case of quadratic problems,  prior work, \cite{WangL20}, proposed a recursive approach which obtained a rate comparable to that of Theorem~\ref{thm:mmintro}, allbeit larger by subpolynomial factors.

\paragraph{Concurrent work.} A pair of independent and concurrent works \cite{KovalevGR21, ThekumparampilHO22} obtained variants of Theorem~\ref{thm:mmintro}. Their results were stated under the restricted setting of bilinear coupling \eqref{eq:bmmintro}, but they each provided alternative results under (different) weakenings of our strong convexity assumptions. The algorithm of \cite{ThekumparampilHO22} is closer to the one developed in this paper (also going through a primal-dual lifting), although the ultimate methods and analyses are somewhat different. Though our results were obtained independently, our presentation was informed by a reading of \cite{KovalevGR21, ThekumparampilHO22} for a comparison. 

\subsection{Finite sum optimization}\label{ssec:fsintro}

In Section~\ref{sec:finitesum}, we study finite sum optimization problems of the form
\begin{equation}\label{eq:fsintro}\min_{x \in \xset}\Ffs(x) \defeq \nsin f_i(x),\end{equation}
where $f_i$ is $L_i$-smooth for each $i \in [n]$, and $\nsin f_i$ is $\mu$-strongly convex. We focus on the strongly convex regime; through generic reductions \cite{ZhuH16}, our results yield accelerated rates depending polynomially on the desired accuracy, without the strong convexity assumption.

Methods for solving \eqref{eq:fsintro} have garnered substantial interest because of their widespread applicability to empirical risk minimization problems over a dataset of $n$ points, which encapsulate a variety of (generalized) regression problems in machine learning (see \cite{BottouCN18} and references therein). 

\paragraph{Our results.} We give the following result on solving \eqref{eq:fsintro}.
\begin{theorem}[informal, cf.\ Theorem~\ref{thm:mainfs}, Corollary~\ref{cor:mainfsunreg}]\label{thm:fsintro}
 There is an algorithm that, given $x_0 \in \xset$ satisfying $\Ffs(x_0) - \Ffs(x_\star) \le \eps_0$ where $x_\star$ minimizes $\Ffs$, returns $x\in \xset$ with $\E \Ffs(x) - \Ffs(x_\star) \le \eps$ using $T$ gradient evaluations (each to some $f_i$) for
	\[T = O\Par{\kappa_{\textup{fs}} \log \Par{\frac{\kappa_{\textup{fs}}\Eps_0}{\eps}}},\text{ for } \kappa_{\textup{fs}} \defeq n + \ssin \frac{\sqrt{L_i}}{\sqrt{n\mu}}.\]
\end{theorem}

\paragraph{Our approach.} Our algorithm for solving~\eqref{eq:fsintro} builds upon an accelerated coordinate descent developed in \cite{CohenST21}, which developed an analysis of a randomized extragradient method to do so. We consider an equivalent primal-dual formulation of (a regularized variant of) \eqref{eq:fsintro}, inspired by analogous developments in the ERM literature \cite{Shalev-Shwartz013, Shalev-Shwartz016}:
\[\min_{x \in \xset} \max_{\{x^*_i\}_{i \in [n]} \subset \xset^*} \frac{\mu} 2 \norm{x}^2 + \frac 1 n \ssin\Par{\inprod{x^*_i}{x} - f^*_i(x^*_i)}.\]
Our algorithm then solves this regularized primal-dual game to high precision.  

A key building block of our method is a randomized extragradient method which is compatible with strongly monotone problems. To this end, we extend the way the randomized extragradient method is applied in \cite{CohenST21}, which does not directly yield a high-precision guarantee. We proceed as follows: for roughly $\kappa_{\textup{fs}}$ iterations (defined in Theorem~\ref{thm:fsintro}) of our method, we run the non-strongly monotone randomized mirror prox method of \cite{CohenST21} to obtain a regret bound. We then subsample a random iterate, which we show halves an appropriate potential in expectation via our regret bound and strong monotonicity. We then recurse on this procedure to obtain a high-precision solver.

\paragraph{Prior work.} Developing accelerated algorithms for \eqref{eq:fsintro} under our regularity assumptions has been the subject of a substantial amount of research effort in the community, see e.g.\ \cite{LinMH15, FrostigGKS15, Shalev-Shwartz016, Allen-Zhu17} and references therein. Previously, the state-of-the-art gradient query complexities (up to logarithmic factors) for \eqref{eq:fsintro} were obtained by \cite{LinMH15, FrostigGKS15, Allen-Zhu17},\footnote{There have been a variety of additional works which have also attained accelerated rates for either the problem \eqref{eq:fsintro} or its ERM specialization, see e.g.\ \cite{Defazio16, ZhangX17, LanLZ19, ZhouDSCLL19}. However, to the best of our knowledge these do not improve upon the state-of-the-art rate of \cite{Allen-Zhu17} in our setting.} and scaled as
\begin{equation}\label{eq:prevsgd}\tO\Par{n + \sqrt{\frac{\ssin L_i}{\mu}}}.\end{equation}
Rates such as \eqref{eq:prevsgd}, which scale as functions of $\ssin \frac{L_i} \mu$, arise in known \emph{variance reduction}-based approaches \cite{Johnson013, DefazioBL14, SchmidtRB17, Allen-Zhu17} due to their applications of a ``dual strong convexity'' lemma (e.g.\ Theorem 1, \cite{Johnson013} or Lemma 2.4, \cite{Allen-Zhu17}) of the form
\begin{align*}
	\norm{\nabla f_i(x) - \nabla f_i(\bx)}^2 \le 2L_i \Par{f_i(\bx) - f_i(x) - \inprod{\nabla f_i(x)}{\bx - x}}.
\end{align*}
The analyses of e.g.\ \cite{Johnson013, Allen-Zhu17} sample $i \in [n]$ proportional to $L_i$, allowing them to bound the variance of a resulting gradient estimator by a quantity related to the divergence in $\Ffs$.

The rate in \eqref{eq:prevsgd} is known to be optimal in the uniform smoothness regime \cite{WoodworthS16}, but in a more general setting its optimality is unclear. Theorem~\ref{thm:fsintro} shows that the rate can be improved for sufficiently non-uniform $L_i$. In particular, Cauchy-Schwarz shows that the quantity $\kappa_{\textup{fs}}$ is never worse than \eqref{eq:prevsgd}, and improves upon it by a factor asymptotically between $1$ and $\sqrt{n}$ when the $\bin{L_i}$ are non-uniform. Moreover, even in the uniform smoothness case, Theorem~\ref{thm:fsintro} matches the tightest rate in \cite{Allen-Zhu17} up to an additive $\log \kappa_{\text{fs}}$ term, as opposed to an additional multiplicative logarithmic overhead incurred by the reduction-based approaches of \cite{LinMH15, FrostigGKS15}.  

Our rate's improvement over \eqref{eq:prevsgd} is comparable to a similar improvement that was achieved previously in the literature on coordinate descent methods. In particular, \cite{LeeS13} first obtained a (generalized) partial derivative query complexity comparable to \eqref{eq:prevsgd} under coordinate smoothness bounds, which was later improved to a query complexity comparable to Theorem~\ref{thm:fsintro} by \cite{ZhuQRY16, NesterovS17}. Due to connections between coordinate-smooth optimization and empirical risk minimization (ERM) previously noted in the literature \cite{Shalev-Shwartz013, Shalev-Shwartz016}, it is natural to conjecture that the rate in Theorem~\ref{thm:fsintro} is achieveable for finite sums \eqref{eq:fsintro} as well. However, prior to our work (to our knowledge) this rate was not known, except in special cases e.g.\ linear regression \cite{AgarwalKKLNS20}. 

Our method is based on using a primal-dual formulation of \eqref{eq:fsintro} to design our gradient estimators. It attains the rate of Theorem~\ref{thm:mainfs} by sampling summands proportional to $\sqrt{L_i}$, trading off primal and dual variances through a careful coupling. It can be viewed as a modified dual formulation to the coordinate descent algorithm in \cite{CohenST21}, which used primal-dual couplings inspired by the fine-grained accelerated algorithms of \cite{ZhuQRY16, NesterovS17}. We believe our result sheds further light on the duality between coordinate-smooth and finite sum optimization, and gives an interesting new approach for algorithmically leveraging primal-dual formulations of finite sum problems.

\subsection{Minimax finite sum optimization}\label{ssec:mmfsintro}

In Section~\ref{sec:mmfs}, we study a family of minimax finite sum optimization problems of the form
\begin{gather}\min_{x \in \xset} \max_{y \in \yset} \Fmmfs(x, y) \defeq \nsin \Par{f_i(x) + h_i(x, y) - g_i(y)}. \label{eq:mmfsintro}
\end{gather}
We assume $f_i$ is $L\x_i$-smooth, $g_i$ is $L\y$-smooth, and $h_i$ is convex-concave and twice-differentiable with blockwise operator norm bounds $\Lam\xx_i$, $\Lam\xy_i$, and $\Lam\yy_i$ for each $i \in [n]$. We also assume the whole problem is $\mu\x$-strongly convex and $\mu\y$-strongly concave.

We propose the family \eqref{eq:mmfsintro} because it encapsulates \eqref{eq:mmintro} and \eqref{eq:fsintro}, and is amenable to techniques from solving both. Moreover, \eqref{eq:mmfsintro} is a natural description of instances of \eqref{eq:mmintro} which arise from primal-dual formulations of empirical risk minimization problems, e.g.\ \cite{ZhangX17, WangX17}. It also generalizes natural minimax finite sum problems previously considered in e.g.\ \cite{CarmonJST19}.
	
\paragraph{Our results.} We give the following result on solving \eqref{eq:mmfsintro}.

\begin{theorem}[informal, cf.\ Theorem~\ref{thm:mmfs}]\label{thm:mmfsintro} There is an algorithm that, given $(x_0, y_0) \in \xset \times \yset$ satisfying $\gap_{\Fmmfs}(x_0, y_0) \le \eps_0$, returns $(x, y)$ with $\E \gap_{\Fmmfs}(x, y) \le \eps$ using $T$ gradient evaluations, each to some $f_i$, $g_i$, or $h_i$, where
\begin{gather*}T = O\Par{\kappa_{\textup{mmfs}} \log\Par{\kappa_{\textup{mmfs}}}\log\Par{\frac{\kappa_{\textup{mmfs}} \Eps_0}{\eps}}},\\
\text{for } \kappa_{\textup{mmfs}} \defeq n + \frac{1}{\sqrt{n}} \sum_{i \in [n]} \Par{\sqrt{\frac{L\x_i}{\mu\x}} + \sqrt{\frac{L\y_i}{\mu\y}} + \frac{\Lam\xx_i}{\mu\x} + \frac{\Lam\xy_i}{\sqrt{\mu\x\mu\y}} + \frac{\Lam\yy_i}{\mu\y}}.\end{gather*}
\end{theorem}

The rate in Theorem~\ref{thm:mmfsintro} captures (up to a logarithmic factor) both of the rates in Theorems~\ref{thm:mmintro} and~\ref{thm:fsintro}, when \eqref{eq:mmfsintro} is appropriately specialized. It can be more generally motivated as follows. When $n$ is not the dominant term in Theorem~\ref{thm:fsintro}'s bound, the remaining term is $\sqrt{n}$ times the average rate attained by Nesterov's accelerated gradient method \cite{Nesterov83} on each summand in \eqref{eq:fsintro}. This improves upon the factor of $n$ overhead which one might naively expect from computing full gradients. In similar fashion, Theorem~\ref{thm:mmfsintro} attains a rate (up to an additive $n$, and logarithmic factors) which is $\sqrt{n}$ times the average rate attained by Theorem~\ref{thm:mmintro} on each summand in \eqref{eq:mmfsintro}.

\paragraph{Our approach.}

Our algorithm for solving~\eqref{eq:mmfsintro} is a natural synthesis of the algorithms suggested in \Cref{ssec:mmintro,ssec:fsintro}. However, to obtain our results we apply additional techniques to bypass complications which arise from the interplay between the minimax method and the finite sum method, inspired by \cite{CarmonJST19}. In particular, to obtain our tightest rate we would like to subsample the components in our gradient operator corresponding to $\bin{f_i}, \bin{g_i}, \bin{h_i}$ all at different frequencies when applying the randomized extragradient method. These different sampling distributions introduce dependencies between iterates, and make our randomized estimators no longer ``unbiased'' for the true gradient operator to directly incur the randomized extragradient analysis.

To circumvent this difficulty, we obtain our result via a partial decoupling, treating components corresponding to $\bin{f_i}$, $\bin{g_i}$ and those corresponding to $\bin{h_i}$ separately. For the first two aforementioned components, which are separable and hence do not interact, we pattern an expected relative Lipschitzness analysis for each block, similar to the finite sum optimization. For the remaining component $\{h_i\}_{i\in[n]}$, we develop a variance-reduced stochastic method which yields a relative variance bound. We put these pieces together in Proposition~\ref{prop:newrmp}, a new randomized extragradient method analysis, to give a method with a convergence rate of roughly
\begin{align*}n + \frac{1}{\sqrt{n}} \ssin\Par{\sqrt{\frac{L_i\x}{\mu\x}} + \sqrt{\frac{L_i\y}{\mu\y}}} + (\kappa_{\textup{mmfs}}^h)^2,
\text{ where } \kappa_{\textup{mmfs}}^h \defeq \nsin \Par{\frac{\Lam\xx_i}{\mu\x} + \frac{\Lam\xy_i}{\sqrt{\mu\x\mu\y}} + \frac{\Lam\yy_i}{\mu\y}}. \end{align*}
The dependence on all pieces above is the same as in Theorem~\ref{thm:mmfsintro}, except for the term corresponding to the $\bin{h_i}$. To improve this dependence, we wrap our solver in an ``outer loop'' proximal point method which solves a sequence of $\gamma$-regularized variants of~\eqref{eq:mmfsintro}. We obtain our final claim by trading off the terms $n$ and $(\kappa_{\textup{mmfs}}^h)^2$ through our choice of $\gamma$, which yields the accelerated convergence rate of Theorem~\ref{thm:mmfsintro}.

\paragraph{Prior work.} To our knowledge, there have been relatively few results for solving \eqref{eq:mmfsintro} under our fine-grained assumptions on problem regularity, although various stochastic minimax algorithms have been developed in natural settings \cite{JuditskyNT11, PalaniappanB16, HsiehIMM19, CarmonJST19, ChavdarovaGFL19, CarmonJST20, AlacaogluM21}. For the general problem of solving $\min_{x\in\xset}\max_{y\in\yset}\frac{1}{n}\sum_{i\in[n]}F_i(x,y)$ where $F_i$ is $L_i$-smooth and convex-concave, and the whole problem is $\mu\x$-strongly convex and $\mu\y$-strongly concave, perhaps the most direct comparisons are Section 5.4 of \cite{CarmonJST19} and Theorem 15 of \cite{TomininTBKGD21}. In particular, \cite{CarmonJST19} provided a high-precision solver using roughly
\[\tO\Par{n + \frac 1 {\sqrt{n}} \ssin \frac{L_i}{\mu}}\]
gradient queries, when $\mu\x = \mu\y = \mu$. This is recovered by Theorem~\ref{thm:mmfs} (possibly up to logarithmic factors) in the special setting of $f_i = g_i \gets 0$, $\mu\x = \mu\y \gets \mu$, and $\Lam\xx_i = \Lam\xy_i = \Lam\yy_i \gets L_i$. More generally, \cite{CarmonJST19} gave a result depending polynomially on the desired accuracy without the strongly convex-concave assumption, which follows from a variant of Theorem~\ref{thm:mmfs} after applying the explicit regularization in \cite{LinJJ20} that reduces to the strongly convex-concave case. 

Moreover, Theorem 15 of \cite{TomininTBKGD21} provided a high-precision solver using roughly
\[\tO\Par{n + \frac{1}{\sqrt{n}} \ssin \frac{L_i}{\sqrt{\mu\x\mu\y}}}\]
gradient queries. Our work recovers (and sharpens dependences in) this result for minimax finite sum problems where each summand has the bilinear coupling \eqref{eq:bmmintro}. In the more general setting where each summand only has a uniform smoothness bound, the \cite{TomininTBKGD21} result can be thought of as the accelerated finite sum analog of the main claim in \cite{LinJJ20}, which is incomparable to our Theorem~\ref{thm:mmintro}. In a similar way, the rate of \cite{TomininTBKGD21} is incomparable to Theorem~\ref{thm:mmfsintro}, and each improves upon the other in different parameter regimes. We believe designing a single algorithm which obtains the best of both worlds for \eqref{eq:mmfsintro} is an interesting future direction.

\subsection{Additional related work}\label{ssec:prev}

We give a brief discussion of several lines of work which our results build upon, and their connection with the techniques used in this paper.

\paragraph{Acceleration via primal-dual extragradient methods.} Our algorithms are based on \emph{extragradient methods}, a framework originally proposed by \cite{Korpelevich76} which was later shown to obtain optimal rates for solving Lipschitz variational inequalities in \cite{Nemirovski04, Nesterov07}. There have been various implementations of extragradient methods including mirror prox \cite{Nemirovski04} and dual extrapolation \cite{Nesterov07}; we focus on adapting the former in this work. Variations of extragradient methods have been studied in the context of primal-dual formulations of smooth convex optimization \cite{AbernethyLLW18, WangA18, CohenST21}, and are known to obtain optimal (accelerated) rates in this setting. In particular, the relative Lipschitzness analysis of acceleration in \cite{CohenST21} is motivated by developments in the bilinear setting, namely the area convexity framework of \cite{Sherman17}. We build upon these works by using primal-dual formulations to design accelerated algorithms in various settings beyond smooth convex optimization, namely \eqref{eq:mmintro}, \eqref{eq:fsintro}, and \eqref{eq:mmfsintro}.

\paragraph{Acceleration under relative regularity assumptions.} Our analysis builds upon a framework for analyzing extragradient methods known as \emph{relative Lipschitzness}, proposed independently by \cite{StonyakinaTGDADPAP20, CohenST21}. We demonstrate that this framework (and randomized variants thereof) obtains improved rates for primal-dual formulations beyond those studied in prior works. 

Curiously, our applications of the relative Lipschitzness framework reveal that the regularity conditions our algorithms require are weaker than standard assumptions of smoothness in a norm. In particular, several technical requirements of specific components of our algorithms are satisfied by setups with regularity assumptions generalizing and strengthening the \emph{relative smoothness} assumption of \cite{BauschkeBT17, LuFN18}. This raises interesting potential implications in terms of the necessary regularity assumptions for non-Euclidean acceleration, because relative smoothness is known to be alone insufficient for obtaining accelerated rates in general \cite{DragomirTAB19}. Notably, \cite{HanzelyRX18} also developed an acceleration framework under a strengthened relative smoothness assumption, which requires strengthened bounds on divergences between three points. We further elaborate on these points in Section~\ref{ssec:mmconv}, when deriving relative Lipschitzness bounds through weaker assumptions in Lemma~\ref{lem:smoothness_implications}. We focus on the Euclidean setup in this paper, but we believe an analogous study of non-Euclidean setups is interesting and merits future exploration.

\section{Preliminaries}
\label{sec:prelims}

\paragraph{General notation.} We use $\tO$ to hide logarithmic factors in problem regularity parameters, initial radius bounds, and target accuracies when clear from context. We denote $[n] \defeq \{i \in \N \mid i \le n\}$. Throughout the paper, $\xset$ (and $\yset$, when relevant) represent Euclidean spaces, and $\norm{\cdot}$ will mean the Euclidean norm in appropriate dimension when applied to a vector. For a variable on a product space, e.g.\ $z \in \xset \times \yset$, we refer to its blocks as $(z\x, z\y)$ when clear from context. For a bilinear operator $\ma: \xset \to \yset^*$, $\norm{\cdot}$ will mean the (Euclidean) operator norm, i.e.\
\[\norm{\ma} \defeq \sup_{\norm{x} = 1} \norm{\ma x} = \sup_{\norm{x} = 1} \sup_{\norm{y}  = 1} y^\top \ma x.\]

\paragraph{Complexity model.} Throughout the paper, we evaluate the complexity of methods by their gradient oracle complexity, and do not discuss the cost of vector operations (which typically are subsumed by the cost of the oracle). In Section~\ref{sec:minimax}, the gradient oracle returns $\nabla f$, $\nabla g$, or $\nabla h$ at any point; in Section~\ref{sec:finitesum} (respectively, Section~\ref{sec:mmfs}), the oracle returns $\nabla f_i$ at a point for some $i \in [n]$ (respectively, $\nabla f_i$, $\nabla g_i$, or $\nabla h_i$ at a point for some $i \in [n]$). 

\paragraph{Divergences.} The Bregman divergence induced by differentiable, convex $r$ is $V^r_x(x') \defeq r(x') - r(x) - \inprod{\nabla r(x)}{x' - x}$, for any $x,x'\in\xset$. For all $x$, $V^r_x$ is nonnegative and convex. Whenever we use no superscript $r$, we assume $r = \half \norm{\cdot}^2$ so that $V_x(x') = \half \norm{x - x'}^2$. Bregman divergences satisfy the equality
\begin{equation}\label{eq:threept}
\inprod{\nabla r(w) - \nabla r(z)}{w - u} = V^r_z(w) + V^r_w(u) - V^r_z(u). 
\end{equation}
We define the proximal operation in $r$ by
\[\Prox^r_x(\gop) \defeq \argmin_{x'\in\xset}\Brace{\inprod{\gop}{x'} + V^r_x(x')}.\]

\paragraph{Functions and operators.} We say $h: \xset \times \yset \to \R$ is convex-concave if its restrictions $h(\cdot, y)$ and $h(x, \cdot)$ are respectively convex and concave, for any $x \in \xset$ and $y \in \yset$. The duality gap of a pair $(x, y)$ is $\gap_h(x, y) \defeq \max_{y' \in \yset} h(x, y') - \min_{x' \in \xset} h(x', y)$; a saddle point is a pair $(x_\star, y_\star) \in \xset \times \yset$ with zero duality gap. 

We call operator $\gop: \zset \to \zset^*$ monotone if $\inprod{\gop(z) - \gop(z')}{z - z'} \ge 0$ for all $z, z' \in \zset$. We say $z_\star$ solves the variational inequality (VI) in $\gop$ if $\inprod{\gop(z_\star)}{z_\star - z} \le 0$ for all $z \in \zset$. We equip differentiable convex-concave $h$ with the ``gradient operator'' $\gop(x, y) \defeq (\nabla_x h(x, y), -\nabla_y h(x, y))$. The gradient of convex $f$ and the gradient operator of convex-concave $h$ are both monotone (see Appendix~\ref{apdx:sm}). Their VIs are respectively solved by any minimizers of $f$ and saddle points of $h$.

\paragraph{Regularity.} We say function $f: \xset \to \R$ is $L$-smooth if $\norm{\nabla f(x) - \nabla f(x')} \le L\norm{x - x'}$ for all $x, x' \in \xset$; if $f$ is twice-differentiable, this is equivalent to $(x'-x)^\top \nabla^2 f(x) (x'-x) \le L\norm{x'-x}^2$ for all $x, x' \in \xset$. We say differentiable function $f: \xset \to \R$ is $\mu$-strongly convex if $V^f_x(x') \ge \frac \mu 2 \norm{x - x'}^2$ for all $x, x' \in \xset$; if $f$ is twice-differentiable, this is equivalent to $(x'-x)^\top \nabla^2 f(x)(x'-x) \ge \mu \norm{x'-x}^2$ for all $x, x' \in \xset$. Finally, we say operator $\gop: \zset \to \zset^*$ is $m$-strongly monotone with respect to convex $r: \zset \to \R$ if for all $z, z' \in \zset$, 
\[\inprod{\gop(z) - \gop(z')}{z - z'} \ge m\inprod{\nabla r(z) - \nabla r(z')}{z - z'} = m\Par{V^r_z(z') + V^r_{z'}(z)}.\]

\paragraph{Convex conjugates.} The (Fenchel dual) convex conjugate of a convex  $f: \xset \to \R$ is denoted
\[f^*(x^*) \defeq \max_{x \in \xset} \inprod{x}{x^*} - f(x).\] 
We allow $f^*$ to take the value $\infty$. We recall the following facts about convex conjugates.
\begin{fact}\label{fact:dualsc}
Let $f: \xset \to \R$ be differentiable.
	\begin{enumerate}
		\item For all $x \in \xset$, $\nabla f(x) \in \argmax_{x^* \in \xset^*} \inprod{x^*}{x} - f^*(x^*)$.
		\item $(f^*)^* = f$.
		\item If $f^*$ is differentiable, for all $x \in \xset$, $\nabla f^*(\nabla f(x)) = x$.
		\item If $f$ is $L$-smooth, then for all $x, x' \in \xset$,
		\[f(x') - f(x) - \inprod{\nabla f(x)}{x' - x} \ge \frac 1 {2L}\norm{\nabla f(x') - \nabla f(x)}^2.\] 
		If $f$ is $\mu$-strongly convex, $f^*$ is $\frac 1 \mu$-smooth.
	\end{enumerate}
\end{fact}
\begin{proof}
The first three items all follow from Chapter 11 of \cite{Rockafellar70}. The first part of the fourth item is shown in Appendix A of \cite{CohenST21}, and the second part is shown in  \cite{KakadeST09}.
\end{proof}
For a function $f: \xset \to \R$, we define the set $\xset^*_f \subset \xset^*$ to be the set of points realizable as a gradient, namely $\xset^*_f \defeq \{\nabla f(x) \mid x \in \xset\}$. This will be come relevant in applications of Item 4 in Fact~\ref{fact:dualsc} throughout the paper, when $\nabla f$ is not surjective (onto $\xset^*$).

%

\section{Minimax optimization}
\label{sec:minimax}

In this section, we provide efficient algorithms for computing an approximate saddle point of the following separable minimax optimization problem:
\begin{equation}\label{eq:minimax}
	\min_{x \in \xset} \max_{y \in \yset} 
	\Fmm(x, y) 
	\text{ for }
	\Fmm \defeq f(x) + h(x, y) - g(y) \,.
\end{equation}
Here and throughout this section $f: \xset \rightarrow \R$ and $g: \yset \rightarrow \R$ are differentiable, convex functions and $h: \xset \times \yset \rightarrow \R$ is a differentiable, convex-concave function. For the remainder, we focus on algorithms for solving the following regularized formulation of \eqref{eq:minimax}:
\begin{equation}\label{eq:minimax_reg}
	\min_{x \in \xset} \max_{y \in \yset} 
	\Freg(x, y) 
	\text{ for }
	\Freg(x, y) \defeq f(x) + h(x, y) - g(y) + \frac{\mu\x}{2}\norm{x}^2 - \frac{\mu\y} 2 \norm{y}^2.
\end{equation}
To instead solve an instance of \eqref{eq:minimax} where $f$ is $\mu\x$-strongly convex and $g$ is $\mu\y$-strongly convex, we may instead equivalently solve \eqref{eq:minimax_reg} by reparameterizing $f \gets f - \frac{\mu\x} 2 \norm{\cdot}^2$, $g \gets g - \frac{\mu\y} 2 \norm{\cdot}^2$. As it is notationally convenient for our analysis, we focus on solving the problem \eqref{eq:minimax_reg} and then give the results for \eqref{eq:minimax} at the end of this section in \Cref{cor:minimax}.

In designing methods for solving \eqref{eq:minimax_reg} we make the following additional regularity assumptions.

\begin{assumption}[Minimax regularity]\label{assume:minimax}
We assume the following about \eqref{eq:minimax_reg}.
	\begin{enumerate}
		\item $f$ is $L\x$-smooth and $g$ is $L\y$-smooth.
		\item $h$ has the following blockwise-smoothness properties: for all $u, v \in \xset \times \yset$,
		\begin{equation}\label{eq:Hlipbound}
			\begin{aligned}
				\norm{\nabla_x h(u) - \nabla_x h(v)} &\le \Lam\xx \norm{u\x - v\x} + \Lam\xy \norm{u\y - v\y},\\
				\norm{\nabla_y h(u) - \nabla_y h(v)} &\le \Lam\xy \norm{u\x - v\x} + \Lam\yy \norm{u\y - v\y}.
		\end{aligned}\end{equation}
	\end{enumerate}
\end{assumption}

Note that when $h$ is twice-differentiable, \eqref{eq:Hlipbound} equates to everywhere operator norm bounds on blocks of $\nabla^2 h$. Namely, for all $w \in \xset \times \yset$, 
\begin{align*}
\norm{\nabla^2_{xx} h(w)}_{\textup{op}} \le \Lam\xx,\;  \norm{\nabla^2_{xy} h(w)}_{\textup{op}} \le \Lam\xy \text{, and } \norm{\nabla^2_{yy} h(w)}_{\textup{op}} \le \Lam\yy.
\end{align*}
In the particular case when $h(x, y) = y^\top \ma x - b^\top y + c^\top x$ is bilinear, clearly $\Lam\xx = \Lam\yy = 0$ (as remarked in the introduction). In this case, we may then set $\Lam\xy \defeq \norm{\ma}_{\textup{op}}$.

The remainder of this section is organized as follows.

\begin{enumerate}
	\item In Section~\ref{ssec:mmsetup}, we state a primal-dual formulation of \eqref{eq:minimax_reg} which we will apply our methods to, and prove that its solution yields a solution to \eqref{eq:minimax_reg}.
	\item In Section~\ref{ssec:mmalgo}, we give our algorithm and prove it is efficiently implementable.
	\item In Section~\ref{ssec:mmconv}, we prove the convergence rate of our algorithm.
	\item In Section~\ref{ssec:mmres}, we state and prove our main result, Theorem~\ref{thm:main}.
\end{enumerate}

\subsection{Setup}\label{ssec:mmsetup}

To solve \eqref{eq:minimax_reg}, we will instead find a saddle point to the expanded primal-dual function
\begin{equation}\label{eq:primaldual}\Fregpd(z) \defeq
\inprod{z\xd}{z\x} - \inprod{z\yd}{z\y} + \frac{\mu\x} 2 \norm{z\x}^2 - \frac{\mu\y} 2 \norm{z\y}^2 + h(z\x, z\y) -  f^*(z\xd) + g^*(z\yd).\end{equation}
We denote the domain of $\Fmmpd$ by $\zset \defeq \xset \times \yset \times \xset^* \times \yset^*$. For $z \in \zset$, we refer to its blocks by $(z\x, z\y, z\xd, z\yd)$. The primal-dual function $\Fregpd$ is related to $\Freg$ in the following way.

\begin{lemma}\label{lem:sameopt}
	Let $z_\star$ be the saddle point to \eqref{eq:primaldual}. Then, $(z_\star\xsup, z_\star\ysup)$ is a saddle point to \eqref{eq:minimax_reg}.
\end{lemma}
\begin{proof}
	By performing the maximization over $z\xd$ and minimization over $z\yd$, we see that the problem of computing a saddle point to the objective in \eqref{eq:primaldual} is equivalent to
	\begin{align*}
	\min_{z\xsup \in \xset} \max_{z\ysup \in \yset} \frac{\mu\x} 2 \norm{z\x}^2 - \frac{\mu\y} 2 \norm{z\y}^2 + h(z\xsup, z\ysup) +  \Par{\max_{z\xd \in \xset^*} \inprod{z\xd}{z\xsup} - f^*(z\xd)} - \Par{\max_{z\yd \in \yset^*} \inprod{z\yd}{z\ysup} - g^*(z\yd)}.
	\end{align*}
	By Item 2 in Fact~\ref{fact:dualsc}, this is the same as \eqref{eq:minimax_reg}.
\end{proof}

We next define $\gop$, the gradient operator of $\Fmmpd$. Before doing so, it will be convenient to define $r: \zset \to \R$, which combines the (unsigned) separable components of $\Fmmpd$:
\begin{equation}\label{eq:rdef}
r(z) \defeq \frac{\mu\x} 2 \norm{z\x}^2 + \frac{\mu\y} 2 \norm{z\y}^2 + f^*(z\xd) + g^*(z\yd).
\end{equation}
The function $r$ will also serve as a regularizer in our algorithm. With this definition, we decompose $\gop$ into three parts, roughly corresponding to the contribution from $r$, the contributions from the bilinear portion of the primal-dual representations of $f$ and $g$, and the contribution from $h$. In particular, we define
\begin{equation}\label{eq:gdecomp}
\begin{aligned}
\gop^r(z) &\defeq \nabla r(z) = \Par{\mu\x z\x, \mu\y z\y, \nabla f^*(z\xd), \nabla g^*(z\yd)} \\
\gcross(z) &\defeq (z\xd, z\yd, -z\x, -z\y), \\
\gop^h(z) &\defeq \Par{\nabla_x h(z\x, z\y), -\nabla_y h(z\x, z\y), 0, 0}. \\
\end{aligned}
\end{equation}
It is straightforward to check that $\gop$, the gradient operator of $\Fmmpd$, satisfies
\begin{equation}\label{eq:gdef} 
\gop(z) \defeq \gop^r(z) + \gcross(z) + \gop^h(z).
\end{equation}
Finally, we note that by construction $\gop$ is $1$-strongly monotone with respect to $r$.

\begin{lemma}[Strong monotonicity]\label{lem:sm}
	The operator $\gop$ (as defined in \eqref{eq:gdef}) is 1-strongly monotone with respect to the function $r: \zset \to \R$ as in \eqref{eq:rdef}.
\end{lemma}
\begin{proof}
	Consider the decomposition of $\gop = \gop^r + \gcross + \gop^h$ defined in \eqref{eq:gdecomp} and \eqref{eq:gdef}. By definition and~\Cref{item:sm-convex,item:sm-convex-concave,item:sm-self} from~\Cref{fact:sm}, we know the operators $\gop^h$ and $\gcross$ are monotone, and $\gop^r = \nabla r$ is $1$-strongly monotone with respect to $r$. Combining the three operators and using additivity of monotonicity in~\Cref{item:sm-additive} of~\Cref{fact:sm} yields the claim.
\end{proof}

\subsection{Algorithm}\label{ssec:mmalgo}

Our algorithm will be an instantiation of \emph{strongly monotone mirror prox} \cite{CohenST21} stated as Algorithm~\ref{alg:smmp} below, an alternative to the mirror prox algorithm originally proposed by \cite{Nemirovski04}.

\begin{algorithm2e}
	\caption{$\textsc{SM-Mirror-Prox}(\lam, T, z_0)$: Strongly monotone mirror prox \cite{CohenST21}}
	\label{alg:smmp}
	\DontPrintSemicolon
		\codeInput Convex $r: \zset \to \R$, $m$-strongly monotone $\gop: \zset \to \zset^*$ (with respect to $r$),  $z_0 \in \zset$\;
		\codeParameter $\lam > 0$, $T \in \N$\;
		\For{$0 \le t < T$}{
		$z_{t+1/2} \gets \Prox^r_{z_t}(\tfrac{1}{\lam}\gop(z_t))$\;
		$z_{t + 1} \gets \argmin_{z \in \zset}\{\frac 1 \lam\inprod{\gop(z_{t+1/2})}{z} + \frac m \lam V^r_{z_{t+1/2}}(z) + V^r_{z_t}(z)\}$\;
		}
\end{algorithm2e}

In order to analyze Algorithm~\ref{alg:smmp}, we need to introduce a definition from \cite{CohenST21}.

\begin{definition}[Relative Lipschitzness]\label{def:rl}
	We say operator $\gop: \zset \to \zset^*$ is $\lam$-relatively Lipschitz with respect to convex $r: \zset \to \R$ over $\zset_{\textup{alg}} \subseteq \zset$ if for every three $z, w, u \in \zset_{\textup{alg}}$,
	\[\inprod{\gop(w) - \gop(z)}{w - u} \le \lam\Par{V^r_z(w) + V^r_w(u)}.\]
\end{definition}

As an example of the above definition, we have the following bound when $\gop = \nabla r$, which follows directly from nonnegativity of Bregman divergences and \eqref{eq:threept}.
\begin{restatable}{lemma}{restaternablar}\label{lem:rnablar}
	Let $r: \zset \to \R$ be convex. Then, $\nabla r$ is $1$-relatively Lipschitz with respect to $r$ over $\zset$.
\end{restatable}

As another example, Lemma 1 of \cite{CohenST21} shows that if $\gop$ is $L$-Lipschitz and $r$ is $\mu$-strongly convex (the setup considered in \cite{Nemirovski04}), then $\gop$ is $\frac L \mu$-relatively Lipschitz with respect to $r$ over $\zset$. This setup was generalized by \cite{CohenST21} via Definition~\ref{def:rl}, who showed the following.

\begin{proposition}[Proposition 3, \cite{CohenST21}]\label{prop:smmp}
	If $\gop$ is $\lam$-relatively Lipschitz with respect to $r$ over $\zset_{\textup{alg}}$ containing all iterates of Algorithm~\ref{alg:smmp}, and its VI is solved by $z_\star$, the iterates of Algorithm~\ref{alg:smmp} satisfy 
	\[V^r_{z_t}(z_\star) \le \Par{1 + \frac m \lam}^{t} V^r_{z_0}(z_\star), \text{ for all } t \in [T].\]
\end{proposition}

Our algorithm in this section, Algorithm~\ref{alg:mm}, will simply apply Algorithm~\ref{alg:smmp} to the operator-regularizer pair $(\gop, r)$ defined in \eqref{eq:gdef} and \eqref{eq:rdef}. We give this implementation as pseudocode below, and show that it is a correct implementation of Algorithm~\ref{alg:smmp} in the following lemma.

\begin{algorithm2e}[ht!]
	\caption{$\textsc{Minimax-Solve}(\Freg, x_0, y_0)$: Separable minimax optimization}
	\label{alg:mm}
	\DontPrintSemicolon
		\codeInput \eqref{eq:minimax_reg} satisfying Assumption~\ref{assume:minimax}, $(x_0, y_0) \in \xset \times \yset$\;
		\codeParameter $\lambda>0$, $T \in \N$\;
		 $(z_0\xsup, z_0\ysup) \gets (x_0, y_0)$, $(z_0\xpsup, z_0\ypsup) \gets (x_0, y_0)$
		\For{$0 \le t < T$}{
		\begin{align*}\gop\x\gets \mu\x z\xsup_t + \nabla f(z\xpsup_t) + \nabla_x h(z\xsup_t, z\ysup_t)\\
		\gop\y\gets \mu\y z\ysup_t + \nabla g(z\ypsup_t) - \nabla_y h(z\xsup_t, z\ysup_t).
		\end{align*}	 \Comment*{gradient step:}
		$z\xsup_{t+1/2} \gets z\x - \tfrac{1}{\lam\mu\x} \gop\x$\;
		$z\ysup_{t+1/2} \gets z\y - \tfrac{1}{\lam\mu\y} \gop\y$\;
		$z\xpsup_{t+1/2} \gets (1 - \frac 1 \lam) z\xpsup_t + \frac 1 \lam z\xsup_t$ and  $z\ypsup_{t+1/2} \gets (1 - \frac 1 \lam) z\ypsup_t + \frac 1 \lam z\ysup_t$\;
		\begin{align*}\gop\x\gets \mu\x z\xsup_{t+1/2} + \nabla f(z\xpsup_{t+1/2}) + \nabla_x h(z\xsup_{t+1/2}, z\ysup_{t+1/2})\\
			\gop\y\gets \mu\y z\ysup_{t+1/2} + \nabla g(z\ypsup_{t+1/2}) - \nabla_y h(z\xsup_{t+1/2}, z\ysup_{t+1/2})
		\end{align*}\Comment*{extragradient step:}
		$z\xsup_{t + 1} \gets \frac{1}{1 + \lam} z_{t + 1/2}\x + \frac{\lam}{1 + \lam} z_t\x - \frac{1}{(1 + \lam)\mu\x}\gop\x$\;
		$z\ysup_{t + 1} \gets  \frac{1}{1 + \lam} z_{t + 1/2}\y + \frac{\lam}{1 + \lam} z_t\y - \frac{1}{(1 + \lam)\mu\y}\gop\y$\;
		$z\xpsup_{t + 1} \gets \frac \lam {1 + \lam} z\xpsup_t + \frac 1 {1 + \lam} z\xsup_{t+1/2}$ and  $z\ypsup_{t + 1} \gets \frac \lam {1 + \lam} z\ypsup_t + \frac 1 {1 + \lam} z\ysup_{t+1/2}$
		}
\end{algorithm2e}

\begin{lemma}\label{lem:itercost}
Algorithm~\ref{alg:mm} implements Algorithm~\ref{alg:smmp} with $m = 1$ on $(\gop, r)$ defined in \eqref{eq:gdef}, \eqref{eq:rdef}. 
\end{lemma}
\begin{proof}
Let $\{z_t, z_{t+1/2}\}_{0 \le t \le T}$ be the iterates of Algorithm~\ref{alg:smmp}. We will inductively show that Algorithm~\ref{alg:mm} preserves the invariants
\[z_t = \Par{z\xsup_t, z\ysup_t, \nabla f\Par{z\xpsup_t}, \nabla g\Par{z\ypsup_t}},\; z_{t+1/2} = \Par{z\xsup_{t+1/2}, z\ysup_{t+1/2}, \nabla f\Par{z\xpsup_{t+1/2}}, \nabla g\Par{z\ypsup_{t+1/2}}},\]
for the iterates of Algorithm~\ref{alg:mm}. Once we prove this claim, it is clear from inspection that Algorithm~\ref{alg:mm} implements Algorithm~\ref{alg:smmp}, upon recalling the definitions \eqref{eq:gdef}, \eqref{eq:rdef}.
	
The base case of our induction follows from our initialization so that $(\nabla f(z_0\xpsup), \nabla g(z_0\ypsup)) \gets (\nabla f(x_0), \nabla f(y_0))$. Next, 
suppose for some $0 \le t < T$, we have $z_t\xd = \nabla f( z\xpsup_t)$ and $z_t\yd = \nabla g(z\ypsup_t)$. By the updates in Algorithm~\ref{alg:smmp},
	\begin{align*}z_{t+1/2}\xd &\gets \argmin_{z\xd \in \xset^*} \Brace{\frac 1 \lam \inprod{\nabla f^*(z_t\xd) - z\xsup_t}{z\xd} + V^{f^*}_{z_t\xd}(z\xd)} \\
	&= \argmin_{z\xd \in \xset^*} \Brace{\frac 1 \lam \inprod{z\xpsup_t - z\xsup_t}{z\xd} - \inprod{z_t\xpsup}{z\xd} + f^*(z\xd)} \\
	&= \argmax_{z\xd \in \xset^*} \Brace{\inprod{\Par{1 - \frac 1 \lam} z\xpsup_t + \frac 1 \lam z\x_t}{z\xd} - f^*(z\xd)} = \nabla f\Par{\Par{1 - \frac 1 \lam} z\xpsup_t + \frac 1 \lam z\xsup_t}. \end{align*}
	The second line used our inductive hypothesis and Item 3 in Fact~\ref{fact:dualsc}, and the last used Item 1 in Fact~\ref{fact:dualsc}. Hence, the update to $z\xpsup_{t+1/2}$ in Algorithm~\ref{alg:mm} preserves our invariant; a symmetric argument yields $z_{t+1/2}\yd = \nabla g(z\ypsup_{t+1/2})$ where $z\ypsup_{t+1/2} \defeq (1 - \frac 1 \lam)z\ypsup_t + \frac 1 \lam z\ysup_t$. 
	
	Similarly, we show we may preserve this invariant for $z_{t + 1}$:
	\begin{align*}
	z_{t + 1}\xd &\gets \argmin_{z\xd \in \xset^*} \Brace{\frac 1 \lam \inprod{z\xpsup_{t+1/2} - z\x_{t+1/2}}{z\xd} - \frac 1 \lam \inprod{z\xpsup_{t+1/2}}{z\xd} - \inprod{z\xpsup_{t}}{z\xd} + \Par{1 + \frac 1 \lam}f^*(z\xd)} \\
	&= \argmax_{a \in \xset^*} \Brace{\inprod{z\xpsup_{t} + \frac 1 \lam z\xsup_{t+1/2}}{z\xd} - \Par{1 + \frac 1 \lam}f^*(z\xd)} = \nabla f\Par{\frac \lam {1 + \lam}z\xpsup_{t} + \frac 1 {1 + \lam} z\x_{t+1/2}}.
	\end{align*}
	Hence, we may set $z\xpsup_{t + 1} \defeq \frac \lam {1 + \lam} z\xpsup_t + \frac 1 {1 + \lam} z\xsup_{t+1/2}$ and similarly, $z\ypsup_{t + 1} \defeq \frac \lam {1 + \lam} z\ypsup_t + \frac \lam {1 + \lam} z\ysup_{t+1/2}$.
\end{proof}

As an immediate corollary of Lemma~\ref{lem:itercost}, we have the following characterization of our iterates, recalling the definitions of $\xset^*_f$ and $\yset^*_g$ from Section~\ref{sec:prelims}.

\begin{corollary}\label{cor:zalgmm}
Define the product space $\zset_{\textup{alg}} \defeq \xset \times \yset \times \xset^*_f \times \yset^*_g$, where $\xset^*_f \defeq \{\nabla f(x) \mid x \in \xset\}$ and $\yset^*_g \defeq \{\nabla g(y) \mid y \in \yset\}$. Then all iterates of Algorithm~\ref{alg:mm} lie in $\zset_{\textup{alg}}$.
\end{corollary}

For a point $z \in \zset_{\textup{alg}}$, we define the points $z\xp \defeq \nabla f^*(z\xd)$ and $z\yp \defeq \nabla g^*(z\yd)$. By Item 3 of Fact~\ref{fact:dualsc}, this implies $z\xp, z\yp \in \xset$ since $z\xd$ and $z\yd$ are appropriate gradients.

\subsection{Convergence analysis}\label{ssec:mmconv}

In order to use Proposition~\ref{prop:smmp} to analyze Algorithm~\ref{alg:mm}, we require a strong monotonicity bound and a relative Lipschitzness bound on the pair $(\gop, r)$; the former is already given by Lemma~\ref{lem:sm}. We build up to the latter bound by first giving the following consequences of Assumption~\ref{assume:minimax}.

\begin{lemma}[Minimax smoothness implications]
\label{lem:smoothness_implications}
Let convex $f : \xset \rightarrow \R$ and $g : \yset \rightarrow \R$, and convex-concave $h : \xset \times \yset \rightarrow \R$ satisfy \Cref{assume:minimax}. Then, the following hold.
\begin{enumerate}
	\item \label{item:minimax_fsmooth_equiv} $\Abs{\inprod{\nabla f\Par{v} - \nabla f\Par{w}}{x - y}} \le \alpha L\x V^f_{v}\Par{w} + \alpha^{-1} V_{x}(y)$ for all $v, w, x, y \in \xset$ and $\alpha > 0$.
	\item \label{item:minimax_gsmooth_equiv}  $\Abs{\inprod{\nabla g\Par{v} - \nabla g\Par{w}}{x - y}} \le \alpha L\y V^g_{v}\Par{w} + \alpha^{-1} V_{x}(y)$ for all $v, w, x, y \in \yset$ and $\alpha > 0$.
	\item \label{item:minimax_hsmooth_equiv} $\gop^h$ is 1-relatively Lipschitz with respect to $r^h_\alpha : \zset \rightarrow \R$ defined for all $z \in \zset$ and $\alpha > 0$ by
	$r_\alpha^h(z) \defeq \half\left(\Lambda\xx + \alpha \Lambda\xy \right) \norm{z\x}^2 + \half\left(\Lambda\yy + \alpha^{-1} \Lambda\xy \right)  \norm{z\y}^2$.
\end{enumerate}
\end{lemma}

\begin{proof}
We will prove Items 1 and 3, as Item 2 follows symmetrically to Item 1.

\paragraph{Proof of \Cref{item:minimax_fsmooth_equiv}.} We compute:
	\begin{align*}
		\Abs{\inprod{\nabla f(v) - \nabla{f(w)}}{x - y}} &\le \norm{\nabla f(v) - \nabla f(w)} \norm{x - y}\\
				&\le \frac{\alpha}{2} \norm{\nabla f(v) - \nabla f(w)}^2 + \frac{1}{2\alpha} \norm{x - y}^2 \\
		&\le \alpha L\x V^{f^*}_{\nabla f(w)}(\nabla f(v)) + \alpha^{-1} V_{x}(y)  = \alpha L\x V^f_{v}(w) + \alpha^{-1} V_{x}(y) .
	\end{align*}

	The first inequality was Cauchy-Schwarz, the second was Young's inequality, and the third used Items 3 and 4 in Fact~\ref{fact:dualsc}. The last equality follows from Fact~\ref{fact:dualsc}.
	
	\paragraph{Proof of \Cref{item:minimax_hsmooth_equiv}.} Let $w,v,z \in \zset$ be arbitrary. We have, 
	\begin{flalign*}
	& \inprod{\gop^h (w) - \gop^h(z)}{w - v}
	\\
	& \hspace{3em} = \inprod{\nabla_x h(w\xsup, w\ysup) - \nabla_x h(z\xsup,z\ysup)}{w\xsup - v\xsup} 
	-  \inprod{\nabla_y h(w\xsup,w\ysup) - \nabla_y h(z\xsup, z\ysup)}{w\ysup - v\ysup}.
	\end{flalign*}
	Applying Cauchy-Schwarz, Young's inequality, and \Cref{assume:minimax} yields
	\begin{flalign*}
		& \inprod{\nabla_x h(w\xsup,w\ysup) - \nabla_x h(z\xsup,z\ysup)}{w\x - v\x} \leq \norm{\nabla_x h(w\xsup,w\ysup) - \nabla_x h(z\xsup,z\ysup)} \norm{w\x - v\x} \\
		&\hspace{6em}\leq \left( \Lam\xx \norm{w\x - z\x} + \Lam\xy \norm{w\y - z\y} \right) \norm{w\x - v\x} \\
		&\hspace{6em} \leq  \frac{\Lam\xx}{2} \norm{w\x - z\x}^2 + \frac{\Lam\xx}{2}  \norm{w\x - v\x}^2
		+ \Lam\xy \norm{w\y - z\y}  \norm{w\x - v\x}.
	\end{flalign*}
	Symmetrically,
	\[
	\inprod{\nabla_y h(w\xsup,w\ysup) - \nabla_y h(z\xsup,z\ysup)}{w\y - v\y}  \leq  \frac{\Lam\yy}{2} \norm{w\y - z\y}^2 + \frac{\Lam\yy}{2}  \norm{w\y - v\y}^2
	+ \Lam\xy  \norm{w\x - z\x}  \norm{w\y - v\y}.
	\]
	Applying Young's inequality again yields
	\begin{align*}
	\Lam\xy \norm{w\y - z\y}  \norm{w\x - v\x} &\leq \frac{\alpha \Lam\xy}{2}   \norm{w\x - v\x}^2 + \frac{ \Lam\xy}{2 \alpha} \norm{w\y - z\y}^2, \\
	\text{and } \Lam\xy  \norm{w\x - z\x}  \norm{w\y - v\y} &\leq \frac{\alpha \Lam\xy}{2}   \norm{w\x - z\x}^2 + \frac{ \Lam\xy}{2 \alpha} \norm{w\y - v\y}^2.
	\end{align*}
	Combining these inequalities yields the desired bound of
	\begin{align*}
	\inprod{\gop^h (w) - \gop^h(z)}{w - v}
	&\leq \left(\Lambda\xx + \alpha \Lambda\xy\right) \left( V_{z\xsup}(w\xsup) + V_{w\xsup}(v\xsup) \right)
	+ \left(\Lambda\yy + \alpha \Lambda\xy\right) \left( V_{z\ysup}(w\ysup) + V_{w\ysup}(v\ysup) \right) \\
	&= V^{r^h_\alpha}_z(w) + V^{r^h_\alpha}_w(v).
	\end{align*}
\end{proof}

Leveraging~\Cref{lem:smoothness_implications} and~\Cref{lem:rnablar}, we prove relative Lipschitzness of $\gop$ with respect to $r$ in \Cref{lem:rl}. Interestingly, the implications in Lemma~\ref{lem:smoothness_implications} are sufficient for this proof, and this serves as a (potentially) weaker replacement for \Cref{assume:minimax} in yielding a convergence rate for our method.

This is particularly interesting when the condition in \Cref{item:minimax_fsmooth_equiv} is replaced with a non-Euclidean divergence, namely $\Abs{\inprod{\nabla f\Par{v} - \nabla f\Par{w}}{x - y}} \le \alpha L\x V^f_{v}\Par{w} + \alpha^{-1} V^\omega_{x}(y)$ for some convex $\omega: \xset \to \R$. Setting, setting $v = y, w = x, \alpha = \frac 1 {L\x}$ in this condition yields $V^f_x(y) \le LV^\omega_x(y)$. Hence, this extension to \Cref{item:minimax_fsmooth_equiv} generalizes \emph{relative smoothness} between $f$ and $\omega$, a condition introduced by \cite{BauschkeBT17, LuFN18}. It has been previously observed \cite{HanzelyRX18, DragomirTAB19} that relative smoothness alone does not suffice for accelerated rates. \Cref{item:minimax_fsmooth_equiv} provides a new strengthening of relative smoothness which, as shown by its (implicit) use in \cite{CohenST21}, suffices for acceleration. We believe a more thorough investigation comparing these conditions is an interesting avenue for future work.

\begin{lemma}[Relative Lipschitzness]\label{lem:rl}
	Define $\gop: \zset \to \zset^*$ as in \eqref{eq:gdef}, and define $r: \zset \to \R$ as in \eqref{eq:rdef}. Then $\gop$ is $\lam$-relatively Lipschitz with respect to $r$ over $\zset_{\textup{alg}}$ defined in Corollary~\ref{cor:zalgmm} for
	\begin{equation}\label{eq:lamdefmm}\lam = 1 + \sqrt{\frac{L\x}{\mu\x}} + \sqrt{\frac{L\y}{\mu\y}} + \frac{\Lam\xx}{\mu\x} + \frac{\Lam\xy}{\sqrt{\mu\x\mu\y}} + \frac{\Lam\yy}{\mu\y}.\end{equation}
\end{lemma}
\begin{proof}
	Let $w,v,z \in \zset_{\textup{alg}}$. We wish to show (cf.\ Definition~\ref{def:rl})
	\[
	\inprod{\gop(w) - \gop(z)}{w - v} \leq \lambda \left(V^r_z(w) + V^r_w(v) \right).
	\]
	Since $\gop = \gop^r + \gcross +\gop^h$ (cf.\ \eqref{eq:gdef}), we bound the contribution of each term individually. The conclusion follows from combining \eqref{eq:sepbound}, \eqref{eq:crossbound}, and \eqref{eq:lipschitz_h_bound}.
	
	\paragraph{Bound on $\gop^r$:} By applying Lemma~\ref{lem:rnablar} to $r$,
	\begin{equation}\label{eq:sepbound}
	\begin{aligned}
	\inprod{\gop^r(w) - \gop^r(z)}{w - v} =
	\inprod{\nabla r(w) - \nabla r(z)}{w - v} &\le V^r_z(w) + V^r_w(v).
	\end{aligned}
	\end{equation}
	\paragraph{Bound on $\gcross$:} For all $a \in \zset_{\textup{alg}}$, we may write for some $a\xp\in \xset$ and $a\yp \in \yset$,
	\begin{align*}
	\gcross(a) &= (a\xd, a\yd, -a\x, -a\y) = (\nabla f(a\xp), \nabla g(a\yp), -a\x, -a\y) \\
	\text{ and } a &= (a\x, a\y, a\xd, a\yd) = (a\x, a\y, \nabla f(a\xp), \nabla g(a\yp)  ) .
	\end{align*}
	Consequently,
	\begin{align*}
	\inprod{\gcross(w) - \gcross(z)}{w - v} 
	&= \inprod{\nabla f(w\xp) - \nabla f(z\xp)}{w\x - v\x} + \inprod{\nabla g(w\xp) - \nabla g(z\xp)}{w\y - v\y}\\
	&- \inprod{ w\x - z\x}{ \nabla f(w\xp) - \nabla f(v\xp)} - \inprod{w\y - z\y}{\nabla g(w\yp) - \nabla g(v\yp)}.
	\end{align*}
    Applying \Cref{lem:smoothness_implications} (\Cref{item:minimax_fsmooth_equiv} and \Cref{item:minimax_gsmooth_equiv}) to each term, with $\alpha = (\mu\x L\x)^{-\half}$ for terms involving $f$ and $\alpha =  (\mu\y L\y)^{-\half}$ for terms involving $g$ yields
    \begin{align*}
    \inprod{\gcross(w) - \gcross(z)}{w - v}  &\leq \sqrt{\frac{L\x}{\mu\x}} \left(V_{w\xp}^f(z\xp) + V_{v\xp}^f(w\xp) \right) + \sqrt{\frac{L\x}{ \mu\x} } \left(\mu\x V_{w\x}(v\x) + \mu\x V_{z\x}(w\x) \right) 
    \\
    &+ \sqrt{\frac{L\y}{ \mu\y}} \left(V_{w\yp}^g(z\yp) + V_{v\yp}^g(w\yp) \right) + \sqrt{\frac{L\y}{ \mu\y}} \left(\mu\y V_{w\y}(v\y) + \mu\y V_{z\y}(w\y) \right).
    \end{align*} 
	Applying Item 3 in Fact~\ref{fact:dualsc} and recalling the definition of $r$ \eqref{eq:rdef} yields
	\begin{align}
	\label{eq:crossbound}
		\inprod{\gcross(w) - \gcross(z)}{w - v} 
		\leq\Par{\sqrt{\frac{L\x}{\mu\x}} + \sqrt{\frac{L\y}{\mu\y}} } 
		 \Par{V^r_z(w) + V^r_w(v)}.
	\end{align}	
	
	\paragraph{Bound on $\gop^{h}$:} 
	Applying \Cref{lem:smoothness_implications} (\Cref{item:minimax_hsmooth_equiv} with $\alpha = \sqrt{\mu\x / \mu\y}$), we have that $\gop^{h}$ is $1$-relatively Lipschitz with respect to $r^h_\alpha : \zset  \rightarrow \R$ defined for all $z \in \xset$ and $\alpha > 0$ by
	\begin{align*}
		r_\alpha^h(z) &\defeq \half \left(\Lambda\xx + \alpha \Lambda\xy \right)  \norm{z\x}^2 + \half \left(\Lambda\yy + \alpha^{-1} \Lambda\xy \right)  \norm{z\y}^2 \\
		&= \left(\frac{\Lambda\xx}{\mu\x} +  \frac{\Lambda\xy}{\sqrt{\mu\x \mu\y}} \right) \cdot \frac{\mu\x} 2 \norm{z\x}^2+ \left(\frac{\Lambda\yy}{\mu\y} +   \frac{\Lambda\xy}{\sqrt{\mu\x \mu\y}} \right) \cdot \frac{\mu\y} 2 \norm{z\y}^2.
	\end{align*}
	Leveraging the nonnegativity of Bregman divergences, we conclude
	\begin{align}
		\inprod{\gop^h(w) - \gop^h(z)}{w - v}
		&\leq V^{r^h_\alpha }_z(w) + V^{r^h_\alpha }_w(v) \nonumber \\
		&\leq \Par{\frac{\Lam\xx}{\mu\x} + \frac{\Lam\xy}{\sqrt{\mu\x\mu\y}} + \frac{\Lam\yy}{\mu\y}} 
		\Par{V^r_z(w) + V^r_w(v)}.
		\label{eq:lipschitz_h_bound}
	\end{align}
\end{proof}

Finally, we provide simple bounds regarding initialization and termination of Algorithm~\ref{alg:mm}.

\begin{lemma}\label{lem:initbound}
Let $(x_0, y_0) \in \xset \times \yset$, and define
\begin{equation}\label{eq:z0def}z_0 \defeq \Par{x_0, y_0, \nabla f(x_0), \nabla g(y_0)}.\end{equation}
Suppose $\gap_{\Freg}(x_0, y_0) \le \Eps_0$. Then, letting $z_\star$ be the solution to \eqref{eq:primaldual},
	\[V^r_{z_0}(z_\star) \le \Par{1 + \frac{L\x}{\mu_x} + \frac{L\y}{\mu\y}}\Eps_0. \]
\end{lemma}
\begin{proof}
By the characterization in Lemma~\ref{lem:sameopt}, we have by Item 1 in Fact~\ref{fact:dualsc}:
\[z_\star = \Par{x_\star, y_\star, \nabla f(x_\star), \nabla g(y_\star)}.\]
	Hence, we bound
	\begin{align*}V^r_{z_0}(z_\star) &= \mu\x V_{x_0}(x_\star) + V^{f}_{x_\star}(x_0) + \mu\y V_{y_0}(y_\star) + V^{g}_{y_\star}(y_0) \\
	&\le  \mu\x V_{x_0}(x_\star) + \frac{L\x}{2} \normx{x_0 - x_\star}^2 + \mu\y V_{y_0}(y_\star) + \frac{L\y}{2} \normy{y_0 - y_\star}^2 \\
	&= \Par{\frac{L\x}{\mu\x} + 1}\mu\x V_{x_0}(x_\star) + \Par{\frac{L\y}{\mu\y} + 1}\mu\y V_{y_0}(y_\star)\\
	& \le \Par{\frac{L\x}{\mu\x} + \frac{L\y}{\mu\y} + 1}\Eps_0.
	\end{align*}
	The first line used Item 3 in Fact~\ref{fact:dualsc}, and the second used smoothness of $f$ and $g$ (Assumption~\ref{assume:minimax}). To obtain the last line, define the functions
	\[\Freg\x(x) \defeq \max_{y \in \yset} \Freg(x, y) \text{ and } \Freg\y(y) \defeq \min_{x \in \xset} \Freg(x, y). \]
	\Cref{fact:br} shows $\Freg\x$ is $\mu\x$-strongly convex and $\Freg\y$ is $\mu\y$-strongly concave, so
	\begin{align*}\gap_{\Freg}(x_0, y_0) &= \Par{\Freg\x(x_0) - \Freg\x(x_\star)} + \Par{\Freg\y(y_\star) - \Freg\y(y_0)} \\
		&\ge \mu\x V_{x_0}(x^\star) + \mu\y V_{y_0}(y^\star).\end{align*}
\end{proof}

\begin{lemma}\label{lem:termbound}
Let $z \in \zset$ have 
\[V^r_z(z_\star) \le \Par{\frac{\mu\x + L\x + \Lam\xx}{\mu\x} + \frac{\mu\y + L\y + \Lam\yy}{\mu\y} + \frac{(\Lam\xy)^2}{\mu\x\mu\y}} \cdot \frac \eps 2,\] 
for $z_\star$ the solution to \eqref{eq:primaldual}. Then,
\[\gap_{\Freg}(z\x, z\y) \le \eps.\] 
\end{lemma}
\begin{proof}
We follow the notation of Lemma~\ref{lem:initbound}. From~\Cref{fact:br} we know $\Freg\x$ is $\mathcal{L}\x$-smooth and $\Freg\y$ is $\mathcal{L}\y$-smooth, where 
\[\mathcal{L}\x \defeq \mu\x+L\x + \Lam\xx + \frac{(\Lam\xy)^2}{\mu\y}\text{ and }\mathcal{L}\y \defeq \mu\y+L\y + \Lam\yy + \frac{(\Lam\xy)^2}{\mu\x},\]
under Assumption~\ref{assume:minimax}. Moreover, by Lemma~\ref{lem:sameopt} and the definition of saddle points, $x_\star \defeq z_\star\x$ is the minimizer to $\Freg\x$, and $y_\star \defeq z_\star\y$ is the maximizer to $\Freg\y$. We conclude via
\begin{align*}
\gap_{\Freg}(z\x, z\y) &= \Par{\Freg\x(x) - \Freg\x(x_\star)} + \Par{\Freg\y(y_\star) - \Freg\y(z\y)} \\
&\le \Par{\mu\x + L\x + \Lam\xx + \frac{(\Lam\xy)^2}{\mu\y}}\norm{x - x_\star}^2 \\
&+ \Par{\mu\y + L\y + \Lam\yy + \frac{(\Lam\xy)^2}{\mu\x}} \norm{y - y_\star}^2 \\
&\le 2\Par{\frac{\mu\x + L\x + \Lam\xx}{\mu\x} + \frac{\mu\y + L\y + \Lam\yy}{\mu\y} + \frac{(\Lam\xy)^2}{\mu\x\mu\y}}V^r_z(z_\star) \le \eps.
\end{align*}
The first inequality was smoothness of $\Freg\x$ and $\Freg\y$ (where we used that the gradients at $x_\star$ and $y_\star$ vanish because the optimization problems they solve are over unconstrained domains), and the last inequality was nonnegativity of Bregman divergences.
\end{proof}

%

\subsection{Main result}\label{ssec:mmres}

We now state and prove our main claim.

\begin{restatable}{theorem}{restatemain}\label{thm:main}
	Suppose $\Freg$ in \eqref{eq:minimax_reg} satisfies Assumption~\ref{assume:minimax}, and suppose we have $(x_0, y_0) \in \xset \times \yset$ such that $\gap_{\Freg}(x_0, y_0) \le \Eps_0$. 
	Algorithm~\ref{alg:mm} with $\lam$ as in \eqref{eq:lamdefmm} returns $(x, y) \in \xset \times \yset$ with $\gap_{\Freg}(x, y) \le \eps$ in $T$ iterations, using a total of $O(T)$ gradient calls to each of $f$, $g$, $h$, where
	\begin{equation}\label{eq:tmm}T = O\Par{\kappa_{\textup{mm}} \log\Par{ \frac{\kappa_{\textup{mm}}\Eps_0}{\eps}}}, \text{ for } \kappa_{\textup{mm}} \defeq \sqrt{\frac{L\x}{\mu\x}} + \sqrt{\frac{L\y}{\mu\y}} + \frac{\Lam\xx}{\mu\x} + \frac{\Lam_{xy}}{\sqrt{\mu\x\mu\y}} + \frac{\Lam\yy}{\mu\y}. \end{equation}
\end{restatable}
\begin{proof}
	By Lemma~\ref{lem:sameopt}, the points $x_\star$ and $y_\star$ are consistent between \eqref{eq:minimax_reg} and \eqref{eq:primaldual}. The gradient complexity of each iteration follows from observation of Algorithm~\ref{alg:mm}.
	
	Next, by Lemma~\ref{lem:itercost}, Algorithm~\ref{alg:mm} implements Algorithm~\ref{alg:smmp} on the pair \eqref{eq:gdef}, \eqref{eq:rdef}. By substituting the bounds on $\lam$ and $m$ in Lemmas~\ref{lem:rl} and~\ref{lem:sm} into Proposition~\ref{prop:smmp} (where we define $\zset_{\textup{alg}}$ as in Corollary~\ref{cor:zalgmm}), it is clear that after $T$ iterations (for a sufficiently large constant in the definition of $T$), we will have $V^r_{z_T}(z_\star)$ is bounded by the quantity in Lemma~\ref{lem:termbound}, where we use the initial bound on $V^r_{z_0}(z^\star)$ from Lemma~\ref{lem:initbound}. The conclusion follows from setting $(x, y) \gets (z_T\x, z_T\y)$.
\end{proof}

As an immediate corollary, we have the following result on solving \eqref{eq:minimax}.

\begin{corollary}\label{cor:minimax}
Suppose for $\Fmm$ in \eqref{eq:minimax} solved by $(x_\star, y_\star)$, $(f - \frac{\mu\x} 2 \norm{\cdot}^2, g - \frac{\mu\y} 2 \norm{\cdot}^2, h)$ satisfies Assumption~\ref{assume:minimax}. There is an algorithm taking $(x_0, y_0) \in \xset \times \yset$ satisfying $\gap_{\Fmm}(x_0, y_0) \le \Eps_0$, which performs $T$ iterations for $T$ in \eqref{eq:tmm}, returns $(x, y) \in \xset \times \yset$ satisfying $\gap_{\Fmm}(x, y) \le \eps$, and uses a total of $O(T)$ gradient calls to each using $O(1)$ gradient calls to each of $f$, $g$, $h$.
\end{corollary}

\section{Finite sum optimization}
\label{sec:finitesum}

In this section, we give an algorithm for efficiently finding an approximate minimizer of the following finite sum optimization problem:
\begin{equation}\label{eq:origfs}
\Ffs(x) \defeq \frac 1 n \sum_{i \in [n]} f_i(x).
\end{equation}
Here and throughout this section $f_i: \xset \to \R$ is a differentiable, convex function for all $i \in [n]$. For the remainder, we focus on algorithms for solving the following regularized formulation of \eqref{eq:origfs}:
\begin{equation}\label{eq:regfs}
\min_{x \in \xset} \Ffsreg(x) \text{ for } \Ffsreg(x) \defeq \nsin f_i(x) + \frac \mu 2 \norm{x}^2 .
\end{equation}
As in Section~\ref{sec:minimax}, to solve an instance of \eqref{eq:origfs} where each $f_i$ is $\mu$-strongly convex, we may instead equivalently solve \eqref{eq:regfs} by reparameterizing $f_i \gets f_i - \frac \mu 2 \norm{\cdot}^2$ for all $i \in [n]$. We further remark that our algorithms extend to solve instances of \eqref{eq:origfs} where $\Ffs$ is $\mu$-strongly convex in $\norm{\cdot}$, but individual summands are not. We provide this result at the end of the section in Corollary~\ref{cor:mainfsunreg}.

In designing methods for solving \eqref{eq:regfs} we make the following additional regularity assumptions.
\begin{assumption}\label{assume:fs}
	For all $i \in [n]$, $f_i$ is $L_i$-smooth.
\end{assumption}

The remainder of this section is organized as follows.
\begin{enumerate}
	\item In Section~\ref{ssec:fssetup}, we state a primal-dual formulation of \eqref{eq:regfs} which we will apply our methods to, and prove that its solution also yields a solution to \eqref{eq:regfs}.
	\item In Section~\ref{ssec:fsalgo}, we give our algorithm and prove it is efficiently implementable.
	\item In Section~\ref{ssec:fsconv}, we prove the convergence rate of our algorithm.
	\item In Section~\ref{ssec:fsres}, we state and prove our main result, Theorem~\ref{thm:mainfs}.
\end{enumerate}

\subsection{Setup}\label{ssec:fssetup}

To solve \eqref{eq:regfs}, we instead find a saddle point to the primal-dual function
\begin{equation}\label{eq:pdfs}\Ffspd\Par{z} \defeq \frac 1 n \sum_{i \in [n]} \Par{\inprod{z\xdi}{z\x} - f_i^*(z\xdi)} + \frac \mu 2 \norm{z\x}^2. \end{equation}
We denote the domain of $\Ffspd$ by $\zset \defeq \xset \times (\xset^*)^n$. For $z \in \zset$, we refer to its blocks by $(z\x, \bin{z\xdi})$. The primal-dual function $\Ffspd$ is related to $\Ffsreg$ in the following way.
\begin{lemma}\label{lem:sameoptfs}
	Let $z_\star$ be the saddle point to \eqref{eq:pdfs}. Then, $z_\star\x$ is a minimizer of \eqref{eq:regfs}.
\end{lemma}
\begin{proof}
	By performing the maximization over each $z\xdi$, we see that the problem of computing a minimizer to the objective in \eqref{eq:pdfs} is equivalent to
	\[\min_{z\xsup \in \xset} \frac \mu 2 \norm{z\x}^2 + \frac 1 n \sum_{i \in [n]} \Par{\max_{z\xdi \in \xset^*} \inprod{z\xdi}{z\xsup} - f^*_i(z\xdi)}.\]
	By Item 2 in Fact~\ref{fact:dualsc}, this is the same as \eqref{eq:regfs}.
\end{proof}

As in Section~\ref{ssec:mmsetup}, it will be convenient to define the convex function $r: \zset \to \R$, which combines the (unsigned) separable components of $\Ffspd$:
\begin{equation}\label{eq:rfsdef}r\Par{z} \defeq \frac \mu 2 \norm{z\x}^2 + \frac 1 n \sum_{i \in [n]} f^*_i(z\xdi).\end{equation}
Again, $r$ serves as a regularizer in our algorithm. We next define $\gop$, the gradient operator of $\Ffspd$:
\begin{equation}\label{eq:gfsdef}\gop(z) \defeq \Par{\frac 1 n \sum_{i \in [n]} z\xdi + \mu z\x, \Brace{ \frac 1 n \Par{\nabla f^*_i(z\xdi) - z\x}}_{i \in [n]}}.\end{equation}

By construction, $\gop$ is $1$-strongly monotone with respect to $r$.

\begin{lemma}[Strong monotonicity]\label{lem:smfs}
	Define $\gop: \zset \to \zset^*$ as in \eqref{eq:gfsdef}, and define $r: \zset \to \R$ as in \eqref{eq:rfsdef}. Then $\gop$ is $1$-strongly-monotone with respect to $r$.
\end{lemma}
\begin{proof}
	The proof is identical to Lemma~\ref{lem:sm} without the $\gop^h$ term: the bilinear component cancels in the definition of strong monotonicity, and the remaining part is exactly the gradient of $r$.
\end{proof}

\subsection{Algorithm}\label{ssec:fsalgo}

Our algorithm is an instantiation of \emph{randomized mirror prox} \cite{CohenST21} stated as Algorithm~\ref{alg:rmp} below, an extension to mirror prox allowing for randomized gradient estimators. We note that the operators $\gop_i$ need only be defined on iterates of the algorithm.

\begin{algorithm2e}[ht!]
	\caption{$\textsc{Rand-Mirror-Prox}( \{\gop_i\}_{i\in[n]}, w_0)$: Randomized mirror prox \cite{CohenST21}}
	\label{alg:rmp}
		\codeInput Convex $r: \zset \to \R$, probability distribution $p: [n] \to \R_{\ge 0}$ with $\sum_{i \in [n]} p_i = 1$, operators $\{\gop_i\}_{i \in [n]}: \zset \to \zset^*$, $z_0 \in \zset$\;
		\codeParameter $\lam > 0$, $S \in \N$  \;
		\For{$0 \le s < S$}{
		Sample $i \sim p$\;
		$w_{s+1/2} \gets \Prox^r_{w_t}(\tfrac{1}{\lam}\gop_i(w_s))$\;
		$w_{s + 1} \gets \Prox^r_{w_t}(\tfrac{1}{\lam}\gop_i(w_{s+1/2}))$\;
		}
\end{algorithm2e}

We provide the following result from \cite{CohenST21} giving a guarantee on Algorithm~\ref{alg:rmp}.

\begin{proposition}[Proposition 2, \cite{CohenST21}]\label{prop:rmp}
	Suppose $\{\gop_i\}_{i \in [n]}$ are defined so that in each iteration $s$, for all $u \in \zset$, there exists a point $\bw_s \in \zset$ and a monotone operator $\gop: \zset \to \zset^*$ such that (where all expectations fix $w_s$, and condition only on the randomness in iteration $s$)
	\begin{equation}\label{eq:giprops}
	\begin{aligned}
	\E_{i \sim p}\Brack{\inprod{\gop_i(w_{s+1/2})}{w_{s+1/2} - u}} &= \inprod{\gop(\bw_s)}{\bw_s - u} \text{ for all } u \in \zset,\\
	\E_{i \sim p}\Brack{\inprod{\gop_i(w_{s+1/2}) - \gop_i(w_s)}{w_{s+1/2} - w_{s + 1}}} &\le \lam\E_{i \sim p}\Brack{V^r_{w_s}(w_{s+1/2}) + V^r_{w_{s+1/2}}(w_{s + 1})}.
	\end{aligned}
	\end{equation}
	Then (where the expectation below is taken over the randomness of the entire algorithm):
	\begin{align*}
	\E\Brack{\frac 1 S \sum_{0 \le s < S} \inprod{\gop(\bw_s)}{\bw_s - u}} \le \frac{\lam V^r_{w_0}(u)}{S}, \text{ for all } u \in \zset.
	\end{align*}
\end{proposition}

The first condition in \eqref{eq:giprops} is an ``unbiasedness'' requirement on the operators $\{\gop_i\}_{i \in [n]}$ with respect to the operator $\gop$, for which we wish to conclude a regret guarantee. The second posits that relative Lipschitzness (Definition~\ref{def:rl}) holds in an expected sense. We recall that Algorithm~\ref{alg:rmp} requires us to specify a set of sampling probabilities $\{p_i\}_{i \in [n]}$. We define 
\begin{equation}\label{eq:pdef}p_i \defeq \frac{\sqrt{L_i}}{2\sum_{j \in [n]} \sqrt{L_j}} + \frac{1}{2n} \text{ for all } i \in [n].\end{equation}
This choice crucially ensures that all $p_i \ge \frac 1 {2n}$, and that all $\frac{\sqrt{L_i}}{p_i} \le 2\sum_{j \in [n]} \sqrt{L_j}$.

Our algorithm, Algorithm~\ref{alg:fs}, recursively applies Algorithm~\ref{alg:rmp} to the operator-pair $(\gop, r)$ defined in \eqref{eq:gfsdef} and \eqref{eq:rfsdef}, for an appropriate specification of $\bin{\gop_i}$. We give this implementation as pseudocode in Algorithms~\ref{alg:fs} and~\ref{alg:fsop} below, and show that Algorithm~\ref{alg:fsop} is a correct implementation of Algorithm~\ref{alg:rmp} with respect to our specified $\bin{\gop_i}$ in the remainder of the section.

\begin{algorithm2e}[ht!]
	\caption{$\textsc{Finite-Sum-Solve}(\Ffsreg, x_0)$: Finite sum optimization}
	\label{alg:fs}
		\codeInput \eqref{eq:regfs} satisfying Assumption~\ref{assume:fs}, $x_0 \in \xset$\;
		\codeParameter $T\in \N$\;
		$z_0\xsup \gets x_0$, $z\xpisup_0 \gets x_0$, $z\xdisup_0 \gets \nabla f_i(x_0)$ for all $i \in [n]$\;\label{line:fs-initialization}
		\For{$0 \le t < T$}{
		$z_{t+1} \gets \textsc{Finite-Sum-One-Phase}(\Ffsreg, z_t)$\;
		}
\end{algorithm2e}

\begin{algorithm2e}[ht!]
	\caption{$\textsc{Finite-Sum-One-Phase}(\Ffsreg, w_0)$: Finite sum optimization subroutine}
	\label{alg:fsop}
		\codeInput \eqref{eq:regfs} satisfying Assumption~\ref{assume:fs}, $w_0 \in \zset$ specified by $w_0\x, \{w_0\xpi\}_{i \in [n]} \in \xset$\;
		\codeParameter $\lam \ge 2$, $S \in \N$\;
		Sample $0 \le \sigma < S$ uniformly at random\;
		\For{$0 \le s \le \sigma$}{
		Sample $j \in [n]$ according to $p$ defined in \eqref{eq:pdef}\;
		$w\x_{s+1/2} \gets w_s\x - \frac{1}{\lam\mu}(\mu w\x_s + \nsin \nabla f_i(w\xpisup_s))$\;
		$w\xpssup{j}_{s+1/2} \gets (1 - \frac{1}{\lam n p_j}) w\xpssup{j}_s + \frac 1 {\lam n p_j} w\x_s$\;
		$w\xpssup{i}_{s+1/2} \gets w\xpssup{i}_{s} $ for all $i \neq j$\;
		$\Delta_s \gets \nabla f_j(w\xpssup{j}_{s+1/2} ) - \nabla f_j(w\xpssup{j}_{s} )$\;
		$w\x_{s + 1} \gets w_s\x - \frac{1}{\lam\mu}(\mu w\x_{s+1/2} + \nsin \nabla f_i(w\xpssup{i}_{s} ) + \frac{1}{np_j} \Delta_s)$\;
		$w\xpssup{j}_{s + 1} \gets w\xpssup{j}_{s}  + \frac{1}{\lam n p_j} (w\xsup_{s+1/2} - w\xpssup{j}_{s+1/2} )$\;
		$w\xpssup{i}_{s + 1} \gets w\xpssup{i}_{s}$ for all $i \neq j$\;
		}
		\codeReturn $(w_{\sigma + 1/2}\x, \{\nabla f_i((1 - \frac{1}{\lam n p_i}) w_\sigma\xpi + \frac 1 {\lam n p_i} w_\sigma\x)\}_{i \in [n]})$
\end{algorithm2e}

We next describe the operators $\bin{\gop_i}$ used in our implementation of Algorithm~\ref{alg:rmp}. Fix some $0 \le s < S$, and consider some iterates $\{w, w\mdpt(j)\} \defeq \{w_s, w_{s+1/2}\}$ of Algorithm~\ref{alg:rmp} (where we use the notation $(j)$ to mean the iterate that would be taken if $j \in [n]$ was sampled in iteration $s$, and we drop the subscript $s$ for simplicity since we only focus on one iteration). We denote the $\xset$ block of $w\mdpt(j)$ by $w\mdpt\x$, since (as made clear in the following) conditioned on $w$, $w\mdpt\x$ is always the same regardless of the sampled $j \in [n]$. For all $j \in [n]$, we then define the operators
\begin{equation}\label{eq:gidef}\begin{aligned}\gop_j\Par{w} &\defeq \Par{\frac 1 n \sum_{i \in [n]} w\xdi + \mu w\x, \Brace{\frac 1 {np_j} (\nabla f^*_j(w\xdj) - w\x) \cdot \1_{i = j}}} , \\
\gop_j\Par{w\mdpt(j)} &\defeq \Par{\frac 1 n \sum_{i \in [n]} w\xdi + \frac 1 {np_j}\Par{w\mdpt\xdj(j) - w\xdj} + \mu w\mdpt\x, \Brace{\frac 1 {np_j}\Par{\nabla f^*_j\Par{w\mdpt\xdj(j)} - w\mdpt\x} \cdot \1_{i = j}}},
\end{aligned}
\end{equation}
where $\1_{i = j}$ is a zero-one indicator. In other words, $\gop_j(w)$ and $\gop_j(w\mdpt(j))$ both only have two nonzero blocks, corresponding to the $\xset$ and $j^{\text{th}}$ $\xset^*$ blocks. We record the following useful observation about our randomized operators \eqref{eq:gidef}, in accordance with the first condition in \eqref{eq:giprops}. To give a brief interpretation of our ``aggregate point'' defined in \eqref{eq:bzdef}, the $\xset$ coordinate is updated deterministically from $w\x$ according to the corresponding block of $\gop$, and every dual block $j \in [n]$ of $\bw$ is set to the corresponding dual block had $j$ been sampled in that step.

\begin{lemma}[Expected regret]\label{lem:exregret}
	Define $\{\gop_j\}_{j \in [n]}: \zset \to \zset^*$ as in \eqref{eq:gidef}, and the ``aggregate point''
	\begin{equation}\label{eq:bzdef}\bar{w} \defeq \Par{w\mdpt\x, \Brace{w\mdpt\xdj(j)}_{j \in [n]}}.\end{equation}
	Then, for all $u \in \zset$, defining $\gop$ as in \eqref{eq:gfsdef},
	\[\E_{j \sim p}\Brack{\inprod{\gop_j(w\mdpt(j))}{w\mdpt(j) - u}} = \inprod{\gop(\bw)}{\bw - u}.\]
\end{lemma}
\begin{proof}
	We expand the expectation, using \eqref{eq:gidef} and taking advantage of the sparsity of $\gop_j$:
	\begin{flalign*}
	& \E_{j \sim p}\Brack{\inprod{\gop_j(w\mdpt(j))}{w\mdpt(j) - u}}\\
	& \hspace{5em} = \inprod{\sum_{j \in [n]} p_j \Par{\frac 1 n \sum_{i \in [n]} w\xdi + \frac 1 {np_j} \Par{w\mdpt\xdj(j) - w\xdj} + \mu w\mdpt\x} }{w\mdpt\x - u\x} \\
	&\hspace{7em}+ \sum_{j \in [n]} p_j \inprod{\frac 1 {np_j}\Par{\nabla f^*_j\Par{w\mdpt\xdj(j)} - w\mdpt\x}}{w\mdpt\xdj(j) - u\xdj} \\
	&\hspace{5em}= \inprod{\frac 1 n \sum_{j \in [n]} w\mdpt\xdj(j) + \mu w\mdpt\x }{w\mdpt\x - u\x} \\
	&\hspace{6em}+ \sum_{j \in [n]} \inprod{\frac 1 n\Par{\nabla f^*_j\Par{w\mdpt\xdj(j)} - w\mdpt\x}}{w\mdpt\xdj(j) - u\yj} = \inprod{\gop(\bw)}{\bw - u}.
	\end{flalign*}
\end{proof}
We conclude this section by demonstrating that Algorithm~\ref{alg:fsop} is an appropriate implementation of Algorithm~\ref{alg:rmp}.

\begin{lemma}[Implementation]\label{lem:itercostfs}
Algorithm~\ref{alg:fsop} implements Algorithm~\ref{alg:rmp} on $(\bin{\gop_i}, r)$ defined in \eqref{eq:gidef}, \eqref{eq:rfsdef}, for $\sigma$ iterations, and returns $\bw_\sigma$, following the definition \eqref{eq:bzdef}. Each iteration $s > 0$ is implementable in $O(1)$ gradient calls to some $f_i$, and $O(1)$ vector operations on $\xset$.
\end{lemma}
\begin{proof}
Let $\{w_s, w_{s+1/2}\}_{0 \le s \le \sigma}$ be the iterates of Algorithm~\ref{alg:rmp}. We will inductively show that Algorithm~\ref{alg:fsop} preserves the invariants
\[w_s = \Par{w\x_s, \bin{\nabla f_i(w\xpisup_s)}},\; w_{s+1/2} = \Par{w\x_s, \Brace{\nabla f_i(w\xpisup_{s+1/2})}_{i\in[n]}}\]
for all $0 \le s \le \sigma$. Once we prove this claim, it is clear from inspection that Algorithm~\ref{alg:fsop} implements Algorithm~\ref{alg:rmp} and returns $\bw_\sigma$, upon recalling the definitions \eqref{eq:gidef}, \eqref{eq:rfsdef}, and \eqref{eq:bzdef}.
	
The base case of our induction follows from the initialization guarantee of~\Cref{line:fs-initialization} in~\Cref{alg:fsop}. Next, suppose for some $0\le s\le \sigma$, we have $w_s\xdi = \nabla f(w\xpi_s)$ for all $i \in [n]$. By the updates in Algorithm~\ref{alg:rmp}, if $j \in [n]$ was sampled on iteration $s$,
	\begin{align*}w_{s+1/2}\xdj &\gets \argmin_{w\xdj \in \xset^*}\Brace{\frac 1 {\lam n p_j} \inprod{w\xdj_s - w\x_s}{w\xdj} - \inprod{w\xdj_s}{w\xdj} + f^*_j\Par{w\xdj}} \\
	&= \argmax_{w\xdj \in \xset^*}\Brace{\inprod{\Par{1 - \frac 1 {\lam n p_j}} w\xdj_s + \frac 1 {\lam n p_j} w\x_s}{w\xdj} - f^*_j\Par{w\xdj}} \\
	&= \nabla f_j\Par{\Par{1 - \frac 1 {\lam n p_j}} w\xdj_s + \frac 1 {\lam n p_j} w\x_s}.
	\end{align*}
	Here, we used the first item in Fact~\ref{fact:dualsc} in the last line. Hence, the update to $w\xdj_{s+1/2}$ in Algorithm~\ref{alg:fsop} preserves our invariant, and all other $w\xdi_{s+1/2}$, $i \neq j$ do not change by sparsity of $\gop_j$. An analogous argument shows the update to each $w\xdi_{s+1}$ preserves our invariant. Finally, in every iteration $s > 0$, the updates to $w\x_{s+1/2}$ and $w\x_{s + 1}$ only require evaluating one new gradient each, by $1$-sparsity of the dual block updates in the prior iteration.
\end{proof}

\subsection{Convergence analysis}\label{ssec:fsconv}

In this section, we prove a convergence result on Algorithm~\ref{alg:fsop} via an application of Proposition~\ref{prop:rmp}. To begin, we require a bound on the quantity $\lam$ in \eqref{eq:giprops}.

\begin{lemma}[Expected relative Lipschitzness]\label{lem:exrl}
Define $\{\gop_j\}_{j \in [n]}: \zset \to \zset^*$ as in \eqref{eq:gidef}, and define $r: \zset \to \R$ as in \eqref{eq:rfsdef}. Letting $w_+(j)$ be $w_{s + 1}$ in Algorithm~\ref{alg:rmp} if $j \in [n]$ was sampled in iteration $s$,
\[\E_{j \sim p}\Brack{\inprod{\gop_j(w\mdpt(j)) - \gop_j(w)}{w\mdpt(j) - w_+(j)}} \le \E_{j \sim p}\Brack{V^r_{w}\Par{w\mdpt(j)} + V^r_{w\mdpt(j)}\Par{w_+(j)} }\]
for
\begin{equation}\label{eq:lamdeffs}\lam = 2n + \frac{2\sum_{j \in [n]} \sqrt{L_j}}{\sqrt{n\mu}}.\end{equation}
\end{lemma}
\begin{proof}
We begin by expanding the expectation of the left-hand side:
\begin{flalign}
& \E_{j \sim p}\Brack{\inprod{\gop_j(w\mdpt(j)) - \gop_j(w)}{w\mdpt(j) - w_+(j)}} = \E_{j \sim p}\Brack{\inprod{\mu w\mdpt\x - \mu w\x}{w\mdpt\x - w_+\x(j)}} \nonumber\\
&\hspace{7em} + \E_{j \sim p}\Brack{\frac 1 {np_j} \inprod{\nabla f^*_j\Par{w\mdpt\xdj(j)} - \nabla f^*_j\Par{w\xdj}}{w\mdpt\xdj(j) - w_+\xdj(j)}} \nonumber\\
&\hspace{7em} + \E_{j \sim p}\Brack{\frac 1 {np_j} \inprod{w\mdpt\xdj(j) - w\xdj}{w\mdpt\x - w_+\x(j)}} \nonumber\\
&\hspace{7em} + \E_{j \sim p}\Brack{\frac 1 {np_j} \inprod{w\x - w\mdpt\x}{w\mdpt\xdj(j) - w_+\xdj(j)}} .\label{eq:4terms}
\end{flalign}
To bound the first two lines of \eqref{eq:4terms}, fix some $j \in [n]$. We apply Lemma~\ref{lem:rnablar} to the functions $\frac \mu 2 \norm{\cdot}^2$ and $\frac 1 n \nabla f^*_j$, and use nonnegativity of Bregman divergences, to conclude
\begin{gather*}\inprod{\mu w\mdpt\x - \mu w\x}{w\mdpt\x - w_+\x(j)} + \frac 1 {np_j} \inprod{\nabla f^*_j\Par{w\mdpt\xdj(j)} - \nabla f^*_j\Par{w\xdj}}{w\mdpt\xdj(j) - w_+\xdj(j)} \\ \le 2n\Par{V^r_{w}\Par{w\mdpt(j)}+ V^r_{w\mdpt(j)}\Par{w_+(j)}}.\end{gather*}
In particular, we used $\frac 1 {p_j} \le 2n$ by assumption, and noted we only need to handle the case where the second inner product term above is positive (in the other case, the above inequality is clearly true). Hence, taking expectations the first two lines in \eqref{eq:4terms} contribute $2n$ to $\lambda$ in the final bound.

To bound the last two lines of \eqref{eq:4terms}, fix $j \in [n]$. By applying \Cref{item:minimax_fsmooth_equiv} in Lemma~\ref{lem:smoothness_implications} to the pair $(\frac \mu 2 \norm{\cdot}^2, nf_i)$, we have
\begin{gather*}
\frac 1 {n} \inprod{w\mdpt\xdj(j) - w\xdj}{w\mdpt\x - w_+\x(j)} + \frac 1 {n} \inprod{w\x - w\mdpt\x}{w\mdpt\xdj(j) - w_+\xdj(j)} \\
\le \frac 1 n \sqrt{\frac{nL_j}{\mu}}\Par{\mu V_{w\mdpt\x}\Par{w_+\x(j)} + V^{f^*_j}_{w\xdj}\Par{w\mdpt\xdj(j)}} + \frac 1 n \sqrt{\frac{nL_j}{\mu}} \Par{\mu V_{w\x}\Par{w\mdpt\x} + V^{f^*_j}_{w\mdpt\xdj(j)}\Par{w_+\xdj(j)}} \\
= \sqrt{\frac{L_j}{n\mu}}\Par{V^r_w\Par{w\mdpt(j)} + V^r_{w\mdpt(j)}\Par{w_+(j)}}.
\end{gather*}

Using $\frac{\sqrt{L_i}}{p_i} \le 2\sum_{j \in [n]} \sqrt{L_j}$ and taking expectations over the above display,
\begin{gather*}
\E_{j \sim p}\Brack{\frac 1 {np_j} \inprod{w\mdpt\xdj(j) - w\xdj}{w\mdpt\x - w_+\x(j)} + \frac 1 {np_j} \inprod{w\x - w\mdpt\x}{w\mdpt\xdj(j) - w_+\xdj(j)} } \\
 \le \frac{2\sum_{j \in [n]} \sqrt{L_j}}{\sqrt{n\mu}} \E_{j \sim p}\Brack{V^r_w\Par{w\mdpt(j)} + V^r_{w\mdpt(j)}\Par{w_+(j)}} .
\end{gather*}
Hence, the last two lines in \eqref{eq:4terms} contribute $\frac{2\sum_{j \in [n]} \sqrt{L_j}}{\sqrt{n\mu}}$ to $\lam$ in the final bound.
\end{proof}

We next apply Proposition~\ref{prop:rmp} to analyze the convergence of Algorithm~\ref{alg:fsop}.

\begin{lemma}\label{lem:fsophalfdiv}
Let $w_0 \defeq (w\x_0, \{\nabla f_i(w\xpi_0)\}_{i \in [n]})$, which is the input $z_t$ to Algorithm~\ref{alg:fsop} at iteration $t$. If $S \ge 2\lam$ in Algorithm~\ref{alg:fsop} with $\lam$ as in \eqref{eq:lamdeffs}, then Algorithm~\ref{alg:fsop} returns $\widetilde{w} \gets \bw_{\sigma}$ as defined in~\eqref{eq:bzdef} such that for $z_\star$ as the saddle point to \eqref{eq:pdfs},
\[\E V^r_{\widetilde{w}}(z_\star) \le \half V^r_{w_0}(z_\star).\]
\end{lemma}
\begin{proof}
We apply Proposition~\ref{prop:rmp}, where \eqref{eq:giprops} is satisfied via Lemmas~\ref{lem:exregret} and~\ref{lem:exrl}. By Proposition~\ref{prop:rmp} with $u = z_\star $ and $S \ge 2 \lam$,
\[\E\Brack{\frac 1 S \sum_{0 \le s < S} \inprod{\gop(\bw_s)}{\bw_s - z_\star}} \le \half V^r_{w_0}(z_\star).\]
Moreover, since $\sigma$ is uniformly chosen in $[0, S - 1]$, we have
\[\E\Brack{\inprod{\gop(\bw_{\sigma})}{\bw_{\sigma} - z_\star}} \le \half V^r_{w_0}(z_\star).\]
Finally, Lemma~\ref{lem:itercostfs} shows that (an implicit representation of) $\bw_{\sigma}$ is indeed returned. We conclude by applying Lemma~\ref{lem:smfs} and using that $z_\star$ solves the VI in $\gop$, yielding
\[\E\Brack{\inprod{\gop(\bw_{\sigma})}{\bw_{\sigma} - z_\star}} \ge \E\Brack{\inprod{\gop(\bw_{\sigma}) - \gop(z_\star)}{\bw_{\sigma} - z_\star}} \ge V^r_{\bw_\sigma}(z_\star).\]
\end{proof}

Finally, we provide a simple bound regarding initialization of Algorithm~\ref{alg:fs}.

\begin{lemma}\label{lem:initboundfs}
Let $x_0 \in \xset$, and define
\begin{equation}\label{eq:z0deffs}z_0 \defeq \Par{x_0, \Brace{\nabla f_i(x_0)}_{i \in [n]}}.\end{equation}
Moreover, suppose that for $x_\star$ the solution to \eqref{eq:regfs}, $\Ffsreg(x_0)-\Ffsreg(x_\star) \le \Eps_0$. Then, letting $z_\star$ be the solution to \eqref{eq:pdfs}, we have
	\[V^r_{z_0}(z_\star) \le \Par{1 + \frac{\ssin L_i}{n\mu}} \Eps_0. \]
\end{lemma}
\begin{proof}
By the characterization in Lemma~\ref{lem:sameoptfs}, we have by Item 1 in Fact~\ref{fact:dualsc}:
\[z_\star = \Par{x_\star, \bin{\nabla f_i(x_\star)}}.\]
Hence, we bound analogously to Lemma~\ref{lem:initbound}:
\begin{align*}
V^r_{z_0}(z_\star) &\le \mu V_{x_0}(x_\star) + V^{\frac 1 n \sum_{i \in [n]} f_i}_{x_\star}(x_0) \\
&\le \mu V_{x_0}(x_\star) + \frac{\ssin L_i}{2n} \norm{x_0 - x_\star}^2 \\
&\le \Par{1 + \frac{\ssin L_i}{n\mu}} \mu V_{x_0}(x_\star) \le \Par{1 + \frac{\ssin L_i}{n\mu}} \Eps_0.\end{align*}
The last line applied strong convexity of $\Ffsreg$.
\end{proof}

\subsection{Main result}\label{ssec:fsres}

We now state and prove our main claim.

\begin{restatable}{theorem}{restatemainfs}\label{thm:mainfs}
	Suppose $\Ffsreg$ satisfies Assumption~\ref{assume:fs} and has minimizer $x_\star$, and suppose we have $x_0 \in \xset$ such that $ \Ffsreg(x_0)-\Ffsreg(x_\star) \le \Eps_0$. 
	Algorithm~\ref{alg:fs} using Algorithm~\ref{alg:fsop} with $\lam$ as in \eqref{eq:lamdeffs} returns $x \in \xset$ with $\E \Ffsreg(x) - \Ffsreg(x_\star) \le \eps$ in $N_{\textup{tot}}$ iterations, using a total of $O(N_{\textup{tot}})$ gradient calls each to some $f_i$ for $i \in [n]$,
	where
	\begin{equation}\label{eq:tkdef}N_{\textup{tot}}= O\Par{\kappa_{\textup{fs}} \log\Par{\frac{\kappa_{\textup{fs}} \Eps_0}{\eps}}},\text{ for } \kappa_{\textup{fs}} \defeq n + \frac{\sum_{i \in [n]} \sqrt{L_i}}{\sqrt{n\mu}}.\end{equation}
\end{restatable}
\begin{proof}
By Lemma~\ref{lem:sameoptfs}, the point $x_\star$ is consistent between \eqref{eq:origfs} and \eqref{eq:pdfs}. We run Algorithm~\ref{alg:fs} with
\[T =O\Par{\log\Par{\frac{\kappa_{\textup{fs}} \Eps_0}{\eps}}}.\]
By recursively applying Lemma~\ref{lem:fsophalfdiv} for $T$ times, we obtain a point $z$ such that 
\[\E V^r_z(z_\star) \le \frac{\eps \mu}{\mathcal{L}} \text{ for } \mathcal{L} = \mu + \nsin L_i,\]
and hence applying $\mathcal{L}$-smoothness of $\Ffsreg$ and optimality of $z_\star\x$ yields the claim. The complexity follows from Lemma~\ref{lem:itercost}, and spending $O(n)$ gradient evaluations on the first and last iterates of each call to Algorithm~\ref{alg:fsop} (which is subsumed by the fact that $S = \Omega(n)$).
\end{proof}

By applying a generic reduction we derive in Appendix~\ref{apdx:error}, we then obtain the following corollary.

\begin{restatable}{corollary}{restatecormainfsunreg}\label{cor:mainfsunreg}
	Suppose the summands $\{f_i\}_{i \in [n]}$ in \eqref{eq:origfs} satisfy Assumption~\ref{assume:fs}, and $\Ffs$ is $\mu$-strongly convex with minimizer $x_\star$. Further, suppose we have $x_0 \in \xset$ such that $\Ffs(x_0) - \Ffs(x_\star) \le \Eps_0.$
	Algorithm~\ref{alg:redx-outer} using Algorithm~\ref{alg:fs} to implement steps returns $x \in \xset$ with $\E \Ffs(x) - \Ffs(x_\star) \le \eps$ in $N_{\textup{tot}}$ iterations, using a total of $O(N_{\textup{tot}})$ gradient calls each to some $f_i$ for $i \in [n]$, where
	\[N_{\textup{tot}} = O\Par{\kappa_{\textup{fs}} \log\Par{\frac{\kappa_{\textup{fs}}\Eps_0}{\eps}}},\text{ for } \kappa_{\textup{fs}} \defeq n + \ssin \frac{\sqrt{L_i}}{\sqrt{n\mu}}.\]
\end{restatable}
 	%
%

\section{Minimax finite sum optimization}
\label{sec:mmfs}

In this section, we provide efficient algorithms for computing an approximate saddle point of the following minimax finite sum optimization problem:
\begin{equation}\label{eq:origmmfs-orig}
\min_{x\in\xset}\max_{y\in\yset}\Fmmfs(x, y)\defeq \nsin \Par{\xfun_i(x) + \xyfun_i(x,y) - \yfun_i(y)}.
\end{equation}
Here and throughout this section $\{f_i:\xset\rightarrow\R\}_{i\in[n]}$, $\{g_i:\yset\rightarrow\R\}_{i\in[n]}$ are differentiable convex functions, and $\{\xyfun_i:\xset\times\yset\rightarrow\R\}_{i\in[n]}$ are differentiable convex-concave functions. For the remainder, we focus on algorithms for solving the following regularized formulation of~\eqref{eq:origmmfs-orig}:
\begin{equation}\label{eq:origmmfs}
\min_{x \in \xset} \max_{y \in \yset} \Fmmfsreg(x, y)\defeq \nsin \Par{\xfun_i(x) + \xyfun_i(x,y) - \yfun_i(y)}+\frac{\mu\x}{2}\norm{x}^2-\frac{\mu\y}{2}\norm{y}^2.
\end{equation}

As in~\Cref{sec:minimax} and~\Cref{sec:finitesum}, to instead solve an instance of~\eqref{eq:origmmfs-orig} where each $f_i$ is $2\mu\x$-strongly convex and each $g_i$ is $2\mu\y$-strongly convex, we may instead equivalently solve \eqref{eq:origmmfs} by reparameterizing $f_i\gets f_i-\mu\x\norm{\cdot}^2$, $g_i\gets g_i-\mu\y\norm{\cdot}^2$ for each $i\in[n]$. The extra factor of $2$ is so we can make a strong convexity assumption in Assumption~\ref{assume:minimax-fs} about separable summands, which only affects our final bounds by constants. We further remark that our algorithms extend to solve instances of~\eqref{eq:origmmfs-orig} where $f$, $g$ is $\mu\x$ and $\mu\y$-strongly convex in $\norm{\cdot}$, but individual summands are not. We provide this result at the end of the section in Corollary~\ref{cor:mainmmfsunreg}.

In designing methods for solving~\eqref{eq:origmmfs} we make the following additional regularity assumptions.

\begin{assumption}\label{assume:minimax-fs}
We assume the following about \eqref{eq:origmmfs} for all $i \in [n]$.
	\begin{enumerate}
		\item $f_i$ is $L_i\x$-smooth and $\mu_i\x$-strongly convex and $g_i$ is $L_i\y$-smooth and $\mu_i\y$-strongly convex.
		\item $h_i$ has the following blockwise-smoothness properties: for all $u, v \in \xset \times \yset$,
		\begin{equation}\label{eq:Hlipbound}
			\begin{aligned}
				\norm{\nabla_x h_i(u) - \nabla_x h_i(v)} &\le \Lam_i\xx \norm{u\x - v\x} + \Lam_i\xy \norm{u\y - v\y}\text{ and }\\
				\norm{\nabla_y h_i(u) - \nabla_y h_i(v)} &\le \Lam_i\xy \norm{u\x - v\x} + \Lam_i\yy \norm{u\y - v\y}.
		\end{aligned}\end{equation}
	\end{enumerate}
\end{assumption}

The remainder of this section is organized as follows.

\begin{enumerate}
	\item In Section~\ref{ssec:mmfs-setup}, we state a primal-dual formulation of \eqref{eq:origmmfs} which we will apply our methods to, and prove that its solution also yields a solution to \eqref{eq:origmmfs}.
	\item In Section~\ref{ssec:mmfs-alg}, we give our algorithm, which is composed of an outer loop and an inner loop, and prove it is efficiently implementable.
	\item In Section~\ref{mmfs:convergence}, we prove the convergence rate of our inner loop.
	\item In Section~\ref{ssec:mmfsouter}, we prove the convergence rate of our outer loop.
	\item In Section~\ref{ssec:mmfs-main}, we state and prove our main result, Theorem~\ref{thm:mmfs}.
\end{enumerate}

\subsection{Setup}\label{ssec:mmfs-setup}

To solve \eqref{eq:origmmfs}, we will instead find a saddle point to the primal-dual function
\begin{equation}\label{eq:mmfspd}
\begin{aligned}
\Fmmfspd\Par{z} &\defeq \frac{\mu\x}{2}\norm{z\xsup}^2 - \frac{\mu\y}{2}\norm{z\ysup}^2 \\
&+ \nsin \Par{\xyfun_i(z\xsup, z\ysup) + \inprod{z\xdisup}{z\xsup} - \inprod{z\ydisup}{z\ysup} - f^*_i\Par{z\xdisup} + \yfun_i^*(z\ydisup)}.
\end{aligned}
\end{equation}
We denote the domain of $\Fmmfspd$ by $\zset \defeq \xset \times \yset \times (\xset^*)^n \times (\yset^*)^n$. For $z \in \zset$, we refer to its blocks by $(z\x, z\y, \{z\xdi\}_{i \in [n]}, \{z\ydi\}_{i \in [n]})$. The primal-dual function $\Fmmfspd$ is related to the original function $\Fmmfs$ in the following way; we omit the proof, as it follows analogously to the proofs of Lemmas~\ref{lem:sameopt} and~\ref{lem:sameoptfs}.

\begin{lemma}\label{lem:sameoptmmfs}
Let $z_\star = (\varfulls{z_\star})$ be the saddle point to \eqref{eq:mmfspd}. Then, $(z\xsup_\star, z\ysup_\star)$ is a saddle point to \eqref{eq:origmmfs}.
\end{lemma}

As in Section~\ref{ssec:mmsetup}, it will be convenient to define the convex function $r: \zset \to \R$, which combines the (unsigned) separable components of $\Fmmfspd$:
\begin{equation}\label{eq:rmmfsdef}
r\Par{\varfull{z}} \defeq \frac{\mu\x}{2} \norm{z\xsup}^2 + \frac{\mu\y}{2}\norm{z\ysup}^2 + \frac 1 n \sum_{i \in [n]}  \xfun^*_i\Par{z\xdisup} + \frac 1 n \sum_{i \in [n]} \yfun^*_i\Par{z\ydisup}.
\end{equation}

Again, $r$ serves as a regularizer in our algorithm. We next define $\goptot$, the gradient operator of $\Fmmfspd$. We decompose $\goptot$ into three parts, roughly corresponding to the contribution from $r$, the contributions from the primal-dual representations of $\bin{f_i}$ and $\bin{g_i}$, and the contribution from $\bin{h_i}$. In particular, we define
\begin{equation}\label{eq:gdefmmfs}
\begin{aligned}
\goptot(z) &\defeq \nabla r(z) + \gop^\xyfun(z) + \gcross(z), \\
\nabla r\Par{\varfull{z}} &\defeq \Par{\mu\x z\x, \mu\y z\y, \bin{\frac 1 n \nabla \xfun^*_i\Par{z\xdisup}}, \bin{\frac 1 n \nabla \yfun^*_i\Par{z\ydisup}}}, \\
\gop^\xyfun\Par{\varfull{z}} &\defeq \Par{\frac 1 n \sum_{i \in [n]} \nabla_x \xyfun_i(z\x, z\y), -\frac 1 n \sum_{i \in [n]} \nabla_y \xyfun_i(z\x, z\y), \bin{0}, \bin{0}}, \\
\gcross\Par{\varfull{z}} &\defeq \Par{\nsin z\xdisup, \nsin z\ydisup, \bin{-\frac 1 n z\xsup}, \bin{-\frac 1 n z\ysup}}.
\end{aligned}
\end{equation}

\subsection{Algorithm}\label{ssec:mmfs-alg}

In this section we present our algorithm which consists of the following two parts; its design is inspired by a similar strategy used in prior work \cite{CarmonJST19, CarmonJST20}.
\begin{enumerate}
	\item Our ``outer loop'' is based on a proximal point method (Algorithm~\ref{alg:mmfs-outer}, adapted from \cite{Nemirovski04}).
	\item Our ``inner loop'' solves each proximal subproblem to high accuracy via a careful analysis of randomized mirror prox (Algorithm~\ref{alg:mmfs}, adapted from Algorithm~\ref{alg:rmp}).
\end{enumerate} 
At each iteration $t$ of the outer loop (Algorithm~\ref{alg:mmfs-outer}), we require an accurate approximation
\begin{equation}\label{eq:mmfs-def-phi}
	z_{t + 1}\approx z^\star_{t+1}~~\text{which solves the VI in}~~\gop \defeq \goptot(z) + \gamma \Par{\nabla r(z)-\nabla r(z_{t})},
\end{equation}
where we recall the definitions of $\gtot$ and $r$ from~\eqref{eq:gdefmmfs} and~\eqref{eq:rmmfsdef}, and when $z_t$ is clear from context (i.e.\ we are analyzing a single implementation of the inner loop).

To implement our inner loop (i.e.\ solve the VI in $\gop$), we apply randomized mirror prox (Algorithm~\ref{alg:rmp}) with a new analysis. In particular, we will not be able to obtain the expected relative Lipschitzness bound required by Proposition~\ref{prop:rmp} for our randomized gradient estimators, so we develop a new ``partial variance'' analysis of Algorithm~\ref{alg:rmp} to obtain our rate. We use this terminology because we use variance bounds on a component of $\gop$ for which we cannot directly obtain expected relative Lipschitzness bounds. We prove Proposition~\ref{prop:newrmp} in Appendix~\ref{apdx:mfs-alg}.

\begin{restatable}[Partial variance analysis of randomized mirror prox]{proposition}{propnewrmp}\label{prop:newrmp}
Suppose (possibly random) $\goptilde$ is defined so that in each iteration $s$, for all $u \in \zset$ and all $\rho > 0$, there exists a (possibly random) point $\bw_{s} \in \zset$ and a $\gamma$-strongly monotone operator $\gop: \zset \to \zset^*$ (with respect to $r$) such that
\begin{equation}\label{eq:newassume}
\begin{aligned}
\E\Brack{\inprod{\goptilde(w_{s+1/2})}{w_{s+1/2} - w_\star}} &= \E\Brack{\inprod{\gop(\bw_{s})}{\bw_{s} - w_\star}}, \\
\E\Brack{\inprod{\goptilde(w_{s+1/2}) - \goptilde(w_{s})}{w_{s+1/2} - w_{s + 1}}} &\le \Par{\lam_0 + \frac 1 \rho}\E\Brack{V^r_{w_s}(w_{s+1/2}) + V^r_{w_{s+1/2}}(w_{s + 1})} \\
&+ \rho \lam_1\E\Brack{V^r_{w_0}(w_\star) + V^r_{\bw_s}(w_\star)},
\end{aligned}
\end{equation}
where $w_\star$ solves the VI in $\gop$. Then by setting 
\[\rho \gets \frac \gamma {5\lam_1},\; \lam \gets \lam_0 + \frac 1 \rho,\; T \gets \frac{5\lam}{\gamma} = \frac{5\lam_0}{\gamma} + \frac{25\lam_1}{\gamma^2}, \]
in Algorithm~\ref{alg:rmp}, and returning $\bw_\sigma$ for $0 \le \sigma < S$ sampled uniformly at random, 
\[\E\Brack{V^r_{\bar{w}_{\sigma}}\Par{w_\star}} \le \half V^r_{w_0}(w_\star).\]
\end{restatable}

For simplicity in the following we denote $\bz \defeq z_t$ whenever we discuss a single proximal subproblem. We next introduce the gradient estimator $\goptilde$ we use in each inner loop, i.e.\ finding a solution to the VI in $\gop$ defined in \eqref{eq:mmfs-def-phi}.  We first define three sampling distributions $p$, $q$, $r$, via
\begin{equation}\label{eq:pqrdef}
\begin{gathered}
p_j \defeq \frac{\sqrt{L\x_j}}{2\ssin \sqrt{L\x_i}} + \frac 1 {2n}\text{ for all } j \in [n],~~q_k \defeq \frac{\sqrt{L\y_k}}{2\ssin \sqrt{L\y_i}} + \frac 1 {2n} \text{ for all } k \in [n], \\
\text{and}~r_\ell \defeq \frac{\Ltot_\ell}{2\ssin \Ltot_i} + \frac 1 {2n} \text{ for all } \ell \in [n], \text{ where } \Ltot_i \defeq \frac{\Lam\xx_i}{\mu\x} + \frac{\Lam\xy_i}{\sqrt{\mu\x\mu\y}} + \frac{\Lam\yy_i}{\mu\y} \text{ for all } i \in [n].
\end{gathered}
\end{equation}

Algorithm~\ref{alg:mmfs} will run in logarithmically many phases, each initialized at an ``anchor point'' $w_0$ (cf.~\Cref{line:mmfs-inner-end}). We construct gradient estimators for Algorithm~\ref{alg:rmp} of $\gop(w) = \goptot(w)+\gamma(\nabla r(w)-\nabla r(\bz))$ as defined in~\eqref{eq:mmfs-def-phi} as follows. In each iteration, for a current anchor point $w_0$, we sample four coordinates $j \sim p$, $k \sim q$, and $\ell, \ell' \sim r$, all independently. We believe that it is likely that other sampling schemes, e.g.\ sampling $j$ and $k$ non-independently, will also suffice for our method but focus on the independent scheme for simplicity. We use $g\xy$ to refer to the $\xset\times \yset$ blocks of a vector $g$ in $\zset^*$, and $\dual$ to refer to all other blocks corresponding to $\left(\xset^*\right)^n\times \left(\yset^*\right)^n$. Then we define for an iterate $w = w_s$ of Algorithm~\ref{alg:mmfs} (where $\gop^\xyfun$ is as in \eqref{eq:gdefmmfs}):
\begin{equation}\label{eq:gjkldef}
\begin{aligned}
\goptilde(w) &\defeq \gopt_{jk\ell}(w) \defeq \gopt^\xyfun_{jk\ell}(w) + \gsept_{jk\ell}(w) + \gcrosst_{jk\ell}(w), \\
\Brack{\gopt^\xyfun_{jk\ell}(w)}\x &\defeq \Brack{\gopt^\xyfun(w_0)}\x + \frac 1 {nr_{\ell}} \Par{\nabla_x \xyfun_{\ell}(w\xsup, w\ysup) - \nabla_x \xyfun_{\ell}(w\xsup_0, w\ysup_0)}, \\
\Brack{\gopt^\xyfun_{jk\ell}(w)}\y&\defeq \Brack{\gopt^\xyfun(w_0)}\y -\frac 1 {nr_{\ell}} \Par{\nabla_y \xyfun_{\ell}(w\xsup, w\ysup) - \nabla_y \xyfun_{\ell}(w\xsup_0, w\ysup_0)}, \\
\Brack{\gopt^\xyfun_{jk\ell}(w)}\dual&\defeq \Par{\bin{0}, \bin{0}},\\
\Brack{\gsept_{jk\ell}(w)}\xy &\defeq \Par{1 + \gamma}\Par{\mu\x w\xsup, \mu\y w\ysup} - \gamma\Par{\mu\x \bz\xsup, \mu\y \bz\ysup}, \\
\Brack{\gsept_{jk\ell}(w)}\dual &\defeq (1 + \gamma)\Par{\bin{\frac 1 {n p_j} \nabla  \xfun^*_j\Par{w\xdssup{j}}  \cdot \1_{i = j}}, \bin{\frac 1 {n q_k} \nabla \yfun^*_k\Par{w\xdssup{k}}  \cdot \1_{i = k}}} \\
&- \gamma \Par{\bin{\frac 1 {n p_j} \nabla \xfun^*_j\Par{\bz\xdssup{j}}  \cdot \1_{i = j}}, \bin{\frac 1 {n q_k} \nabla \yfun^*_k\Par{\bz\ydssup{k}} \cdot \1_{i = k}}}, \\
\Brack{\gcrosst_{jk\ell}(w)}\xy &\defeq \Par{\nsin w\xdisup,\nsin w\ydisup}, \\
\Brack{\gcrosst_{jk\ell}(w)}\dual &\defeq \Par{\bin{-\frac 1 {n p_j} w\xsup\cdot \1_{i = j}}, \bin{-\frac 1 {n q_k}w\ysup \cdot \1_{i = k}}}.
\end{aligned}
\end{equation}
In particular, the estimator $\gopt_{jk\ell}(w)$ only depends on the sampled indices $j, k, \ell$, and not $\ell'$. Next, consider taking the step $w\mdpt(jk\ell) \gets \Prox_{w}^r(\frac 1 \lam g_{jk\ell}(w))$ as in Algorithm~\ref{alg:rmp}, where we use the shorthand $w\mdpt(jk\ell) = w_{s + 1/2}$ to indicate the iterate of Algorithm~\ref{alg:rmp} taken from $w_s$ assuming $j, k, \ell$ were sampled. Observing the form of $g_{jk\ell}$, we denote the blocks of $w\mdpt(jk\ell)$ by
\begin{align*}
w\mdpt(jk\ell) &\defeq \Par{w\xsup\mdpt(\ell), w\ysup\mdpt(\ell), \bin{w\xdssup{i}\mdpt(j)}, \bin{w\ydssup{i}\mdpt(k)}},
\end{align*}
where we write $w\xsup\mdpt(\ell)$ to indicate that it only depends on the random choice of $\ell$ (and not $j$ or $k$); we use similar notation for the other blocks. We also define
\[\Delta\x(j) \defeq w\xdssup{j}\mdpt(j) - w\xdssup{j}(j),\; \Delta\y(k) \defeq w\ydssup{k}\mdpt(k)- w\ydssup{k}(k),\]
and then set (where we use the notation $\gop_{jk\ell'}$ to signify its dependence on $j, k, \ell'$, and not $\ell$):
\begin{equation}\label{eq:gpjkldef}
\begin{aligned}
\goptilde(w\mdpt(jk\ell)) &\defeq \gopt_{jk\ell'}(w\mdpt(jk\ell)) \defeq \gopt^\xyfun_{jk\ell'}(w\mdpt(jk\ell)) + \gsept_{jk\ell'}(w\mdpt(jk\ell)) + \gcrosst_{jk\ell'}(w\mdpt(jk\ell)), \\
\Brack{\gopt^\xyfun_{jk\ell'}(w\mdpt(jk\ell))}\x &\defeq \Brack{\gopt^\xyfun(w_0)}\x + \frac 1 {nr_{\ell'}} \Par{\nabla_x \xyfun_{\ell'}(w\mdpt\xsup(\ell), w\mdpt\ysup(\ell)) - \nabla_x \xyfun_{\ell'}(w\xsup_0, w\ysup_0)}, \\
\Brack{\gopt^\xyfun_{jk\ell'}(w\mdpt(jk\ell))}\y&\defeq \Brack{\gopt^\xyfun(w_0)}\y -\frac 1 {nr_{\ell'}} \Par{\nabla_y \xyfun_{\ell'}(w\xsup\mdpt(\ell), w\ysup\mdpt(\ell)) - \nabla_y \xyfun_{\ell'}(w\xsup_0, w\ysup_0)}, \\
\Brack{\gopt^\xyfun_{jk\ell'}(w\mdpt(jk\ell))}\dual &\defeq \Par{\bin{0}, \bin{0}},\\
\Brack{\gsept_{jk\ell'}(w\mdpt(jk\ell))}\xy &\defeq \Par{1 + \gamma}\Par{\mu\x w\mdpt\xsup(\ell), \mu\y w\mdpt\ysup(\ell)} - \gamma\Par{\mu\x \bz\xsup, \mu\y \bz\ysup}, \\
\Brack{\gsept_{jk\ell'}(w\mdpt(jk\ell))}\dual &\defeq (1 + \gamma)\Par{\bin{\frac 1 {n p_j} \nabla  \xfun^*_j\Par{w\mdpt\xdssup{j}}  \cdot \1_{i = j}}, \bin{\frac 1 {n q_k} \nabla \yfun^*_k\Par{w\mdpt\xdssup{k}}  \cdot \1_{i = k}}} \\
&- \gamma \Par{\bin{\frac 1 {n p_j} \nabla \xfun^*_j\Par{\bz\xdssup{j}}  \cdot \1_{i = j}}, \bin{\frac 1 {n q_k} \nabla \yfun^*_k\Par{\bz\ydssup{k}} \cdot \1_{i = k}}}, \\
\Brack{\gcrosst_{jk\ell'}(w\mdpt(jk\ell))}\xy &\defeq \Par{\nsin w\xdisup + \frac 1 {np_j} \Delta\x(j),\nsin w\ydisup + \frac 1 {nq_k}\Delta\y(k)}, \\
\Brack{\gcrosst_{jk\ell'}(w\mdpt(jk\ell))}\dual &\defeq \Par{\bin{-\frac 1 {n p_j} w\xsup \mdpt(\ell) \cdot \1_{i = j}}, \bin{-\frac 1 {n q_k}w\ysup \mdpt(\ell) \cdot \1_{i = k}}}.
\end{aligned}
\end{equation}

We also define the random ``aggregate point'' we will use in Proposition~\ref{prop:newrmp}:
\begin{equation}\label{eq:bwdef}\bw(\ell) \defeq w + \Par{w\xsup_{\mdpt}(\ell) - w\xsup, w\ysup_{\mdpt}(\ell) - w\ysup, \{\Delta\x(j)\}_{j\in[n]}, \{\Delta\y(k)\}_{k\in[n]}}.\end{equation}
Notably, $\bw(\ell)$ depends only on the randomly sampled $\ell$. We record the following useful observation about our randomized operators \eqref{eq:gjkldef}, \eqref{eq:gpjkldef}, in accordance with the first condition in \eqref{eq:newassume}.

\begin{restatable}{lemma}{lemexregretmmfs}\label{lem:exregretmmfs}
	Define $\{\gopt_{jk\ell}, \gopt_{jk\ell'}\}: \zset \to \zset^*$ as in \eqref{eq:gjkldef}, \eqref{eq:gpjkldef}, and the random ``aggregate point'' $\bw(\ell)$ as in \eqref{eq:bwdef}. Then, for all $u \in \zset$,
	recalling the definition of $\gop = \goptot + \gamma (\nabla r - \nabla r(\bz))$ from \eqref{eq:mmfs-def-phi},
	\[\E\Brack{\inprod{\gopt_{jk\ell'}(w\mdpt(jk\ell))}{w\mdpt(jk\ell) - u}} = \E_{\ell \sim r}\Brack{\inprod{\gop(\bw(\ell))}{\bw(\ell) - u}}.\]
\end{restatable}

We prove Lemma~\ref{lem:exregretmmfs} in Appendix~\ref{apdx:mfs-alg}. Finally, we give a complete implementation of our method as pseudocode below in Algorithms~\ref{alg:mmfs-outer} (the outer loop) and~\ref{alg:mmfs} (the inner loop). We also show that it is a correct implementation in the following Lemma~\ref{lem:imp-mmfs}, which we prove in Appendix~\ref{apdx:mfs-alg}.

\begin{algorithm2e}
	\caption{$\mm$-$\fs$-$\textsc{Solve}(\Fmmfsreg, x_0, y_0)$: Minimax finite sum optimization}
     \label{alg:mmfs-outer}
     \DontPrintSemicolon
		\codeInput \eqref{eq:origmmfs} satisfying Assumption~\ref{assume:minimax-fs}, $(x_0, y_0) \in \xset \times \yset$\;
		\codeParameter $T \in \N$\;
		$z_0\x \gets x_0$, $z_0\y \gets y_0$, $z_0\xpi \gets x_0, z_0\xdi \gets \nabla f_i(x_0), z_0\ypi \gets y_0, z_0\ydi \gets \nabla g_i(y_0)$ for all $i \in [n]$\;\label{line:mmfs-initialization}
		\For{$0 \le t < T$}{
		 $z_{t+1}\gets \mm$-$\fs$-$\textsc{Inner}(\Fmmfsreg, \{z_t\x, z_t\y, \{z_t\xpi\}_{i \in [n]}, \{z_t\ypi\}_{i \in [n]}\})$\;\label{line:mmfs-z-update}
		}
		\codeReturn $(z_T\x, z_T\y)$
\end{algorithm2e}

\begin{algorithm2e}
	\caption{$\mm$-$\fs$-$\textsc{Inner}(\Fmmfsreg, \bz\x, \bz\y, \{\bz\xpi\}_{i \in [n]} , \{\bz\ypi\}_{i \in [n]} )$: Minimax finite sum optimization subroutine}
	\label{alg:mmfs}
	\DontPrintSemicolon
	\codeInput \eqref{eq:origmmfs} satisfying Assumption~\ref{assume:minimax-fs}, $\bz\x, \{\bz\xpi\}_{i \in [n]} \in \xset$, $\bz\y, \{\bz\ypi\}_{i \in [n]} \in \yset$ \;
	\codeParameter $\gamma \ge 1$, $\lambda > 0$, $N, S \in \N$\;
	$w_0\gets \bz$\;
			\For{$0\le \tau < N$}{
			Sample $0\le \sigma <S$ uniformly at random\;
		\For{$0 \le s \le \sigma$}{\label{line:mmfs-inner-start}
		Sample $j, k, \ell, \ell' \in [n]$ independently according to $p, q, r, r$ respectively defined in \eqref{eq:pqrdef}, and define
		\begin{flalign*}
		\left[\gsept\right]\x &\defeq \Par{1 + \gamma}\mu\x w_s\xsup - \gamma\mu\x \bz\xsup, ~~~\left[\gsept\right]\y =\Par{1 + \gamma}\mu\y w_s\ysup - \gamma\mu\y \bz\ysup,\\
		\left[\gcrosst\right]\x &\defeq \frac {\sum_{i\in[n]}\nabla f_i\Par{w_s\xpisup}} {n}, ~~~\left[\gcrosst\right]\y \defeq \frac {\sum_{i\in[n]}\nabla g_i\Par{w_s\ypisup}} {n},\\ 
	\gop\x &\defeq \Brack{\gop^\xyfun(w_0)}\x + \frac {\nabla_x \xyfun_{\ell}(w_s\xsup, w_s\ysup) - \nabla_x \xyfun_{\ell}(w\xsup_0, w\ysup_0)}{nr_{\ell}} +\Brack{\gsept}\x +\Brack{\gcrosst}\x,\\
	\gop\y &\defeq \Brack{\gop^\xyfun(w_0)}\y - \frac{\nabla_y \xyfun_{\ell}(w_s\xsup, w_s\ysup) - \nabla_y \xyfun_{\ell}(w\xsup_0, w\ysup_0)}{nr_{\ell}}  +\Brack{\gsept}\y +\Brack{\gcrosst}\y
\end{flalign*}\\
$w\xsup_{s+1/2}\gets w\xsup_{s}- \frac{1}{\lambda\mu\x}\gop\x$, $w\ysup_{s+1/2}\gets  w\ysup_{s}- \frac{1}{\lambda\mu\y}\gop\y$\;\label{line:mmfs-grad-primal}
$w\xpssup{j}_{s+1/2}\gets w\xpssup{j}_s-\frac{1}{n\lambda p_j}\Par{\Par{1+\gamma}w\xpssup{j}_s-\gamma \bz\xpssup{j}-w\xsup_s}$\label{line:mmfs-grad-dual-x}\;
 $w\ypssup{k}_{s+1/2}\gets w\xpssup{k}_s-\frac{1}{n\lambda p_k}\Par{\Par{1+\gamma}w\ypssup{k}_s-\gamma \bz\ypssup{k}-w\ysup_s}$\label{line:mmfs-grad-dual-y}\;
Define
		\begin{flalign*}
		\left[\gsept\right]\x &\defeq \Par{1 + \gamma}\mu\x w_{s+1/2}\xsup - \gamma\mu\x \bz\xsup, ~~\left[\gsept\right]\y \defeq\Par{1 + \gamma}\mu\y w_{s+1/2}\ysup - \gamma\mu\y \bz\ysup,\\
		\left[\gcrosst\right]\x &\defeq \frac {\sum_{i\in[n]}\nabla f_i\Par{w_s\xpisup}} {n}+\frac{\nabla f_j\Par{w_{s+1/2}\xpssup{j}}-\nabla f_j\Par{w_{s}\xpssup{j}}}{np_j},\\
		\left[\gcrosst\right]\y &\defeq \frac {\sum_{i\in[n]}\nabla g_i\Par{w_s\ypisup}} {n}+\frac{\nabla g_k\Par{w_{s+1/2}\ypssup{k}}-\nabla g_k\Par{w_{s}\ypssup{k}}}{nq_k},\\ 
	\gop\x &\defeq \Brack{\gop^\xyfun(w_0)}\x + \frac {\nabla_x \xyfun_{\ell'}(w_{s+1/2}\xsup, w_{s+1/2}\ysup) - \nabla_x \xyfun_{\ell'}(w\xsup_0, w\ysup_0)}{nr_{\ell'}} +\Brack{\gsept}\x +\Brack{\gcrosst}\x,\\
	\gop\y &\defeq \Brack{\gop^\xyfun(w_0)}\y - \frac{\nabla_y \xyfun_{\ell'}(w_{s+1/2}\xsup, w_{s+1/2}\ysup) - \nabla_y \xyfun_{\ell'}(w\xsup_0, w\ysup_0)}{nr_{\ell'}}  +\Brack{\gsept}\y +\Brack{\gcrosst}\y
\end{flalign*}\\
$w\xsup_{s+1}\gets w\xsup_{s}- \frac{1}{\lambda\mu\x}\gop\x$, $w\ysup_{s+1}\gets  w\ysup_{s}- \frac{1}{\lambda\mu\y}\gop\y$\;\label{line:mmfs-extragrad-primal}
 $w\xpssup{j}_{s+1}\gets w\xpssup{j}_{s}-\frac{1}{n\lambda p_j}\Par{\Par{1+\gamma}w\xpssup{j}_{s+1/2}-\gamma \bz\xpssup{j}-w\xsup_{s+1/2}}$\label{line:mmfs-extragrad-dual-x}\;
 $w\ypssup{k}_{s+1}\gets w\xpssup{k}_{s}-\frac{1}{n\lambda p_k}\Par{\Par{1+\gamma}w\ypssup{k}_{s+1/2}-\gamma \bz\ypssup{k}-w\ysup_{s+1/2}}$\label{line:mmfs-extragrad-dual-y}\;
		}

		$\bar{w}\xpssup{i}\gets w\xpssup{i}_{\sigma}-\frac{1}{n\lambda p_i}\Par{\Par{1+\gamma}w\xpssup{i}_{\sigma}-\gamma \bz\xpssup{i}-w\xsup_{\sigma}}$ for each $i \in[n]$\;\label{line:mmfs-return-start}
		 $\bar{w}\ypssup{i}\gets w\ypssup{i}_{\sigma}-\frac{1}{n\lambda q_i}\Par{\Par{1+\gamma}w\ypssup{i}_{\sigma}-\gamma \bz\ypssup{i}-w\ysup_{\sigma}}$ for each $i \in[n]$\;
		 $w_{0}\xy \gets w_{\sigma+1/2}\xy$, $w_0\xpi \gets \bw\xpi$, $w_0\ypi \gets \bw\ypi$ for all $i \in [n]$\label{line:mmfs-inner-end}\;
		 }
	 \codeReturn $(w_0\x, w_0\y, \{\nabla f_i(w_0\xpi)\}_{i \in [n]}, \{\nabla g_i(w_0\ypi)\}_{i \in [n]})$
\end{algorithm2e}

\begin{restatable}{lemma}{lemimpmmfs}\label{lem:imp-mmfs}
Lines~\ref{line:mmfs-inner-start} to~\ref{line:mmfs-inner-end} of Algorithm~\ref{alg:mmfs} implement Algorithm~\ref{alg:rmp} on $(\{\goptilde\}, r)$ defined in \eqref{eq:gjkldef}, \eqref{eq:gpjkldef}, \eqref{eq:rmmfsdef}, for $\sigma$ iterations, and returns $\bw_\sigma$, following the definition \eqref{eq:bwdef}. Each iteration $s > 0$ is implementable in $O(1)$ gradient calls to some $\{f_j, g_k, h_l\}$, and $O(1)$ vector operations on $\xset$ and $\yset$.
\end{restatable}

\subsection{Inner loop convergence analysis}\label{mmfs:convergence}

We give a convergence guarantee on Algorithm~\ref{alg:mmfs} for solving the VI in $\gop \defeq \gtot + \gamma (\nabla r - \nabla r(\bz))$. In order to use~\Cref{prop:newrmp} to solve our problem, we must prove strong monotonicity of $\gop$ and specify the parameters $\lambda_0$, $\lam_1$ and $\rho$ in \eqref{eq:newassume}; note that Lemma~\ref{lem:exregretmmfs} handles the first condition in \eqref{eq:newassume}. To this end we give the following properties on $\gop$, $\goptilde$ as defined in~\eqref{eq:gjkldef} and~\eqref{eq:gpjkldef}; proofs of Lemmas~\ref{lem:lam0bound} and~\ref{lem:lam1bound} are deferred to Appendix~\ref{apdx:mmfsinner}.

\paragraph{Strong monotonicity.} We begin by proving strong monotonicity of $\gop$.

\begin{lemma}[Strong monotonicity]\label{lem:mmfs-sm}
Define $\gop: \zset \to \zset^*$ as in \eqref{eq:mmfs-def-phi}, and define $r: \zset \to \R$ as in \eqref{eq:rmmfsdef}. Then $\gop$ is $(1 +\gamma)$-strongly monotone with respect to $r$.
\end{lemma}
\begin{proof}
We decompose $\gop(z) = (1+\gamma)\nabla r(z) + \gcross(z) + \gop^\xyfun(z) - \gamma \nabla r(\bz)$, using the definitions in \eqref{eq:gdefmmfs}. By a similar argument as~\Cref{lem:sm}, we obtain the claim.
\end{proof}

\paragraph{Expected relative Lipschitzness.} We next provide bounds on the components of \eqref{eq:newassume} corresponding to $\gsep$ and $\gcross$, where we use the shorthand $\gsep \defeq (1 + \gamma) \nabla r - \gamma \nabla r(\bz)$ in the remainder of this section. In particular, we provide a partial bound on the quantity $\lam_0$.

\begin{restatable}{lemma}{lamzerobound}\label{lem:lam0bound}
Define $\{\gopt_{jk\ell}, \gopt_{jk\ell'}\}: \zset \to \zset^*$ as in \eqref{eq:gjkldef}, \eqref{eq:gpjkldef}, and define $r: \zset \to \R$ as in \eqref{eq:rmmfsdef}. Letting $w_{+}(jk\ell\ell')$ be $w_{s + 1}$ in Algorithm~\ref{alg:mmfs} if $j, k, \ell, \ell'$ were sampled in iteration $s$, defining
\begin{align*}
\gopt\Hno_{jk\ell}(w) &\defeq \gsept_{jk\ell}(w) + \gcrosst_{jk\ell}(w), \\
\gopt\Hno_{jk\ell'}(w\mdpt(jk\ell)) &\defeq \gsept_{jk\ell'}(w\mdpt(jk\ell)) + \gcrosst_{jk\ell'}(w\mdpt(jk\ell)),
\end{align*}
we have
\[\E\Brack{\inprod{\gopt\Hno_{jk\ell'}(w\mdpt(jk\ell)) -\gopt\Hno_{jk\ell}(w)}{w\mdpt(jk\ell) - w_{+}(jk\ell\ell')}} \le \lam\Hno \E\Brack{V^r_{w}\Par{w\mdpt(jk\ell)} + V^r_{w\mdpt(jk\ell)}\Par{w_{+}(jk\ell\ell')}},\]
for
\[\lam\Hno = 2n(1 + \gamma) + \frac{\ssin \sqrt{L\x_i}}{\sqrt{n\mu\x}} + \frac{\ssin \sqrt{L\y_i}}{\sqrt{n\mu\y}}.\]
\end{restatable}

\paragraph{Partial variance bound.} Finally, we provide bounds on the components of \eqref{eq:newassume} corresponding to $\gop^\xyfun$. Namely, we bound the quantity $\lam_1$, and complete the bound on $\lam_0$ within~\Cref{prop:newrmp}.

\begin{restatable}{lemma}{lamonebound}\label{lem:lam1bound}
Following notation of Lemma~\ref{lem:lam0bound}, and recalling the definition \eqref{eq:lamhdef}, for
\[\lam_1 \defeq 32(\lam^\xyfun)^2,\]
where we define
\begin{equation}\label{eq:lamhdef}\lam^\xyfun \defeq \nsin\Par{\frac{\Lam\xx_i}{\mu\x} + \frac{\Lam_i\xy}{\sqrt{\mu\x\mu\y}} + \frac{\Lam\yy_i}{\mu\y}}.\end{equation}
we have for any $\rho > 0$,
\begin{flalign}
& \E\Brack{\inprod{\gopt^\xyfun_{jk\ell'}(w\mdpt(jk\ell)) - \gopt^\xyfun_{jk\ell}(w)}{w\mdpt(jk\ell) - w_{+}(jk\ell\ell')}}\nonumber\\
& \hspace{1em} \le \Par{2\lam^\xyfun + \frac 1 \rho}\E\Brack{V^r_{w}\Par{w\mdpt(jk\ell)} + V^r_{w\mdpt(jk\ell)}\Par{w_{+}(jk\ell\ell')}} + \rho\lam_1\E\Brack{V^r_{w_0}(w^\star) + V^r_{\bw(\ell)}(w_\star)}.\label{eq:ghbound}
\end{flalign}
\end{restatable}

Combining the properties we prove above with~\Cref{prop:newrmp}, we obtain the following convergence guarantee for each loop $0 \le \tau < N$ of Lines~\ref{line:mmfs-inner-start} to~\ref{line:mmfs-inner-end} in~\Cref{alg:mmfs}.

\begin{proposition}\label{prop:onephase}
Consider a run of Lines~\ref{line:mmfs-inner-start} to~\ref{line:mmfs-inner-end} in~\Cref{alg:mmfs} initialized at $w_0 \in \zset$, with
\begin{equation}\label{eq:lamSdef}
\begin{aligned}\lam \gets \Par{2n(1 + \gamma) + \frac{2\ssin \sqrt{L\x_i}}{\sqrt{n\mu\x}} + \frac{2\ssin \sqrt{L\y_i}}{\sqrt{n\mu\y}} + 2\lam^\xyfun} + \frac{160(\lam^\xyfun)^2}{\gamma},\; S \gets \frac{5\lam}{\gamma},
\end{aligned}
\end{equation}
where $\lam^\xyfun$ is defined in \eqref{eq:lamhdef}. Letting $\widetilde{w}$ be the new setting of $w_0$ in Line~\ref{line:mmfs-inner-end} at the end of the run,
\[\E\Brack{V^r_{\widetilde{w}}(w^\star)} \le \half V^r_{w_0}(w^\star),\]
where $w^\star$ solves the VI in $\gop$ (defined in \eqref{eq:mmfs-def-phi}).
\end{proposition}

\subsection{Outer loop convergence analysis}\label{ssec:mmfsouter}

We state the following convergence guarantee on our outer loop, Algorithm~\ref{alg:mmfs-outer}. The analysis is a somewhat technical modification of the standard proximal point analysis for solving VIs \cite{Nemirovski04}, to handle approximation error. As a result, we state the claim here and defer a proof to Appendix~\ref{apdx:mmfsouter}.

\begin{restatable}{proposition}{propmultiphase}\label{prop:multiphase}
Consider a single iteration $0 \le t < T$ of Algorithm~\ref{alg:mmfs-outer}, and let $z_\star$ is the saddle point to $\Fmmfspd$ (defined in \eqref{eq:mmfspd}). Setting $S$ as in \eqref{eq:lamSdef} and
\begin{equation}\label{def:kappa-max}
N \defeq O\Par{\log\Par{\gamma \lam}},
\end{equation}
for an appropriately large constant in our implementation of Algorithm~\ref{alg:mmfs} and $\lam$ as in \eqref{eq:lamSdef}, we have
\[\E V^r_{z_{t + 1}}(z_\star) \le \frac{4\gamma}{1 + 4\gamma} V^r_{z_t}(z_\star).\]
\end{restatable}

\subsection{Main result}\label{ssec:mmfs-main}

We now state and prove our main claim.

\begin{theorem}\label{thm:mmfs}
Suppose $\Fmmfs$ in~\eqref{eq:origmmfs} satisfies Assumption~\ref{assume:minimax-fs}, and has saddle point $(x_\star, y_\star)$. Further, suppose we have $(x_0, y_0) \in \xset \times \yset$ such that $\gap_{\Fmmfsreg}(x_0, y_0)\le \Eps_0$.
Algorithm~\ref{alg:mmfs-outer} using Algorithm~\ref{alg:mmfs} with $\lam$ as in \eqref{eq:lamSdef} returns $(x, y) \in \xset \times \yset$ with $\E \gap_{\Fmmfsreg}(x, y) \le \eps$ in $N_{\textup{tot}}$ iterations, using a total of $O(N_{\textup{tot}})$ gradient calls each to some $f_i$, $g_i$, or $h_i$ for $i \in [n]$, where
\begin{equation}\label{eq:tnsdef}
\begin{gathered}
N_{\textup{tot}} = O\Par{\kappa_{\textup{mmfs}}\log\Par{\kappa_{\textup{mmfs}}}\log \Par{\frac{\kappa_{\textup{mmfs}} \Eps_0}{\eps}}},\\
\text{for } \kappa_{\textup{mmfs}} \defeq n + \frac{1}{\sqrt{n}} \sum_{i \in [n]} \Par{\sqrt{\frac{L\x_i}{\mu\x}} + \sqrt{\frac{L\y_i}{\mu\y}} + \frac{\Lam\xx_i}{\mu\x} + \frac{\Lam\xy_i}{\sqrt{\mu\x\mu\y}} + \frac{\Lam\yy_i}{\mu\y}}. 
\end{gathered}
\end{equation}
In particular, we use $N_{\textup{tot}} = NTS$ for
\[T = O\Par{\gamma \log\Par{\frac{\kappa_{\textup{fs}} \Eps_0}{\eps}}},\; N = O\Par{\log\Par{\kappa_{\textup{mmfs}}}},\; S = O\Par{n + \frac{\kappa_{\textup{mmfs}}}{\gamma}+ \frac{(\lam^h)^2}{\gamma^2}},\; \gamma = \frac{\lam^h}{\sqrt{n}}.\]
\end{theorem}

\begin{proof} By Lemma~\ref{lem:sameoptmmfs}, the point $(x_\star, y_\star)$ is consistent between \eqref{eq:origmmfs} and \eqref{eq:mmfspd}. The complexity of each iteration follows from observation of~\Cref{alg:mmfs-outer} and~\ref{alg:mmfs}.

Next, by~\Cref{prop:onephase} and~\Cref{prop:multiphase}, and our choices of $T$, $N$, and $S$ for appropriately large constants, we obtain a point $(x, y) \in \xset \times \yset$ such that
\[\E V^r_{(x, y)}(z_\star) \le \frac \eps 4 \Par{\frac 1 {\kappa_{\textup{mmfs}}}}^2.\]
Here we used an analogous argument to \Cref{lem:initbound} to bound the initial divergence. We then use a similar bound as in \Cref{lem:termbound} 
to obtain the desired duality gap bound.
\end{proof}

We now revisit the problem \eqref{eq:origmmfs-orig}. We apply a generic reduction framework we develop in Appendix~\ref{apdx:error} to develop a solver for this problem under a relaxed version of~\Cref{assume:minimax-fs}, without requiring strong convexity of individual summands.

\begin{assumption}\label{assume:minimax-fs-gen}
	We assume the following about \eqref{eq:origmmfs-orig} for all $i \in [n]$.
	\begin{enumerate}
		\item $f_i$ is $L_i\x$-smooth, and $g_i$ is $L_i\y$-smooth.
		\item $h$ has the following blockwise-smoothness properties: for all $u, v \in \xset \times \yset$,
		\begin{equation}\label{eq:Hlipbound-gen}
			\begin{aligned}
				\norm{\nabla_x h_i(u) - \nabla_x h_i(v)} &\le \Lam_i\xx \norm{u\x - v\x} + \Lam_i\xy \norm{u\y - v\y},\\
				\norm{\nabla_y h_i(u) - \nabla_y h_i(v)} &\le \Lam_i\xy \norm{u\x - v\x} + \Lam_i\yy \norm{u\y - v\y}.
		\end{aligned}\end{equation}
	\end{enumerate}
\end{assumption}

For minimax finite sum optimization problems with this set of relaxed conditions, we conclude with the following corollary of Theorem~\ref{thm:mmfs}.

\begin{restatable}{corollary}{restatecormainmmfsunreg}\label{cor:mainmmfsunreg}
	Suppose the summands $\{f_i, g_i, h_i\}_{i \in [n]}$ in \eqref{eq:origmmfs-orig} satisfy Assumption~\ref{assume:minimax-fs-gen}, and $\Fmmfs$ is $\mu\x$-strongly convex in $x$, $\mu\y$-strongly convex in $y$, with saddle point $(x_\star,y_\star)$. Further, suppose we have $(x_0, y_0) \in \xset\times\yset$ such that $\gap_{\Fmmfs}(x_0, y_0) \le \Eps_0$.	Algorithm~\ref{alg:redx-outer} using Algorithm~\ref{alg:mmfs-outer} and~\ref{alg:mmfs} to implement steps returns $(x,y) \in \xset\times\yset$ with $\E\gap(x, y) \le \eps$ in $N_{\textup{tot}}$ iterations, using a total of $O(N_{\textup{tot}})$ gradient calls each to some $f_i$, $g_i$, or $h_i$ for $i \in [n]$, where
	\begin{gather*}
		N_{\textup{tot}} = O\Par{\kappa_{\textup{mmfs}}\log(\kappa_{\textup{mmfs}})\log\Par{\frac{\kappa_{\textup{mmfs}} \Eps_0} \eps}},\\
		\text{for } \kappa_{\textup{mmfs}} \defeq n + \frac{1}{\sqrt{n}} \sum_{i \in [n]} \Par{\sqrt{\frac{L\x_i}{\mu\x}} + \sqrt{\frac{L\y_i}{\mu\y}} + \frac{\Lam\xx_i}{\mu\x} + \frac{\Lam\xy_i}{\sqrt{\mu\x\mu\y}} + \frac{\Lam\yy_i}{\mu\y}}.
	\end{gather*}
\end{restatable}

	\subsection*{Acknowledgments}
	We thank Yair Carmon, Arun Jambulapati and Guanghui Lan for helpful conversations. YJ was supported by a Stanford Graduate Fellowship. AS was supported in part by a Microsoft Research Faculty Fellowship, NSF CAREER Award CCF-1844855, NSF Grant CCF-1955039, a PayPal research award, and a Sloan Research Fellowship. KT was supported in part by a Google Ph.D. Fellowship, a Simons-Berkeley VMware Research Fellowship, a Microsoft Research Faculty Fellowship, NSF CAREER Award CCF-1844855, NSF Grant CCF-1955039, and a PayPal research award.

	\bibliographystyle{alpha}	
\newcommand{\etalchar}[1]{$^{#1}$}

	\newpage
	\begin{appendix}
\section{Reducing strongly monotone problems to regularized subproblems}\label{apdx:error}

\subsection{Convex optimization}\label{ssec:error_convex}

We give the following generic reduction for strongly convex optimization in the form of an algorithm. Similar reductions are standard in the literature~\cite{FrostigGKS15}, but we include the algorithm and full analysis here for completeness.

\begin{algorithm2e}
	\caption{$\textsc{Redx-Convex}$: Strongly convex optimization reduction}
	\label{alg:redx-outer}
	\DontPrintSemicolon
		\codeInput $\mu$-strongly convex $f:\xset\rightarrow \R$,  $x_0 \in \xset$\;
		\codeParameter $K \in \N$\;
		\For{$0 \le k < K$}{
		$x_{k+1} \gets $ any (possibly random) point satisfying
		\[\E V_{x_{k + 1}}(x^\star_{k + 1}) \le \frac 1 4 V_{x_k}(x^\star_{k + 1}), \text{ where } x^\star_{k + 1} \defeq \argmin_{x \in \xset} f(x) + \frac \mu 4 V_{x_k}(x) \]
		}
\end{algorithm2e}

\begin{lemma}\label{redx:convex}
In Algorithm~\ref{alg:redx-outer}, letting $x_\star$ minimize $f$, we have for every $k \in [K]$:
	\[
	\E V_{x_k}\Par{x_\star}\le \frac{1}{2^k}V_{x_0}\Par{x_\star}.
	\]
\end{lemma}
\begin{proof}
By applying the optimality condition on $x^\star_{k + 1}$, strong convexity of $f$, and \eqref{eq:threept},
\begin{align*}
\inprod{\nabla f(x^\star_{k+1})} {x^\star_{k+1}-x_\star} &\le \frac \mu 4 \inprod{x_k - x^\star_{k + 1}}{x^\star_{k + 1} - x_\star}\\
\implies \mu V_{x^\star_{k+1}}(x_\star) &\le f(x^\star_{k+1})-f(x_\star) \\
&\le \inprod{\nabla f(x^\star_{k+1})}{x^\star_{k+1}-x_\star} \\
&\le \frac{\mu}{4} V_{x_{k}}(x_\star)-  \frac{\mu}{4}  V_{x^\star_{k+1}}(x_\star) -  \frac{\mu}{4}  V_{x_{k}}(x^\star_{k+1}).
\end{align*}
Further by the triangle inequality and $(a + b)^2 \le 2a^2 + 2b^2$, we have
\[V_{x_{k+1}}(x_\star) \le 2V_{x_{k + 1}}(x_{k + 1}^\star) + 2V_{x_{k + 1}^\star}(x_\star).\]
Hence, combining these pieces,
\begin{align*}
	\E V_{x_{k+1}}(x_\star) & \le 2V_{x^\star_{k+1}}(x_\star) + 2\E V_{x_{k+1}}(x^\star_{k+1}) \\
	&\le 2V_{x^\star_{k+1}}(x_\star) + \half V_{x_k}(x_{k + 1}^\star) \\
	&\le  \half V_{x_{k}}(x_\star)- \half  V_{x^\star_{k+1}}(x_\star) \le \half V_{x_k}(x_\star).
\end{align*}

\end{proof}

We apply this reduction in order to prove Corollary~\ref{cor:mainfsunreg}. 

\restatecormainfsunreg*
\begin{proof}
The overhead $K$ is asymptotically the same here as the parameter $T$ in Theorem~\ref{thm:mainfs}, by analogous smoothness and strong convexity arguments. Moreover, we use Theorem~\ref{thm:mainfs} to solve each subproblem required by Algorithm~\ref{alg:redx-outer}; in particular, the subproblem is equivalent to approximately minimizing $\Ffs + \frac \mu 8 \norm{\cdot}^2$, up to a linear shift which does not affect any smoothness bounds, and a constant in the strong convexity. We note that we will initialize the subproblem solver in iteration $k$ with $x_k$. We hence can set $T = 2$ and $S = O(\kappa_{\textup{fs}})$, yielding the desired iteration bound.
\end{proof}

\subsection{Convex-concave optimization}\label{ssec:error_minimax}

We give the following generic reduction for strongly convex-concave optimization in the form of an algorithm. For simplicity in this section, we define for $z = (z\x, z\y) \in \xset \times \yset$,
\[\omega(z) \defeq \frac{\mu\x} 2 \norm{z\x}^2 + \frac{\mu\y} 2 \norm{z\y}^2.\]

\begin{algorithm2e}
	\caption{$\textsc{Redx-Minimax}$: Reduction for minimax}
	\label{alg:redx-outer-mm}
	\DontPrintSemicolon
		\codeInput $F: \xset \times \yset \to \R$ such that $F(\cdot, y)$ is $\mu\x$-strongly convex for all $y \in \yset$ and $F(x, \cdot)$ $\mu\y$-strongly concave for all $x \in \xset$, $z_0 \in \xset \times \yset$ \;
		\codeParameter $K \in \N$\;
		\For{$0 \le k < K$}{
		$z_{k + 1} \gets$ any (possibly random) point satisfying
		\begin{gather*}
		\E\Brack{ V^\omega_{z_{k + 1}}\Par{z_{k + 1}^\star} }\le \frac{1}{4}\Par{V^\omega_{z_k}\Par{z_{k + 1}^\star}},\\
		 \text{where } z_{k + 1}^\star \defeq \argmin_{z\x \in \xset} \argmax_{z\y \in \yset} F(z\x,z\y) + \frac{\mu\x} 4 V_{z_k\x}\Par{z\x} - \frac{\mu\y} 4 V_{z_k\y}\Par{z\y}
		\end{gather*}
		}
\end{algorithm2e}

\begin{lemma}\label{lem:redx-minimax}
In Algorithm~\ref{alg:redx-outer-mm}, letting $(x_\star, y_\star)$ be the saddle point of $F$, we have for every $k \in [K]$:
	\[
	\E\Brack{ V^\omega_{z_k}(z_\star)} \le \frac{1}{2^k} V^\omega_{z_0}(z_\star) .
	\]
\end{lemma}

\begin{proof}
By applying the optimality conditions on $z^\star_{k + 1}$, strong convexity-concavity of $F$, and \eqref{eq:threept}, and letting $\gop^F$ be the gradient operator of $F$,
\begin{flalign*}
\inprod{\gop^F(z^\star_{k + 1})}{z^\star_{k + 1} - z_\star} &\le \frac {\mu\x} 4 \inprod{z\x_k - [z^\star_{k + 1}]\x}{[z^\star_{k + 1}]\x - z_\star\x} \\
&+ \frac{\mu\y} 4 \inprod{z\y_k - [z^\star_{k + 1}]\y}{[z^\star_{k + 1}]\y - z_\star\y} \\
\implies V^\omega_{z^\star_{k + 1}}(z_\star) &\le \inprod{\gop^F(z^\star_{k + 1})}{z^\star_{k + 1} - z_\star} \\
&\le \frac 1 4 V^\omega_{z_k}(z_\star) - \frac 1 4 V^\omega_{z^\star_{k + 1}}(z_\star) - \frac 1 4 V_{z_k}(z_{k + 1}^\star).
\end{flalign*}
Further by the triangle inequality and $(a + b)^2 \le 2a^2 + 2b^2$, we have
\[V^\omega_{z_{k + 1}}(z_\star) \le 2 V^\omega_{z_{k + 1}}(z_{k + 1}^\star) + 2 V^\omega_{z_{k + 1}^\star}(z_\star).\]
Hence, combining these pieces,
\begin{align*}
\E V^\omega_{z_{k + 1}}(z_\star) &\le 2 V^\omega_{z_{k + 1}^\star}(z_\star) + 2 \E V^\omega_{z_{k + 1}}(z_{k + 1}^\star) \\
&\le 2 V^\omega_{z_{k + 1}^\star}(z_\star) + \frac 1 2 V^\omega_{z_k}(z_{k + 1}^\star) \\
&\le \half V^\omega_{z_k}(z_\star) - \half V^\omega_{z^\star_{k + 1}}(z_\star) \le \half V^\omega_{z_k}(z_\star).
\end{align*}
\end{proof}

We apply this reduction in order to prove Corollary~\ref{cor:mainmmfsunreg}.

\restatecormainmmfsunreg*
\begin{proof}
The overhead $K$ is asymptotically the same here as the logarithmic term in the parameter $T$ in Theorem~\ref{thm:mmfs}, by analogous smoothness and strong convexity arguments. Moreover, we use Theorem~\ref{thm:mmfs} with $\mu\x$, $\mu\y$ rescaled by constants to solve each subproblem required by Algorithm~\ref{alg:redx-outer-mm}; in particular, the subproblem is equivalent to approximately finding a saddle point to $\Ffs(z) + \frac {\mu\x} 8 \norm{z\x}^2 - \frac{\mu\y} 8 \norm{z\y}^2$, up to a linear shift which does not affect any smoothness bounds. We note that we will initialize the subproblem solver in iteration $k$ with $z_k$. We hence can set $T = O(\gamma)$, yielding the desired iteration bound.
\end{proof}

\section{Helper facts}
\label{apdx:sm}

Here we state two helper facts that are used throughout the analysis, for completeness of the paper. The first gives a few properties on monotone operators. We first recall by definition, an operator $\gop:\zset\rightarrow \zset^*$ is monotone if
\[
\inprod{\gop(z)-\gop(z')}{z-z'}\ge 0,~~\text{for all}~z',z'\in\zset.
\]
An operator $\gop$ is $m$-strongly monotone with respect to convex $r:\zset\rightarrow\R$ if for all $z,z'\in\zset$, 
\[
\inprod{\gop(z)-\gop(z')}{z-z'}\ge m\inprod{\nabla r(z)-\nabla r(z')}{z-z'},~~\text{for all}~z',z'\in\zset.
\]

We state the following standard facts about monotone operators and their specialization to convex-concave functions, and include references or proofs for completeness.

\begin{fact}\label{fact:sm}
The following facts about monotone operators hold true:
	\begin{enumerate}
		\item Given a convex function $f(x):\xset\rightarrow \R$, its induced operator $\gop = \nabla f:\xset\rightarrow\xset^*$ is monotone.\label{item:sm-convex}
		\item Given a convex-concave function $h(x,y):\xset\times\yset\rightarrow\R$, its induced operator $\gop(x,y)$ $=  (\nabla_x h(x, y), -\nabla_y h(x, y)):\xset\times\yset\rightarrow \xset^*\times\yset^*$ is monotone.\label{item:sm-convex-concave}
		\item Given a convex function $f$, its induced operator $\gop = \nabla f$ is  $1$-strongly monotone with respect to itself.\label{item:sm-self}
		\item Monotonicity is preserved under addition: For any $m,m'\ge 0$, if $\gop$ is $m$-strongly monotone and $\Psi$ is $m'$-strongly monotone with respect to convex $r$, then $\gop+\Psi$ is $(m+m')$-strongly monotone with respect to $r$. \label{item:sm-additive}
	\end{enumerate}
\end{fact}
\begin{proof}
The first two items are basic fact of convexity and minimax optimization~\cite{Rockafellar70b}. For the third item, we note that for any $x,x'\in\xset$
\[
\inprod{\gop(x)-\gop(x')}{x-x'}= \inprod{\nabla f(x)-\nabla f(x')}{x-x'},
\]
which satisfies $1$-strong monotonicity with respect to $f$ by definition.

For the fourth item, we note that for any $m,m'\ge0$ and assumed $\gop, \Psi$,
\begin{gather*}
\inprod{\gop(z)-\gop(z')}{z-z'}\ge m\inprod{\nabla r(z)-\nabla r(z')}{z-z'},\\
\inprod{\Psi(z)-\Psi(z')}{z-z'}\ge m'\inprod{\nabla r(z)-\nabla r(z')}{z-z'},\\
\implies\inprod{\gop(z)+\Psi(z)-\Par{\gop(z')+\Psi(z')}}{z-z'}\ge (m+m')\inprod{\nabla r(z)-\nabla r(z')}{z-z'}.
\end{gather*}

\end{proof}

These facts about monotone operators find usage in proving (relative) strong monotonicity of our operators; see~\Cref{lem:sm},~\ref{lem:smfs} and~\ref{lem:mmfs-sm} in the main paper.

The second fact bounds the smoothness of best-response function of some given convex-concave function $h;\xset\times\yset\rightarrow \R$. We refer readers to Fact 1 of \cite{WangL20} for a complete proof.

\begin{fact}[Fact 1, \cite{WangL20}]\label{fact:br}
	Suppose $h$ satisfies the blockwise-smoothness properties: for all $u, v \in \xset \times \yset$,
		\begin{equation}\label{eq:Hlipbound-fact}
			\begin{aligned}
				\norm{\nabla_x h(u) - \nabla_x h(v)} &\le \Lam\xx \norm{u\x - v\x} + \Lam\xy \norm{u\y - v\y},\\
				\norm{\nabla_y h(u) - \nabla_y h(v)} &\le \Lam\xy \norm{u\x - v\x} + \Lam\yy \norm{u\y - v\y},
		\end{aligned}\end{equation}
		and suppose $h$ is $\mu\x$-strongly convex in $x$ and $\mu\y$-strongly concave in $y$. The best response function $h\y(x) \defeq \max_{y\in\yset}h(x,y)$ is $\mu\x$-strongly convex and $\Lam\xx+\frac{(\Lam\xy)^2}{\mu\y}$-smooth, and $h\x(y) \defeq \min_{x\in\yset}h(x,y)$ is $\mu\y$-strongly concave and $\Lam\yy+\frac{(\Lam\xy)^2}{\mu\x}$-smooth.
\end{fact}

We use this fact when converting radius bounds to duality gap bounds in~\Cref{lem:initbound} and~\ref{lem:termbound}.
\section{Proofs for~\Cref{sec:mmfs}}\label{apdx:mmfs}

\subsection{Proofs for~\Cref{ssec:mmfs-alg}}\label{apdx:mfs-alg}

\propnewrmp*

\begin{proof}
First, consider a single iteration $0 \le s < S$, and fix the point $w_s$ in Algorithm~\ref{alg:rmp}. By the optimality conditions on $w_{s+1/2}$ and $w_{s + 1}$, we have
\begin{align*}
\frac 1 \lam \inprod{\goptilde(w_s)}{w_{s+1/2} - w_{s+1}} &\le V^r_{w_s}(w_{s+1}) - V^r_{w_{s+1/2}}(w_{s+1}) - V^r_{w_s}(w_{s+1/2}), \\
\frac 1 \lam \inprod{\goptilde(w_{s+1/2})}{w_{s+1} - w_\star} &\le V^r_{w_s}(w_\star) - V^r_{w_{s+1}}(w_\star) - V^r_{w_s}(w_{s+1}).
\end{align*} 
Summing the above, rearranging, and taking expectations yields
\begin{flalign*}
& \E\Brack{\frac 1 \lam \inprod{\gop(\bw_s)}{\bw_s - w_\star}} = \E\Brack{\frac 1 \lam \inprod{\goptilde(w_{s+1/2})}{w_{s+1/2} - w_\star}}\\
& \hspace{1em} \le \E\Brack{V^r_{w_s}(w_\star) - V^r_{w_{s+1}}(w_\star)} + \E\Brack{\frac 1 \lam \inprod{\goptilde(w_{s+1/2}) - \goptilde(w_s)}{w_{s+1/2} - w_{s+1}} - V^r_{w_s}(w_{s+1/2}) + V^r_{w_{s+1/2}}(w_{s+1})} \\
& \hspace{1em} \le \E\Brack{V^r_{w_s}(w_\star) - V^r_{w_{s+1}}(w_\star)} + \frac{\rho\lam_1}{\lam} \E\Brack{V^r_{w_0}(w_\star) + V^r_{\bw_s}(w_\star)}.
\end{flalign*}
In the last line we used the assumption \eqref{eq:newassume}. Since $w_\star$ solves the VI in $\gop$, adding $\E\frac 1 \lam \inprod{\gop(w_\star)}{w_\star - \bw_s}$ to the left-hand side above and applying strong monotonicity of $g$ in $r$ yields
\[\E\Brack{\frac 1 \lam V^r_{\bw_s}(w_\star)} \le \E\Brack{V^r_{w_s}(w_\star) - V^r_{w_{s+1}}(w_\star)} + \frac{\rho\lam_1}{\lam} \E\Brack{V^r_{w_0}(w_\star) + V^r_{\bw_s}(w_\star)}.\]
Telescoping the above for $0 \le s < S$ and using nonnegativity of Bregman divergences yields
\begin{align*}
\Par{\gamma - \rho\lam_1} \E\Brack{\frac 1 T \sum_{0 \le t < T} V^r_{\bw_s}(w_\star)} &\le \Par{\frac{\lam}{T} + \rho\lam_1} V^r_{w_0}(w_\star).
\end{align*}
Substituting our choices of $\bw_s$, $\rho$, $\lam$, and $T$ yields the claim.
\end{proof}

\lemexregretmmfs*

\begin{proof}
	We demonstrate this equality for the $\xset$ and $(\xset^*)^n$ blocks; the others (the $\yset$ and $(\yset^*)^n$ blocks) follow symmetrically. We will use the definitions of $\gop^h$ and $\gcross$ from \eqref{eq:gdefmmfs}.
	
	\paragraph{$\xset$ block.} Fix $\ell \in [n]$. We first observe that
	\begin{align*}
	\E_{\ell' \sim r}\Brack{\Brack{\gopt^\xyfun_{jk\ell'}(w\mdpt(jk\ell))}\x} &= \Brack{\gop^\xyfun(\bw(\ell))}\x, \\
	\E_{\ell' \sim r}\Brack{\Brack{\gsept_{jk\ell'}(w\mdpt(jk\ell))}\x} &= (1 + \gamma) \Brack{\nabla r(\bw(\ell))}\x - \gamma \Brack{\nabla r(\bz)}\x.
	\end{align*}
	Moreover, by expanding the expectation over $j \sim p$,
	\begin{align*}
	\E_{j \sim p}\Brack{\inprod{\Brack{\gcrosst_{jk\ell'}(w\mdpt(jk\ell))}\x}{w\mdpt\xsup(\ell) - u\x}} &= \inprod{\frac 1 n \sum_{j \in [n]} (w\xdssup{j} + \Delta\x(j))}{w\mdpt\xsup(\ell) - u\x} \\
	&= \inprod{\Brack{\gcross(\bw(\ell)) }\x}{w\mdpt\xsup(\ell) - u\x}.
	\end{align*}
	Summing, we conclude that for fixed $\ell$ and taking expectations over $j, k, \ell'$,
	\[\E\Brack{\inprod{\Brack{\gop_{jk\ell'}(w\mdpt(jk\ell))}\x}{w\mdpt\xsup(\ell) - u\x}} = \inprod{\Brack{\gop(\bw(\ell)) }\x}{w\mdpt\xsup(\ell) - u\x}.\]
	The conclusion for the $\xset$ block follows by taking expectations over $\ell$.
	
	\paragraph{$\xset^*$ blocks.} Note that the $[\gopt^h_{jk\ell'}]\xdsup$ blocks are always zero. Next, for the $[\gsept_{jk\ell'}]\xdsup$ component, by expanding the expectation over $j \sim p$ and taking advantage of sparsity, for any $\ell \in [n]$,
	\begin{align*}
	\E_{j \sim p}\Brack{\inprod{\Brack{\gsept_{jk\ell'}(w\mdpt(jk\ell))}\xdsup}{w\mdpt\xdsup (jk\ell)- u\xdsup}} &= (1 + \gamma)\sum_{j \in [n]} \inprod{\frac 1 n \nabla f^*_j\Par{w\mdpt\xdssup{j}}}{w\mdpt\xdssup{j} - u\xdssup{j}} \\
	&- \gamma \sum_{j \in [n]} \inprod{\frac 1 n \nabla f^*_j\Par{\bz\xdssup{j}}}{w\mdpt\xdssup{j} - u^{\mathsf{a_j}}} \\
	&= \inprod{(1 + \gamma) \Brack{\nabla r(\bw(\ell))}\xdsup - \gamma \Brack{\nabla r(\bz)}\xdsup}{\bw\xdsup(\ell) - u\xdsup}.
	\end{align*}
	Here, we recall $\xdssup{j}$ denotes the block corresponding to the $j^{\text{th}}$ copy of $\xset^*$. Finally, for the $[\gcrosst_{jk\ell'}]\xdsup$ component, fix $\ell \in [n]$. Expanding the expectation over $j \sim p$ and taking advantage of sparsity,
	\begin{align*}
	\E_{j \sim p}\Brack{\inprod{\Brack{\gcrosst_{jk\ell'}(w\mdpt(jk\ell))}\xdsup}{\Brack{w\mdpt(jk\ell)}\xdsup - u\xdsup}} = \inprod{\Brack{\gcross(\bw(\ell))}\xdsup}{\bw\xdsup(\ell) - u\xdsup}.
	\end{align*}
	Summing, we conclude that for fixed $\ell$ and taking expectations over $j, k, \ell'$,
	\[\E\inprod{\Brack{g_{jk\ell'}(w\mdpt(jk\ell))}\xdsup}{\Brack{w\mdpt(jk\ell)}\xdsup - u\xdsup} = \inprod{\Brack{\gtot(\bw(\ell)) }\xdsup}{\bw\xdsup(\ell) - u\xdsup}.\]
	The conclusion for the $\xset^*$ blocks follows by taking expectations over $\ell$.
\end{proof}

\lemimpmmfs*
\begin{proof}
Let $\{w_s, w_{s + 1/2}\}_{0 \le s \le \sigma}$ be the iterates of Algorithm~\ref{alg:rmp}. We will inductively show that some run of Lines~\ref{line:mmfs-inner-start} to~\ref{line:mmfs-inner-end} in Algorithm~\ref{alg:mmfs} preserves the invariants
\begin{align*}w_s &= \Par{w_s\x, w_s\y, \bin{\nabla f_i(w_s\xpi)}, \bin{\nabla f_i(w_s\ypi)}},\\
w_{s + 1/2} &= \Par{w_{s + 1/2}\x, w_{s + 1/2}\y, \bin{\nabla f_i(w_{s + 1/2}\xpi)}, \bin{\nabla f_i(w_{s + 1/2}\ypi)}}\end{align*}
for all $0 \le s \le \sigma$. Once we prove this claim, it is clear that Lines~\ref{line:mmfs-inner-start} to~\ref{line:mmfs-inner-end} in Algorithm~\ref{alg:mmfs} implements Algorithm~\ref{alg:rmp} and returns $\bw_\sigma$, upon recalling the definitions \eqref{eq:gjkldef}, \eqref{eq:gpjkldef}, \eqref{eq:rmmfsdef}, and \eqref{eq:bwdef}.

The base case of our induction follows from the way $w_0$ is initialized in Line~\ref{line:mmfs-inner-end}. Next, suppose for some $0 \le s \le \sigma$, our inductive claim holds. By the update in Line~\ref{line:mmfs-grad-dual-x} of Algorithm~\ref{alg:mmfs}, if $j \in [n]$ was sampled in iteration $s$, using the first item in Fact~\ref{fact:dualsc},
	\begin{align*}
	w_{s+1/2}\xdssup{j} & \gets \argmin_{w\xdssup{j} \in \xset^*}\Brace{\inprod{\frac{1}{n\lambda p_j}\Par{(1+\gamma)  w_s\xpssup{j}-\gamma   \bz\xpssup{j}-w_s\xsup}}{w\xdssup{j}}+V^{f_j^*}_{w_s\xdssup{j}}\Par{w\xdssup{j}}}\\
	& = \nabla f_j \Par{w\xpssup{j}_{s}-\frac{1}{n\lambda p_j}\Par{\Par{1+\gamma}w\xpssup{j}_{s}-\gamma \bz\xpssup{j}-w\xsup_{s}}}.
	\end{align*}
Similarly, by the update in Line~\ref{line:mmfs-grad-dual-y}, if $k \in [n]$ was sampled in iteration $s$,
 \begin{align*}
	w_{s+1/2}\ydssup{k} & \gets \argmin_{w\ydssup{k}}\inprod{\frac{1}{n\lambda q_k}\Par{(1+\gamma)w_s\ypssup{k}-\gamma  \bz\ypssup{k}-w_s\ysup}}{w\ydssup{k}}+V^{g_k^*}_{w_s\ydssup{k}}\Par{w\ydssup{k}}.\\
	& = \nabla g_k \Par{w\ypssup{k}_{s}-\frac{1}{n\lambda q_k}\Par{\Par{1+\gamma}w\ypssup{k}_{s}-\gamma \bz\ypssup{k}-w\ysup_{s}}}.
	\end{align*}
Hence, the updates to $w_{s + 1/2}\xdj$ and $w_{s + 1/2}\ydk$ preserve our invariant, and all other $w_{s + 1/2}\xdi$, $i \neq j$ and $w_{s + 1/2}\ydi$, $i \neq k$ do not change by sparsity of $\gop_{jk\ell}$. Analogously the updates to each $w_{s + 1}\xdi$ and $w_{s + 1}\ydi$ preserve our invariant. Finally, in every iteration $s > 0$, the updates to $w\xy_{s + 1/2}$ and $w\xy_{s + 1}$ only require evaluating $O(1)$ new gradients each, by $1$-sparsity of the dual block updates.
\end{proof}

\subsection{Proofs for~\Cref{mmfs:convergence}}\label{apdx:mmfsinner}

\lamzerobound*

\begin{proof}
This is immediate upon combining the following Lemmas~\ref{lem:lamsep} and~\ref{lem:lamcross}.
\end{proof}

\begin{lemma}\label{lem:lamsep}
Following notation of Lemma~\ref{lem:lam0bound}, for $\lam^{\textup{sep}} \defeq 2n(1 + \gamma)$, we have
\[\E\Brack{\inprod{\gsept_{jk\ell'}(w\mdpt(jk\ell)) - \gsept_{jk\ell}(w)}{w\mdpt(jk\ell) - w_{+}(jk\ell\ell')}} \le \lam^{\textup{sep}}\E\Brack{V^r_{w}\Par{w\mdpt(jk\ell)} + V^r_{w\mdpt(jk\ell)}\Par{w_{+}(jk\ell\ell')}}.\]
\end{lemma}
\begin{proof}
The proof is similar to (part of) the proof of Lemma~\ref{lem:exrl}. We claim that for any $j, k, \ell, \ell'$,
\[\inprod{\gsept_{jk\ell'}(w\mdpt(jk\ell)) - \gsept_{jk\ell}(w)}{w\mdpt(jk\ell) - w_{+}(jk\ell\ell')} \le \lam^{\textup{sep}}\Par{V^r_{w}\Par{w\mdpt(jk\ell)} + V^r_{w\mdpt(jk\ell)}\Par{w_{+}(jk\ell\ell')}}.\]
Fix $j, k, \ell, \ell'$. Since all $p_j$ and $q_k$ are lower bounded by $\frac 1 {2n}$ by assumption, applying Lemma~\ref{lem:rnablar} to the relevant blocks of $r$ and nonnegativity of Bregman divergences proves the above display.
\end{proof}

\begin{lemma}\label{lem:lamcross}
Following notation of Lemma~\ref{lem:lam0bound}, for
\[\lam^{\textup{cross}} \defeq \frac{2\ssin \sqrt{L\x_i}}{\sqrt{n\mu\x}} + \frac{2\ssin \sqrt{L\y_i}}{\sqrt{n\mu\y}},\]
we have
\[\E\Brack{\inprod{\gcrosst_{jk\ell'}(w\mdpt(jk\ell)) - \gcrosst_{jk\ell}(w)}{w\mdpt(jk\ell) - w_{+}(jk\ell\ell')}} \le \lam^{\textup{cross}}\E\Brack{V^r_{w}\Par{w\mdpt(jk\ell)} + V^r_{w\mdpt(jk\ell)}\Par{w_{+}(jk\ell\ell')}}.\]
\end{lemma}
\begin{proof}
The proof is similar to (part of) the proof of Lemma~\ref{lem:exrl}. We claim that for any $j, k, \ell, \ell'$,
\[\inprod{\gcrosst_{jk\ell'}(w\mdpt(jk\ell)) - \gcrosst_{jk\ell}(w)}{w\mdpt(jk\ell) - w_{+}(jk\ell\ell')} \le \lam^{\textup{cross}}\Par{V^r_{w}\Par{w\mdpt(jk\ell)} + V^r_{w\mdpt(jk\ell)}\Par{w_{+}(jk\ell\ell')}}.\]
Fix $j, k, \ell, \ell'$. By applying \Cref{item:minimax_fsmooth_equiv} in Lemma~\ref{lem:smoothness_implications} with $f = f_j$,  $\alpha = (L\x_j\mu\x)^{-\half}$,
\begin{gather*}
\E_j\Brack{\frac 1 {np_j} \inprod{w\mdpt\xdssup{j} - w\xdssup{j}}{w\mdpt\xsup(\ell) - w_{+}\xsup(jk\ell\ell')} + \frac 1 {np_j} \inprod{w\xsup - w\mdpt\xsup(\ell)}{w\mdpt\xdssup{j} - w_{+}\xdssup{j}(jk\ell\ell')}} \\
\le \frac{2\ssin \sqrt{L\x_i}}{\sqrt{n\mu\x}}\Par{V^r_{w}\Par{w\mdpt(jk\ell)} + V^r_{w\mdpt(jk\ell)}\Par{w_{+}(jk\ell\ell')}}.
\end{gather*}
Similarly, by applying \Cref{item:minimax_fsmooth_equiv} in Lemma~\ref{lem:smoothness_implications} with $f = g_k$,  $\alpha = (L\y_k\mu\y)^{-\half}$,
\begin{gather*}
\E_j\Brack{\frac 1 {nq_k} \inprod{w\mdpt\ydssup{k} - w\ydssup{k}}{w\mdpt\ysup(\ell) - w_{+}\ysup(jk\ell\ell')} + \frac 1 {nq_k} \inprod{w\ysup - w\mdpt\ysup(\ell)}{w\mdpt\ydssup{k} - w_{+}\ydssup{k}(jk\ell\ell')}} \\
\le \frac{2\ssin \sqrt{L\y_i}}{\sqrt{n\mu\y}}\Par{V^r_{w}\Par{w\mdpt(jk\ell)} + V^r_{w\mdpt(jk\ell)}\Par{w_{+}(jk\ell\ell')}}.
\end{gather*}
Summing the above displays yields the desired claim.
\end{proof}

\lamonebound*

\begin{proof}
The proof is similar to (part of) the proof of Lemma~\ref{lem:rl}. Fix $j, k, \ell, \ell'$. By definition,
\begin{gather*}\Brack{\gopt^\xyfun_{jk\ell'}(w\mdpt(jk\ell)) - \gopt^\xyfun_{jk\ell}(w)}\xy \\
= \frac 1 {nr_{\ell'}}\Par{\nabla_x \xyfun_{\ell'}(w\mdpt\xsup(\ell), w\ysup\mdpt(\ell)) - \nabla_x \xyfun_{\ell'}(w_0\xsup, w_0\ysup), \nabla_y \xyfun_{\ell'}(w_0\xsup, w_0\ysup) - \nabla_y \xyfun_{\ell'}(w\mdpt\xsup(\ell), w\ysup\mdpt(\ell))} \\
- \frac 1 {nr_\ell}\Par{\nabla_x \xyfun_\ell(w\xsup, w\ysup) - \nabla_x \xyfun_\ell(w_0\xsup, w_0\ysup), \nabla_y \xyfun_\ell(w_0\xsup, w_0\ysup) - \nabla_y \xyfun_\ell(w\xsup, w\ysup) }.
\end{gather*}
We decompose the $x$ blocks of the left-hand side of \eqref{eq:ghbound} as
\begin{align*}\inprod{\Brack{\gopt^\xyfun_{jk\ell'}(w\mdpt(jk\ell)) - \gopt^\xyfun_{jk\ell}(w)}\x}{w\mdpt\xsup(\ell) - w_{+}\xsup(jk\ell\ell')} = \circled{1} +  \circled{2} + \circled{3},\\
\circled{1} \defeq \frac 1 {nr_{\ell'}}\inprod{\nabla_x \xyfun_{\ell'}(w\mdpt\xsup(\ell), w\ysup\mdpt(\ell)) - \nabla_x \xyfun_{\ell'}(w_0\xsup, w_0\ysup)}{w\mdpt\xsup(\ell) - w_{+}\xsup(jk\ell\ell')}, \\
\circled{2} \defeq \frac 1 {nr_{\ell}}\inprod{\nabla_x \xyfun_{\ell}(w_0\xsup, w_0\ysup) - \nabla_x \xyfun_{\ell}(w\mdpt\xsup(\ell), w\ysup\mdpt(\ell))}{w\mdpt\xsup(\ell) - w_{+}\xsup(jk\ell\ell')}, \\
\circled{3} \defeq \frac 1 {nr_{\ell}}\inprod{\nabla_x \xyfun_{\ell}(w\mdpt\xsup(\ell), w\ysup\mdpt(\ell)) - \nabla_x \xyfun_{\ell}(w\xsup, w\ysup) }{w\mdpt\xsup(\ell) - w_{+}\xsup(jk\ell\ell')}.
\end{align*}
By the Lipschitzness bounds in \eqref{eq:Hlipbound} and Young's inequality,
\begin{align*}
\circled{1} &\le \frac 1 {nr_{\ell'}} \norm{\nabla_x \xyfun_{\ell'}(w\mdpt\xsup(\ell), w\ysup\mdpt(\ell)) - \nabla_x \xyfun_{\ell'}(w_0\xsup, w_0\ysup)}\norm{w\mdpt\xsup(\ell) - w_{+}\xsup(jk\ell\ell')} \\
&\le \frac 1 {nr_{\ell'}} \Par{\Lam_{\ell'}\xx \norm{w\mdpt\xsup(\ell) - w_0\xsup} \norm{w\mdpt\xsup(\ell) - w_{+}\xsup(jk\ell\ell')} + \Lam\xy_{\ell'} \norm{w\ysup\mdpt(\ell) - w_0\ysup} \norm{w\mdpt\xsup(\ell) - w_{+}\xsup(jk\ell\ell')}} \\
&\le \frac{2\rho(\Lam\xx_{\ell'})^2}{\mu\x n^2 r_{\ell'}^2} \norm{w\mdpt\xsup(\ell) - w_0\xsup}^2 + \frac{2\rho(\Lam\xy_{\ell'})^2}{\mu\x n^2 r_{\ell'}^2}\norm{w\ysup\mdpt(\ell) - w_0\ysup}^2 + \frac{\mu\x}{4\rho} \norm{w\mdpt\xsup(\ell) - w_{+}\xsup(jk\ell\ell')}^2.
\end{align*}
Symmetrically, we bound
\begin{align*}
\circled{2} &\le \frac{2\rho(\Lam\xx_{\ell})^2}{\mu\x n^2 r_{\ell}^2} \norm{w\mdpt\xsup(\ell) - w_0\xsup}^2 + \frac{2\rho(\Lam\xy_{\ell})^2}{\mu\x n^2 r_{\ell}^2}\norm{w\ysup\mdpt(\ell) - w_0\ysup}^2 + \frac{\mu\x}{4\rho} \norm{w\mdpt\xsup(\ell) - w_{+}\xsup(jk\ell\ell')}^2.
\end{align*}
Finally, we have
\begin{align*}
\circled{3} &\le \frac 1 {nr_\ell} \norm{\nabla_x \xyfun_{\ell}(w\mdpt\xsup(\ell), w\ysup\mdpt(\ell)) - \nabla_x \xyfun_{\ell}(w\xsup, w\ysup) } \norm{w\mdpt\xsup(\ell) - w_{+}\xsup(jk\ell\ell')} \\
&\le \frac 1 {nr_\ell} \Par{\Lam\xx_\ell \norm{w\mdpt\xsup(\ell) - w\xsup}\norm{w\mdpt\xsup(\ell) - w_{+}\xsup(jk\ell\ell')} + \Lam\xy_\ell\norm{w\ysup\mdpt(\ell) - w\ysup}\norm{w\mdpt\xsup(\ell) - w_{+}\xsup(jk\ell\ell')}} \\
&\le \frac 1 {nr_\ell} \Par{\frac{\Lam\xx_\ell}{\mu\x} \Par{\frac{\mu\x}{2}\norm{w\mdpt\xsup(\ell) - w\xsup}^2 + \frac{\mu\x}{2}\norm{w\mdpt\xsup(\ell) - w_{+}\xsup(jk\ell\ell')}^2}} \\
&+ \frac 1 {nr_\ell}\Par{\frac{\Lam_{\ell}\xy}{\sqrt{\mu\x\mu\y}} \Par{\frac{\mu\y}{2}\norm{w\ysup\mdpt(\ell) - w\ysup}^2 + \frac{\mu\x}{2}\norm{w\mdpt\xsup(\ell) - w_{+}\xsup(jk\ell\ell')}^2}}.
\end{align*}
We may similarly decompose the $y$ blocks of the left-hand side of \eqref{eq:ghbound} as $\circled{4} + \circled{5} + \circled{6}$, where symmetrically, we have
\begin{align*}
\circled{4} &\le \frac{2\rho(\Lam\yy_{\ell'})^2}{\mu\y n^2 r_{\ell'}^2} \norm{w\ysup\mdpt(\ell) - w_0\ysup}^2 + \frac{2\rho(\Lam\xy_{\ell'})^2}{\mu\y n^2 r_{\ell'}^2}\norm{w\mdpt\xsup(\ell) - w_0\xsup}^2 + \frac{\mu\y}{4\rho} \norm{w\ysup\mdpt(\ell) - w\ysup_{+}(jk\ell\ell')}^2, \\
\circled{5} &\le \frac{2\rho(\Lam\yy_{\ell})^2}{\mu\y n^2 r_{\ell}^2} \norm{w\ysup\mdpt(\ell) - w_0\ysup}^2 + \frac{2\rho(\Lam\xy_{\ell})^2}{\mu\y n^2 r_{\ell}^2}\norm{w\mdpt\xsup(\ell) - w_0\xsup}^2 + \frac{\mu\y}{4\rho} \norm{w\ysup\mdpt(\ell) - w\ysup_{+}(jk\ell\ell')}^2, \\
\circled{6} &\le \frac 1 {nr_\ell} \Par{\frac{\Lam\yy_\ell}{\mu\y} \Par{\frac{\mu\y}{2}\norm{w\ysup\mdpt(\ell) - w\ysup}^2 + \frac{\mu\y}{2}\norm{w\ysup\mdpt(\ell) - w\ysup_{+}(jk\ell\ell')}^2}} \\
&+ \frac 1 {nr_\ell}\Par{\frac{\Lam_{\ell}\xy}{\sqrt{\mu\x\mu\y}} \Par{\frac{\mu\x}{2}\norm{w\mdpt\xsup(\ell) - w\xsup}^2 + \frac{\mu\y}{2}\norm{w\ysup\mdpt(\ell) - w\ysup_{+}(jk\ell\ell')}^2}}.
\end{align*}
We first observe that by definition of $r$ and nonnegativity of Bregman divergences,
\begin{align*}
\circled{3} + \circled{6} &\le \frac 1 {nr_\ell} \Par{\frac{\Lam\xx_\ell}{\mu\x} + \frac{\Lam_{\ell}\xy}{\sqrt{\mu\x\mu\y}} + \frac{\Lam\yy_\ell}{\mu\y}} \Par{V^r_{w}(w\mdpt(jk\ell)) + V^r_{w\mdpt(jk\ell)}(w_{+}(jk\ell\ell'))} \\
&\le 2\lam^\xyfun\Par{V^r_{w}(w\mdpt(jk\ell)) + V^r_{w\mdpt(jk\ell)}(w_{+}(jk\ell\ell'))}.
\end{align*}
Moreover, since by the triangle inequality and $(a + b)^2 \le 2a^2 + 2b^2$,
\begin{align*}
\norm{w\mdpt\xsup(\ell) - w_0\xsup}^2 &\le 2\norm{w\mdpt\xsup(\ell) - w_\star\x}^2 + 2\norm{w_0\xsup - w_\star\x}^2, \\
\norm{w\mdpt\ysup(\ell) - w_0\ysup}^2 &\le 2\norm{w\ysup\mdpt(\ell) - w_\star\y}^2 + 2\norm{w_0\ysup - w_\star\y}^2,
\end{align*}
we have by definition of $r$ and $\lam_1$,
\begin{align*}
\circled{1} + \circled{2} + \circled{4} + \circled{5} &\le \frac 1 \rho \Par{V^r_{w}(w\mdpt(jk\ell)) + V^r_{w\mdpt(jk\ell)}(w_{+}(jk\ell\ell'))} \\
&+ \rho\lam_1\Par{V^r_{w_0}(w_\star) + V^r_{\bw(\ell)}(w_\star)}.
\end{align*}
Summing the above displays and taking expectations yields the claim.
\end{proof}

\subsection{Proofs for~\Cref{ssec:mmfsouter}}\label{apdx:mmfsouter}

\propmultiphase*

\begin{proof}
Fix an iteration $t \in [T]$ of Algorithm~\ref{alg:mmfs-outer}, and let $z^\star_{t + 1}$ be the exact solution to the VI in $\goptot + \gamma \nabla r - \nabla r(z_t)$. By the guarantee of \Cref{prop:onephase}, after the stated number of $NS$ iterations in Algorithm~\ref{alg:mmfs} (for an appropriately large constant), we obtain a point $z_{t + 1}$ such that
\begin{equation}\label{eq:mmfs-kmax}
\begin{aligned}
	\E\Brack{V^r_{z_{t+1}}(z^\star_{t+1})} \le \frac{1}{1 + 3\gamma \widetilde{\kappa}} V^r_{z_t}(\zhat_{t+1}), \text{ where } \widetilde{\kappa} \defeq 10\sum_{i\in[n]}\Par{\frac{L_i\x + \Lam_i\xx}{\mu\x} + \frac{L_i\y + \Lam_i\yy}{\mu\y}+ \frac{\Lam_i\xy}{\sqrt{\mu\x\mu\y}}}^2.
\end{aligned}
\end{equation}
The optimality condition on $z^\star_{t + 1}$ yields
\[
\inprod{\goptot\Par{z^\star_{t+1}}}{z^\star_{t+1}-z_\star}\le\gamma V^r_{z_t}\Par{z_\star}-\gamma V^r_{z^\star_{t+1}}\Par{z_\star}-\gamma V_{z_t}\Par{z^\star_{t+1}}.
\]	
Rearranging terms then gives:
\begin{equation}\label{eq:mmfs-coro-total}
\begin{aligned}
 \inprod{\goptot\Par{z_{t+1}}}{z_{t+1}-z_\star}  &\le\gamma V^r_{z_t}\Par{z_\star}-\gamma V^r_{z_{t+1}}\Par{z_\star}-\gamma V^r_{z_t}\Par{z^\star_{t+1}}\\
 &+\gamma\Par{V^r_{z_{t+1}}\Par{z_\star}-V^r_{z^\star_{t+1}}\Par{z_\star}}\\
	& +\inprod{\goptot\Par{z_{t+1}}-\goptot\Par{z^\star_{t+1}}}{z^\star_{t+1}-z_\star}\\
	&+\inprod{\goptot\Par{z_{t+1}}}{z_{t+1}-z^\star_{t+1}}\\
	&= \gamma V^r_{z_t}\Par{z_\star}-\gamma V^r_{z_{t+1}}\Par{z_\star} -\gamma V^r_{z_t}\Par{z^\star_{t+1}} \\
	&+\gamma V^r_{z_{t+1}}\Par{z^\star_{t+1}} + \gamma \inprod{\nabla r\Par{z_{t+1}}-\nabla r\Par{z^\star_{t+1}}}{z^\star_{t+1}-z_\star}\\
	& +\inprod{\goptot\Par{z_{t+1}}-\goptot\Par{z^\star_{t+1}}}{z^\star_{t+1}-z_\star}\\
	&+\inprod{\goptot\Par{z_{t+1}}}{z_{t+1}-z^\star_{t+1}}
	\\
	&\le  \gamma V^r_{z_t}\Par{z_\star}-\gamma V^r_{z_{t+1}}\Par{z_\star} -\gamma V^r_{z_t}\Par{z^\star_{t+1}} +\gamma V^r_{z_{t+1}}\Par{z^\star_{t+1}} \\
	&+\gamma \inprod{\nabla r\Par{z_{t+1}}-\nabla r\Par{z^\star_{t+1}}}{z_{t+1}-z_\star}\\
	& +\inprod{\goptot\Par{z_{t+1}}-\goptot\Par{z^\star_{t+1}}}{z_{t+1}-z_\star}\\
	&+\inprod{\goptot\Par{z_{t+1}}-\goptot\Par{z_\star}}{z_{t+1}-z^\star_{t+1}}.
\end{aligned}
\end{equation}
In the only equality, we used the identity \eqref{eq:threept}. The last inequality used monotonicity of the operators $\gamma \nabla r$ and $\goptot$, as well as $\goptot(z_\star) = 0$ because it is an unconstrained minimax optimization problem. In the remainder of the proof, we will bound the last three lines of~\eqref{eq:mmfs-coro-total}. 

First, for any $\alpha > 0$, we bound:
\begin{flalign}
\inprod{\nabla r\Par{z_{t+1}}-\nabla r\Par{z^\star_{t+1}}}{z_{t+1}-z_\star} 
	 &= \mu\x\inprod{z\x_{t+1}-(z^\star_{t+1})\x}{z\x_{t+1}-z\x_\star}+\mu\y\inprod{z\y_{t+1}-(z^\star_{t+1})\y}{z\y_{t+1}-z\y_\star}\nonumber\\
	&+ \frac{1}{n}\sum_{i\in[n]}\inprod{\nabla f_i^*(z_{t+1}\xdisup)-\nabla f_i^*((z^\star_{t + 1})\xdisup)}{z_{t+1}\xdisup-z\xdisup_\star} \nonumber\\
	&+ \nsin \inprod{\nabla g_i^*(z_{t+1}\ydisup)-\nabla g_i^*((z^\star_{t + 1})\ydisup)}{z_{t+1}\ydisup-z\ydisup_\star}\nonumber \\
	&\le 2\alpha\mu\x\norm{z\x_{t+1}-(z^\star_{t + 1})\x}^2 +\frac{\mu\x}{8\alpha}\norm{z\x_{t+1}-z\x_\star}^2\nonumber\\
	&+2\alpha\mu\y\norm{z\y_{t+1}-(z^\star_{t + 1})\y}^2+\frac{\mu\y}{8\alpha}\norm{z\y_{t+1}-z\y_\star}^2\nonumber\\
	&+ \frac{1}{n}\sum_{i\in[n]}\Par{\frac{2\alpha L_i\x}{(\mu\x)^2}\norm{z_{t+1}\xdisup-(z^\star_{t + 1})\xdisup}^2+\frac{1}{8\alpha L_i\x}\norm{z_{t+1}\xdisup-z\xdisup_\star}^2}\nonumber\\
	&+\nsin\Par{ \frac{2\alpha L_i\y}{(\mu\y)^2}\norm{z_{t+1}\ydisup-(z^\star_{t + 1})\ydisup}^2+\frac{1}{8\alpha L_i\y}\norm{z_{t+1}\ydisup-z\ydisup_\star}^2}\nonumber\\
	&\le \frac{1}{4\alpha}V^r_{z_{t+1}}(z_\star)+ \widetilde{\kappa} \alpha V^r_{z_{t+1}}(z^\star_{t + 1}).\label{eq:mmfs-coro-1}
\end{flalign}
The equality used the definition of $r$ in~\eqref{eq:rmmfsdef}. The first inequality used Young's and Cauchy-Schwarz on the $\xset \times \yset$ blocks, as well as $\frac 1 {\mu\x_i}$-smoothness of the $f^*_i$ from Assumption~\ref{assume:minimax-fs} and Item 4 in Fact~\ref{fact:dualsc} (and similar bounds on each $g^*_i$). The last inequality used strong convexity of each piece of $r$.

Similarly, by definition of $\goptot$~\eqref{eq:gdefmmfs} which we denote for $\gop$ for brevity in the following:
\begin{flalign}
\inprod{\gop \Par{z_{t+1}}-\gop \Par{z^\star_{t+1}}}{z_{t+1}-z_\star}\nonumber
	 &\le \frac{1}{8}V^r_{z_{t+1}}(z_\star)+2\widetilde{\kappa} V^r_{z_{t+1}}(z^\star_{t+1})\nonumber\\
	&+ \frac{1}{n}\sum_{i\in[n]}\inprod{\nabla_x h_i(z_{t+1}\x, z_{t+1}\y)-\nabla_x h_i((z_{t+1}^\star)\x, (z_{t+1}^\star)\y)}{z_{t+1}\x-z_\star\x}\nonumber\\
	&+ \frac{1}{n}\sum_{i\in[n]}\inprod{\nabla_y h_i(z_{t+1}\x, z_{t+1}\y)-\nabla_y h_i((z_{t+1}^\star)\x, (z_{t+1}^\star)\y)}{z_{t+1}\y-z_\star\y}\nonumber\\
		&+ \frac{1}{n}\sum_{i\in[n]}\Par{\inprod{z_{t+1}\xdisup- (z_{t+1}^\star)\xdisup}{z_{t+1}\x-z_\star\x} + \inprod{z_{t+1}\ydisup- (z_{t+1}^\star)\ydisup}{z_{t+1}\y-z_\star\y}}\nonumber\\
		& -\frac{1}{n}\sum_{i\in[n]}\Par{\inprod{z_{t+1}\x- (z_{t+1}^\star)\x}{z_{t+1}\xdisup-z_\star\xdisup} + \inprod{z_{t+1}\y- (z_{t+1}^\star)\y}{z_{t+1}\ydisup-z_\star\ydisup}}\nonumber
		\end{flalign}
		where we used \eqref{eq:mmfs-coro-1} to bound the $\nabla r$ terms. 
		Consequently,
\begin{flalign}
\inprod{\gop \Par{z_{t+1}}-\gop \Par{z^\star_{t+1}}}{z_{t+1}-z_\star}\nonumber
	 &\le \frac{1}{8}V^r_{z_{t+1}}(z_\star)+2\widetilde{\kappa} V^r_{z_{t+1}}(z_{t+1}^star)\nonumber\\
	&+ \frac{1}{n}\sum_{i\in[n]}\Par{\frac{\mu\x}{16}V_{z_{t+1}\x}(z_\star\x)+\frac{\mu\y}{16}V_{z_{t+1}\y}(z_\star\y)}\nonumber\\
	&+ \frac{1}{n}\sum_{i\in[n]}\Par{16\Par{\frac{(\Lam\xx_{i})^2}{\mu\x}+\frac{(\Lam\xy_{i})^2}{\mu\y}}V_{z\x_{t+1}}((z^\star_{t+1})\x)}\nonumber\\
	&+\nsin \Par{16\Par{\frac{(\Lam\xy_{i})^2}{\mu\x}+\frac{(\Lam\yy_{i})^2}{\mu\y}}V_{z\y_{t+1}}((z^\star_{t+1})\y)}\nonumber\\
	&+ \frac{1}{n}\sum_{i\in[n]}\Par{\frac{\mu\x}{16}V_{z_{t+1}\x}(z_\star\x)+\frac{\mu\y}{16}V_{z_{t+1}\y}(z_\star\y)} \nonumber\\
	&+\nsin\Par{\frac{8}{\mu\x}\norm{z_{t+1}\xdisup- (z^\star_{t+1})\xdisup}^2+ \frac{8}{\mu\y}\norm{z_{t+1}\ydisup- (z^\star_{t+1})\ydisup}^2}\nonumber\\
	&+ \frac{1}{n}\sum_{i\in[n]}\Par{\frac 1 8 V^{f^*_i}_{z_{t + 1}\xdi}\Par{z_\star\xdi}+\frac{1}{8}V^{g^*_i}_{z_{t + 1}\ydi}\Par{z_\star\ydi}}\nonumber\\
	&+\nsin\Par{8L_i\x V_{z_{t+1}\x}((z_{t + 1}^\star)\x)+ 8L_i\y V_{z_{t+1}\y}((z_{t + 1}^\star)\y)}\nonumber\\
	&\le \frac{1}{4}V^r_{z_{t+1}}(z_\star) + \widetilde{\kappa} V^r_{z_{t + 1}}(z_{t + 1}^\star).\label{eq:mmfs-coro-2}
\end{flalign}
In the first inequality, we used Cauchy-Schwarz, Young's, and our various smoothness assumptions (as well as strong convexity of each $f^*_i$ and $g^*_i$). The last inequality used strong convexity of each piece of $r$.

For the last term, by a similar argument as in the previous bounds, we have
\begin{flalign}
 \inprod{\gop(z_{t+1})-\gop(z_\star)}{z_{t+1}-z^\star_{t+1}} \le  \frac{1}{4}V^r_{z_{t+1}}(z_\star)+ \widetilde{\kappa} V^r_{z_{t + 1}}(z^\star_{t + 1}).\label{eq:mmfs-coro-3}
\end{flalign}

Plugging the inequalities~\eqref{eq:mmfs-coro-1} with $\alpha=\gamma$,~\eqref{eq:mmfs-coro-2} and~\eqref{eq:mmfs-coro-3} back into~\eqref{eq:mmfs-coro-total}, this implies
\begin{flalign} \inprod{\goptot\Par{z_{t+1}}}{z_{t+1}-z_\star}&\le \gamma V^r_{z_t}\Par{z_\star}-\gamma V^r_{z_{t+1}}\Par{z_\star} -\gamma V^r_{z_t}\Par{z^\star_{t+1}} +\gamma V^r_{z_{t+1}}\Par{z^\star_{t+1}} \\
&+\frac{3}{4}V^r_{z_{t+1}}(z_\star) + 3\widetilde{\kappa}\gamma^2 V^r_{z_{t + 1}}(z^\star_{t + 1}).\label{eq:mmfs-coro-last-1}
\end{flalign}

By strong monotonicity of $\goptot$ with respect to $r$, we also have 
\begin{flalign}
\inprod{\goptot\Par{z_{t+1}}}{z_{t+1}-z_\star}\ge 	\inprod{\goptot\Par{z_{t+1}}-\goptot\Par{z_\star}}{z_{t+1}-z_\star}\ge V^r_{z_{t+1}}\Par{z_\star}.\label{eq:mmfs-coro-last-2}
\end{flalign}

Combining~\eqref{eq:mmfs-coro-last-1} and~\eqref{eq:mmfs-coro-last-2} with the assumption~\eqref{eq:mmfs-kmax}, and taking expectations, we obtain
\begin{flalign*}
	\Par{\frac{1}{4}+\gamma}\E V^r_{z_{t+1}}(z_\star)\le \gamma V^r_{z_t}(z_\star) \implies \E V^r_{z_{t+1}}(z_\star)\le \frac{4\gamma}{1+4\gamma}V^r_{z_t}(z_\star).
\end{flalign*}

\end{proof}

 	\end{appendix}

\end{document}